\documentclass[11pt]{article} % use larger type; default would be 10pt
\usepackage{graphicx}
\usepackage{amsmath}
\usepackage{tikz}
\usepackage{eufrak}
\usepackage{mathrsfs}
\usetikzlibrary{patterns}
\usepackage{amsmath, amsthm, amssymb}
\usepackage{enumerate}
\usepackage{fancyhdr} % This should be set AFTER setting up the page geometry
\usepackage{sectsty}
\allsectionsfont{\sffamily\mdseries\upshape} % (See the fntguide.pdf for font help)
\graphicspath{ {c:/user/olofgi.NET/downloads/} }
\usepackage{verbatim} % adds environment for commenting out blocks of text & for better verbatim

\usepackage{yfonts}[1998/10/03]
\newtheorem{thm}{Theorem}[section]
\newtheorem{cor}[thm]{Corollary}
\newtheorem{lem}[thm]{Lemma}

\newtheorem{prop}[thm]{Proposition}
\newtheorem{rem}[thm]{Remark}

\theoremstyle{definition}
\newtheorem{exa}[thm]{Example}
\theoremstyle{definition}
\newtheorem{defn}[thm]{Definition}
\theoremstyle{remark}

\newtheorem{thm*}[thm]{Theorem}

\DeclareMathOperator{\id}{id}
\DeclareMathOperator{\spann}{span}

\title{The Universal $C^{*}$-Algebra of the Quantum Matrix Ball and its Irreducible $*$-representations.}
\author{Olof Giselsson}
\begin{document}
\maketitle

\abstract{We study $*$-representations of the quantized algebra $\mathrm{Pol}(\mathrm{Mat}_{n})_{q}$ of polynomials on the space of $n\times n $-matrices. We prove that any such representation can be lifted to a $*$-representation of the $*$-algebra $\mathbb{C}[SU_{2n}]_{q}.$ Using this result we prove the existence of the universal enveloping $C^{*}$-algebra of $\mathrm{Pol}(\mathrm{Mat}_{n})_{q}$ and show that the Fock representation of $\mathrm{Pol}(\mathrm{Mat}_{n})_{q}$ can be extended to a faithful representation of this $C^{*}$-algebra; the latter can be considered as a quantization of the algebra of continuous functions on the unit matrix ball $\overline{\mathbb{D}}_{n}=\{x\in \mathrm{Mat}_{n}:x^{*}x\leq 1\}.$ Finally, we classify all irreducible $*$-representations of $\mathrm{Pol}(\mathrm{Mat}_{n})_{q}$ using a diagram approach.}

\section{Introduction}
Let $q\in (0,1).$ In~\cite{SV2}, L. Vaksman and S. Sinel’shchikov put forward a construction of a $q$-analogue of Hermitian symmetric spaces of non-compact type via a $q$-analogue of the Harish-Chandra embedding. This construction gives a $*$-algebra that is a $q$-analogue of the $*$-algebra of polynomial functions on the bounded symmetric domain that is the image of the Harish-Chandra embedding. In the simplest case (see~\cite{SV2}, section $9$), this yields the algebra $Pol(\mathbb{C})_{q},$ the quantum disc; this is the $*$-algebra generated by a single generator $z$ subject to the relation
$$
z^{*} z=q^{2}z z^{*}+(1-q^{2})I.
$$
Another particular case, derived explicitly in~\cite{ssvs1}, gives the algebra $\mathrm{Pol}(\mathrm{Mat}_{n})_{q},$ a $q$-analogue of polynomial functions on the open matrix ball $\mathbb{D}_{n}:=\{Z\in \mathrm{Mat}_{n}:Z^{*}Z<I\}.$ In~\cite{ssv,ssvs1,vaksman_shilov1,vaksman-boundary}, L. Vaksman et. al. considered $\mathrm{Pol}(\mathrm{Mat}_{n})_{q}$ more carefully and explicitly developed some of the more general theory from~\cite{SV2}. In particular, the Fock representation $\pi_{F,n}$ of $\mathrm{Pol}(\mathrm{Mat}_{n})_{q}$ was introduced and proved to exist in~\cite{ssv}. This is the unique (up to equivalence), faithful, irreducible $*$-representation of $\mathrm{Pol}(\mathrm{Mat}_{n})_{q}$ determined by a non-zero vector $v$ subject to the condition
$$
\begin{array}{ccc}
\pi_{F,n}(z_{m}^{j})^{*}v=0, & m,j=1,\dots,n
\end{array}
$$
for the set of generators $\{z_{m}^{j}\}_{m,j=1}^{n}\subseteq \mathrm{Pol}(\mathrm{Mat}_{n})_{q}$ (see section $2.1$).
A complete list of irreducible $*$-representations of $\mathrm{Pol}(\mathrm{Mat}_{2})_{q}$ was given in~\cite{Lyudmyla} and in~\cite{pro_tur} the inequality \begin{equation}\label{cccc}\begin{array}{ccc}\lVert\pi(a)\rVert\leq \lVert\pi_{F,2}(a)\rVert, &  \forall a\in  \mathrm{Pol}(\mathrm{Mat}_{2})_{q}\end{array}\end{equation} was derived for any $*$-representation $\pi.$ As a consequence, this shows that the universal enveloping $C^*$-algebra of $\mathrm{Pol}(\mathrm{Mat}_{2})_{q}$ exists and is isomorphic to the closure of $\pi_{F,2}(\mathrm{Pol}(\mathrm{Mat}_{n})_{q}).$
\\

In~\cite{ool_preprint}, the author, together with O. Bershtein and L. Turowska considered the $C^{*}$-algebra $C_{F}(\overline{\mathbb{D}}_{n})_{q},$ the completion of the image of $\pi_{F,n}.$ We then described the non-commutative Shilov boundary (see~\cite{Ar}) for the closed sub-algebra $A_{F}(\mathbb{D}_{n})_{q}\subseteq C_{F}(\overline{\mathbb{D}}_{n})_{q},$ the closure of the image under $\pi_{F,n}$ of the 'analytic polynomials' in $\mathrm{Pol}(\mathrm{Mat}_{n})_{q}.$ To do this, we proposed a novel construction of the Fock representation as the composition $\pi_{F,n}\cong \pi_{u}\circ \zeta$ of an embedding $\zeta$ of $\mathrm{Pol}(\mathrm{Mat}_{n})_{q}$ into the quantum group ${C}[SU_{2n}]_{q}$ and an irreducible $*$-representation $\pi_u$ of $\mathbb{C}[SU_{2n}]_{q}.$ We introduced a neat way of calculating the action of the generators of $\mathrm{Pol}(\mathrm{Mat}_{n})_{q}$ using certain directed paths and diagrams. Moreover, it was observed that several of the known irreducible $*$-representations of $\mathrm{Pol}(\mathrm{Mat}_{n})_{q}$ factors through the $C^{*}$-algebra $ C_{F}(\overline{\mathbb{D}}_{n})_{q}.$
%The Fock representation $\pi_{F,n}$ of $\mathrm{Pol}(\mathrm{Mat}_{n})_{q}$ is a $*$-representation of $\mathrm{Pol}(\mathrm{Mat}_{n})_{q}$ to bounded operators on a Hilbert space $H_{F,n}$ spanned as $\mathrm{Pol}(\mathrm{Mat}_{n})_{q}$-representation by a unique one dimensional subspace $\{v_{0}\},$ called the vacuum vector, with the property that $\pi_{F,n}(z_{k}^{j})^{*}v_{0}=0$ for all $1\leq k,j\leq n.$ It was proven in~\cite{ssv} that the Fock representation exists and it is unique (up to unitary equivalence), faithful and irreducible.
This raises the question of whether for general $n\in \mathbb{N},$ all irreducible $*$-representations could be constructed as quotient representations of $\pi_{F,n},$ or equivalently, if the universal enveloping $C^{*}$-algebra of $\mathrm{Pol}(\mathrm{Mat}_{n})_{q}$ exists and is isomorphic to $C_{F}(\overline{\mathbb{D}}_{n})_{q}.$ 
\\

In this paper we answer this question affirmatively and we also give a complete picture of the irreducible representations of $\mathrm{Pol}(\mathrm{Mat}_{n})_{q}.$ The main idea in the proof is to use the previously mentioned embedding $\zeta:\mathrm{Pol}(\mathrm{Mat}_{n})_{q}\to\mathbb{C}[SU_{2n}]_{q}$ and to prove that for every irreducible $*$-representation $\pi$ of $\mathrm{Pol}(\mathrm{Mat}_{n})_{q}$ on a Hilbert space $K,$ we can lift $\pi$ to an irreducible $*$-representation $\Pi$ of $\mathbb{C}[SU_{2n}]_{q}$ on the same space $K,$ such that \begin{equation}\label{cc}\pi=\Pi\circ \zeta.\end{equation}
The result~\eqref{cc} is stated more precisely in Theorem~\ref{main} and the proof makes up Sections $5.$
In Corollary~\ref{maincor} we then use the lifting result to prove~\eqref{cccc} for general $n\in \mathbb{N}.$ By Soibelman's result (see~\cite{KorS}), all irreducible $*$-representations of $\mathbb{C}[SU_{2n}]_{q}$ are parametrized by elements in the symmetric group $S_{2n}$ (the Weyl group of $\mathfrak{su}_{2n}$) and $2n$-tuples $\phi=(\phi_{1},\dots, \phi_{2n})\in [0,2\pi)$ such that $\sum_{j}\phi_{j}\equiv 0 \pmod{2\pi}.$ In Section $6$ we use this result to do a more careful analysis of which irreducible $*$-representations $\Pi$ of $\mathbb{C}[SU_{2n}]_{q}$ that give rise to non-equivalent irreducible $*$-representation $\pi=\Pi\circ \zeta$ of $\mathrm{Pol}(\mathrm{Mat}_{n})_{q}$. It turns out that for $\pi=\Pi\circ \zeta$ the element $s\in S_{2n}$ corresponding to a lift $\Pi$ is uniquely determined while the coordinates of the corresponding $\phi=(\phi_{1},\dots,\phi_{2n})$ are uniquely determined only for those indices $n+1\leq j\leq 2n$ such that $n+1\leq s(j)\leq 2n.$ We then use this to give a presentation of the irreducible $*$-representations of $\mathrm{Pol}(\mathrm{Mat}_{n})_{q}$ as quotient representations of $\pi_{F,n}$ (see Section $3.1$). One advantage of this presentation is that it yields an easy way to calculate the action of the generators, by use of directed paths and diagrams in a similar way as for $\pi_{F,n}.$
\\

This paper is organized as follows: In section $2$ we introduce the $*$-algebras $\mathrm{Pol}(\mathrm{Mat}_{n})_{q}$ and $\mathbb{C}[SU_{n}]_{q}$ and review some of their basic theory.
\\

In Section $3$ we explain how to associate $*$-representations of $\mathrm{Pol}(\mathrm{Mat}_{n})_{q}$ to some kind of directed diagrams and how to use these diagrams to compute the images of the generators $z_{k}^{j}$ under the $*$-representations. We will also explain there how every irreducible $*$-representation of $\mathrm{Pol}(\mathrm{Mat}_{n})_{q}$ can be realized as one of these diagrams and we give a complete classification of irreducible $*$-representations of  $\mathrm{Pol}(\mathrm{Mat}_{n})_{q}.$
\\

In Section $4,$ we state the main result, Theorem~\ref{main}, and we show, in Corollary~\ref{maincor}, how it implies that the universal enveloping $C^{*}$-algebra of $\mathrm{Pol}(\mathrm{Mat}_{n})_{q}$ is isomorphic to $\overline{\pi_{F,n}(\mathrm{Pol}(\mathrm{Mat}_{n})_{q})}.$  
\\

The beginning of Section $5$ is devoted to explaining, in a slightly informal way, how we need to split the proof of the main result into two cases, $\textbf{A}$ and $\textbf{B},$ and the idea behind the mechanics of the proof of Theorem~\ref{main} in both cases. The hope is that this will make the later sections clearer. Subsections $5.3$ and $5.4$ complete the proof of Theorem~\ref{main} for the case $\textbf{A}$ and $\textbf{B}$ respectively.  
\\

In Section $6,$ we use Theorem~\ref{main} to complete the classification of irreducible $*$-representations of  $\mathrm{Pol}(\mathrm{Mat}_{n})_{q}$ as it was presented in Section $3.$

\section{The $*$-algebras $\mathrm{Pol}(\mathrm{Mat}_{n})_{q}$ and $\mathbb{C}[SU_{n}]_{q}$}
\subsection{$\mathrm{Pol}(\mathrm{Mat}_{n})_{q}$}
In what follows $\mathbb{C}$ is a ground field and $q \in(0,1)$. We assume
that all the algebras under consideration has a unit $I.$ Consider the well known algebra
$\mathbb{C}[\mathrm{Mat}_n]_q$ defined by its generators $z_a^\alpha$,
$\alpha,a=1,\dots,n$, and the commutation relations
\begin{flalign}
& z_a^\alpha z_b^\beta-qz_b^\beta z_a^\alpha=0, & a=b \quad \& \quad
\alpha<\beta,& \quad \text{or}\quad a<b \quad \& \quad \alpha=\beta,
\label{zaa1}
\\ & z_a^\alpha z_b^\beta-z_b^\beta z_a^\alpha=0,& \alpha<\beta \quad
\&\quad a>b,& \label{zaa2}
\\ & z_a^\alpha z_b^\beta-z_b^\beta z_a^\alpha-(q-q^{-1})z_a^\beta
z_b^\alpha=0,& \alpha<\beta \quad \& \quad a<b. & \label{zaa3}
\end{flalign}
This algebra is a quantum analogue of the polynomial algebra
$\mathbb{C}[\mathrm{Mat}_n]$ on the space of $n \times n$-matrices.
Let $A=(a_{j,k})_{j,k}\in M_{n}(\mathbb{Z}_{+})$ be a $n\times n$ matrix of positive integers $a_{j,k}$, and denote by $z(A)$ the monomial
\begin{equation}\label{fock}
\begin{array}{cc}
z(A):=\\
(z_n^n)^{a_{n,n}}(z_n^{n-1})^{a_{n,n-1}}\ldots(z_n^1)^{a_{n,1}}(z_{n-1}^n)^{a_{n-1, n}}\ldots(z_{n-1}^1)^{a_{n-1, 1}}\ldots(z_1^n)^{a_{1,n}}\ldots (z_1^1)^{a_{1,1}}.
\end{array}
\end{equation}

 It follows from the Bergman diamond lemma that monomials $z(A)$, $A\in M_{n}(\mathbb{Z}_{+}),$ form a basis for the vector space ${\mathbb C}[\mathrm{Mat}_n]_q.$ Hence ${\mathbb C}[\mathrm{Mat}_n]_q$ admits a natural gradation given by $\text{deg} (z_{k}^{j})=1$, and in general 
$\text{deg} (z(A))=|A|,$ where
\begin{equation}\label{matrixnorm}
|A|=\sum_{k,j=1}^{n}a_{k,j}.
\end{equation}

In a similar way, we write $\mathbb{C}[\overline{\mathrm{Mat}}_n]_q$ for the algebra defined by generators
$(z_a^\alpha)^*$, $\alpha,a=1,\ldots,n$, and the relations
\begin{flalign}
& (z_b^\beta)^*(z_a^\alpha)^* -q(z_a^\alpha)^*(z_b^\beta)^*=0, \quad a=b
\quad \& \quad \alpha<\beta, \qquad \text{or} & a<b \quad \& \quad
\alpha=\beta, \label{zaa1*}
\\ & (z_b^\beta)^*(z_a^\alpha)^*-(z_a^\alpha)^*(z_b^\beta)^*=0,&
\alpha<\beta \quad\&\quad a>b,& \label{zaa2*}
\\ & (z_b^\beta)^*(z_a^\alpha)^*-(z_a^\alpha)^*(z_b^\beta)^*-
(q-q^{-1})(z_b^\alpha)^*(z_a^\beta)^*=0,& \alpha<\beta \quad \& \quad a<b. &
\label{zaa3*}
\end{flalign}
A gradation in $\mathbb{C}[\overline{\mathrm{Mat}}_n]_q$ is given by $\text{deg}(z_a^\alpha)^*=-1$.
\\

Consider now the algebra $\mathrm{Pol}(\mathrm{Mat}_n)_q,$ whose generators are $z_a^\alpha$, $(z_a^\alpha)^*$, $\alpha,a=1,\dots,n$, and the list of relations is formed by \eqref{zaa1} --
\eqref{zaa3*} and
\begin{flalign}
&(z_b^\beta)^*z_a^\alpha=q^2 \cdot \sum_{a',b'=1}^n
\sum_{\alpha',\beta'=1}^n R_{ba}^{b'a'}R_{\beta \alpha}^{\beta'\alpha'}\cdot
z_{a'}^{\alpha'}(z_{b'}^{\beta'})^*+(1-q^2)\delta_{ab}\delta^{\alpha \beta},
& \label{zaa4}
\end{flalign}
with $\delta_{ab}$, $\delta^{\alpha \beta}$ being the Kronecker symbols, and
$$
R_{ij}^{kl}=
\begin{cases}
q^{-1},& i \ne j \quad \& \quad i=k \quad \& \quad j=l
\\ 1,& i=j=k=l
\\ -(q^{-2}-1), & i=j \quad \& \quad k=l \quad \& \quad l>j
\\ 0,& \text{otherwise}.
\end{cases}
$$
The involution in $\mathrm{Pol}(\mathrm{Mat}_n)_q$ is introduced in an
obvious way: $*:z_a^\alpha \mapsto(z_a^\alpha)^*$.
\\

It is known from~\cite{ssv} (Corollary 10.4) that we have an isomorphism of vector spaces
\begin{equation}\label{isotensor}
\mathbb{C}[\mathrm{Mat}_n]_q\otimes\mathbb{C}[\overline{\mathrm{Mat}}_n]_q\rightarrow \mathrm{Pol}(\mathrm{Mat}_{n})_{q}
\end{equation}
induced by the linear map $a\otimes b\mapsto a\cdot b.$
Thus, we can consider the algebras $\mathbb{C}[\mathrm{Mat}_n]_q$ and $\mathbb{C}[\overline{\mathrm{Mat}}_n]_q$ as sub-algebras of $\mathrm{Pol}(\mathrm{Mat}_{n})_{q},$ if we identify them with the images of $\mathbb{C}[\mathrm{Mat}_n]_q\otimes I$ and $I\otimes\mathbb{C}[\overline{\mathrm{Mat}}_n]_q$ respectively.
\\

For the ease of reference, we split~\eqref{zaa4} into $4$ cases that we write down explicitly as
\begin{enumerate}
\item
If $a\neq b,$ $\alpha\neq \beta$
\begin{equation}\label{zaa41}
(z_b^\beta)^*z_a^\alpha=z_a^\alpha(z_b^\beta)^*.
\end{equation}
\item If $a= b, \alpha\neq \beta$
\begin{equation}\label{zaa42}
(z_a^\beta)^*z_a^\alpha=q z_a^\alpha(z_a^\beta)^*-(q^{-1}-q)\sum_{j=a+1}^{n}z_j^\alpha(z_j^\beta)^*.
\end{equation}
\item If $a\neq b, \alpha=\beta$
\begin{equation}\label{zaa43}
(z_b^\alpha)^*z_a^\alpha=q z_a^\alpha(z_b^\alpha)^*-(q^{-1}-q)\sum_{j=\alpha+1}^{n}z_a^j(z_b^j)^*.
\end{equation}
\item If $a= b, \alpha=\beta$
\begin{multline}\label{zaa44}
(z_a^\alpha)^*z_a^\alpha=q^{2} z_a^\alpha(z_a^\alpha)^*-(1-q^{2})\sum_{j=\alpha+1}^{n}z_a^j(z_a^j)^*-(1-q^{2})\sum_{j=a+1}^{n}z_j^\alpha(z_j^\alpha)^*+\\
q^{-2}(1-q^{2})^{2}\sum_{j=\alpha+1,m=a+1}^{n}z_m^j(z_m^j)^*+(1-q^{2})I.
\end{multline}
\end{enumerate}
\subsection{The Quantum group $\mathbb{C}[SU_{n}]_{q}$}
Recall the definition of the Hopf algebra
$\mathbb{C}[SL_n]_q$. It is defined by the generators $\{t_{i,j}\}_{i,j=1,\ldots,n}$ and the relations
\begin{flalign*}
& t_{\alpha, a}t_{\beta, b}-qt_{\beta, b}t_{\alpha, a}=0, & a=b \quad \& \quad
\alpha<\beta,& \quad \text{or}\quad a<b \quad \& \quad \alpha=\beta,
\\ & t_{\alpha, a}t_{\beta, b}-t_{\beta, b}t_{\alpha, a}=0,& \alpha<\beta \quad
\&\quad a>b,&
\\ & t_{\alpha, a}t_{\beta, b}-t_{\beta, b}t_{\alpha, a}-(q-q^{-1})t_{\beta, a}
t_{\alpha, b}=0,& \alpha<\beta \quad \& \quad a<b, &
\\ & \det \nolimits_q \mathbf{t}=1.
\end{flalign*}
Here $\det_q \mathbf{t}$ is the  q-determinant of the matrix
$\mathbf{t}=(t_{i,j})_{i,j=1}^{n}$.

It is well known (see \cite{KlSh} or other standard book on quantum groups) that
$\mathbb{C}[SU_n]_q\stackrel{\mathrm{def}}{=}(\mathbb{C}[SL_n]_q,\star)$ is
a Hopf $*$-algebra. The co-multiplication $\Delta$, the co-unit $\varepsilon$, the antipode $S$ and the involution $\star$ are defined as follows
\begin{equation*}
\Delta(t_{i,j})=\sum_k t_{i,k}\otimes t_{k,j},\quad \varepsilon(t_{i,j})=\delta_{ij},\quad S(t_{i,j})=(-q)^{i-j}\det\nolimits_q\mathbf{t}_{ji},
\end{equation*}
 and
\begin{equation*}\label{star5}
t_{i,j}^\star=(-q)^{j-i}\det \nolimits_q\mathbf{t}_{ij},
\end{equation*}
where ${\mathbf t}_{ij}$ is the matrix derived from ${\mathbf t}$ by discarding its $i$-th row and $j$-th column.
\\

Let $\{e_{k}|k\in \mathbb{Z}_{+}\}$ be the standard orthonormal basis of $\ell^{2}(\mathbb{Z}_{+})$ and let $S$ be the isometric shift $Se_{k}=e_{k+1}$ on $\ell^{2}(\mathbb{Z}_{+}).$ For $q\in (0,1),$ we now introduce the operators $D_{q},C_{q}\in B(\ell^{2}(\mathbb{Z}_{+}))$ given by the formulas 
$$
\begin{array}{ccc}
C_{q}e_{m}=\sqrt{1-q^{2m}}e_{m}, &D_{q}e_{m}=q^{m}e_{m}& \text{ for all $k\in \mathbb{Z}_{+}.$}\\
\end{array} 
$$ 
Now let 

\begin{equation}\label{fan}
\begin{array}{cc}
T_{11}=S^{*}C_{q}, & T_{12}=-q D_{q},\\
T_{21}=D_{q}, & T_{22}=C_{q}S\\.
\end{array}
\end{equation}
It is not hard to verify the formulas 
$$
\begin{array}{cc}D_{q}=\sum_{j=0}^{\infty}q^{j}S^{j}(I-S S^{*})S^{*j}, & C_{q}^{2}=I-D_{q}^{2}\end{array}.
$$
It follows from this that $D_{q},C_{q}\in C^{*}(S),$ the $C^{*}$-algebra generated by $S,$ and therefore $T_{ij}\in C^{*}(S),$ for $i,j=1,2.$ 
\\

Notice that the following relations hold
\begin{equation}\label{tsu2q}
\begin{array}{ccc}
T_{11}=T_{22}^{*}& T_{11}T_{22}-q T_{12}T_{21}=0 & T_{12}T_{21}=T_{21}T_{12}\\
T_{11}T_{12}=qT_{12}T_{11} & T_{11}T_{21}=qT_{21}T_{11}& T_{21}=\sqrt{I-T_{22}T_{11}}\\
& T_{11}T_{22}=q^{2}T_{22}T_{11}+(1-q^{2})I &\\
\end{array}
\end{equation}
from which we conclude that the operators $T_{ij}$ determine an irreducible $*$-representation $\pi:\mathbb{C}[SU_{2}]_{q}\to C^{*}(S)\subset B(\ell^{2}(\mathbb{Z}_{+}))$  \begin{equation}\label{basrep}\pi(t_{ij})= T_{ij},1\leq i,j\leq 2.\end{equation}
For $1\leq i\leq n-1,$ let $\phi_{i}:\mathbb{C}[SU_{n}]_{q}\rightarrow \mathbb{C}[SU_{2}]_{q}$ be the $*$-homomorphism determined by

\begin{equation}\label{phi}
\begin{array}{ccc}
\phi_{i}(t_{i,i})=t_{1, 1}, & & \phi_{i}(t_{i+1,i+1})=t_{2, 2},\\
\phi_{i}(t_{i,i+1})=t_{1, 2}, & & \phi_{i}(t_{i+1,i})=t_{2, 1}\\
 &\text{and } \phi_{i}(t_{k,j})=\delta_{k,j}I \text{ otherwise.}&
\end{array}
\end{equation}
Then $\pi_{i}=\pi\circ \phi_{i}$ is a $*$-representation of $\mathbb{C}[SU_{n}]_{q}$ on $\ell^{2}(\mathbb{Z}_{+}).$ Moreover $\pi_{i}(\mathbb{C}[SU_{n}])\subseteq C^{*}(S).$
\\

Let $s_{i}$ denote the adjacent transposition $(i,i+1)$ in the symmetric group $S_{n}.$ 
\begin{defn}\label{surep1}
For an element $s\in S_{n}$ consider a reduced decomposition of $s=s_{j_{1}}s_{j_{2}}\dots s_{j_{m}}$ into a product of adjacent transposition and put
$$
\pi_{s}=\pi_{j_{1}}\otimes \dots \otimes \pi_{j_{m}}.
$$
\end{defn}
It is known that $\pi_{s}$ is independent of the specific reduced expression of $s,$ in the sense that another reduced decomposition gives a unitary equivalent $*$-representation.
\\

Recall that the \textit{length} of $s\in S_{n},$ denoted by $\ell(s),$ is the number of adjacent transpositions in a reduced decomposition of $s=s_{j_{1}}s_{j_{2}}\dots s_{j_{\ell(s)}}.$ For the identity element $e\in S_{n}$, we let $\ell(e)=0.$ For $s,t\in S_{n},$ it is easy to see that the inequality
$$
\ell(st)\leq \ell(s)+\ell(t).
$$
holds, and that $\ell(s^{-1})=\ell(s).$
\\

Let $\varphi=\{\varphi_{1},\dots,\varphi_{n}\}\in [0,2\pi)^{n}$ be a $n$-tuple such that 
$$
\sum_{j=1}^{n}\varphi_{j}\equiv 0 \pmod{2\pi}.
$$
Then we can define a one-dimensional $*$-representation $\chi_{\varphi}:\mathbb{C}[SU_{n}]_{q}\to \mathbb{C}$ by the formula
\begin{equation}\label{chi}
\chi_{\varphi}(t_{ij})=e^{i\varphi_{j}}\delta_{ij}.
\end{equation}
From the work of Soibelman, the following is known (with $S^{n}$ the Weyl group of $SU_{n}$)
\begin{prop}\label{soib}[~\cite{KorS}, Theorem $6.2.7$]
Every irreducible $*$-representation $\Pi$ of $\mathbb{C}[SU_{n}]_{q}$ is equivalent to one of the form 
\begin{equation}\label{surep}
\pi_{s}\otimes \chi_{\varphi}
\end{equation}
for $s\in S_{n}$ and $\varphi=[\varphi_{1},\dots,\varphi_{n}]\in [0,2\pi)^{n},\sum_{j=1}^{n}\varphi_{j}\equiv 0 \pmod{2\pi}.$ Conversly, such pairs give rise to non-equivalent irreducible $*$-representations of $\mathbb{C}[SU_{n}]_{q}.$ 
\end{prop}

It follows from Proposition~\ref{soib} that for every $s\in S_{n}$ and $\varphi=[\varphi_{1},\dots,\varphi_{n}]\in [0,2\pi)^{n}$, such that $\sum_{j=1}^{n}\varphi_{j}\equiv 0 \pmod{2\pi},$ there is another $n$-tuple $\varphi'=[\varphi'_{1},\dots,\varphi'_{n}]\in [0,2\pi)^{n}$, $\sum_{j=1}^{n}\varphi'_{j}\equiv 0 \pmod{2\pi},$ such that 
\begin{equation}\label{commutate}
\pi_{s}\otimes\chi_{\varphi}=\chi_{\varphi'}\otimes \pi_{s}.
\end{equation}
A direct calculation, shows that if $\langle\cdot,\cdot\rangle$ is the inner product on $(\ell^{2}(\mathbb{Z}_{+}))^{\otimes \ell(s)}$ and $e_{\textbf{0}}=e_{0}\otimes \cdots \otimes e_{0}\in (\ell^{2}(\mathbb{Z}_{+}))^{\otimes \ell(s)},$ then for $\Pi=\pi_{s}\otimes \chi_{\varphi}$
\begin{equation}\label{matrix}
\langle \Pi(t_{ij})e_{\textbf{0}},e_{\textbf{0}}\rangle=e^{i\varphi_{j}}(-q)^{l_{j}^{s}}\delta_{i, s(j)},
\end{equation}
with $l_{j}^{s}=\#(\{1\leq k<j|s(j)<s(k)\})$ ($\#(\cdot)$ denotes the number of element in the set).
Notice that the right-hand side of~\eqref{matrix} completely determines $s$ and $\varphi.$ We can use~\eqref{matrix} to determine $\varphi'$ in the right-hand side of~\eqref{commutate} to be \begin{equation}\label{conj}\varphi'=s^{-1}(\varphi):=[\varphi_{s^{-1}(1)},\dots,\varphi_{s^{-1}(n)}]\end{equation}
so that 
\begin{equation}\label{twist}
\pi_{s}\otimes\chi_{\varphi}\cong\chi_{s^{-1}(\varphi)}\otimes \pi_{s}.
\end{equation}
\section{$*$-Representations of $\mathrm{Pol}(\mathrm{Mat}_{n})_{q}$}
The Fock representation $$\pi_{F,n}:\mathrm{Pol}(\mathrm{Mat}_{n})_{q}\rightarrow B(H_{F,n})$$ is known to be a $*$-representation of $\mathrm{Pol}(\mathrm{Mat}_{n})_{q}$ with the property that there exists a cyclic non-zero unit vector $v_{0}\in H_{F,n},$ called a \textit{vacuum vector}, such that $$\begin{array}{cc}\pi_{F,n}(z_{k}^{j})^{*}v_{0}=0,& k,j=1,\dots ,n\end{array}$$ (see~\cite{ssv}). Strictly speaking, the term \text{vacuum vector} refers to the subspace generated by $v_{0},$ but we will abuse the terminology slightly by calling a nonzero vector $v\in H$ a vacuum vector for $\mathrm{Pol}(\mathrm{Mat}_{n})_{q}$ if there exists a $*$-representation $\pi:\mathrm{Pol}(\mathrm{Mat}_{n})_{q}\rightarrow B(H)$ such that $\pi(z_{k}^{j})^{*}v=0$ for all $1\leq k,j\leq n.$ It is then clear that the sub $*$-representation of $\pi$ generated by $v$ is isomorphic to the Fock representation.
\\

The following is known from~\cite{ool_preprint} [Corollary $1$].
\begin{prop}\label{fockbase}
The set $\{z(A)v_{0}|A\in M_{n}(\mathbb{Z}_{+})\}$ is an orthogonal basis for $H_{F,n}$ and $z(A)v_{0}\neq 0$ for all $A\in M_{n}(\mathbb{Z}_{+}).$
\end{prop}
\begin{rem}
The ordering of the generators $\{z_{k}^{j}\}$ is done slightly different in [~\cite{ool_preprint},Corollary $1$], but the statements are equivalent as the difference in ordering is restricted to the columns in the generator matrix $(z_{k}^{j})_{k,j}$ and the elements in the columns $q$-commutes by~\eqref{zaa1}.
\end{rem}
We collect some results regarding $*$-homomorphisms on $\mathrm{Pol}(\mathrm{Mat}_{n})_{q}.$
\begin{lem}\label{maps}
\begin{enumerate}[(1)]
\item The map $z_{k}^{j}\mapsto z_{k+n}^{j+n}$ uniquely extends to an injective $*$-homomorphism $\rho:\mathrm{Pol}(\mathrm{Mat}_{n})_{q}\rightarrow \mathrm{Pol}(\mathrm{Mat}_{2n})_{q}.$
\item There is a surjective $*$-homomorphism $\xi:\mathrm{Pol}(\mathrm{Mat}_{n})_{q}\rightarrow \mathbb{C}[SU_{n}]_{q}$ such that $\xi(z_{k}^{j})= (-q)^{k-n}t_{k,j}.$
\end{enumerate}
\end{lem}
The integer $n$ in the above $*$-homomorphisms will always be clear from context. 
\begin{proof}
That $\rho$ is a well defined $*$-homomorphism can be seen by noticing that equation~\eqref{zaa4} for $z_{k}^{j}$ only involves generators $z_{m}^{l}$ with $m \geq k$ and $l\geq j$ and its form depends only on the relative difference $(n-k,n-j),$ and not on the particular $n.$ Hence, for all $m\geq n,$ the map $z_{k}^{j}\mapsto z_{k+m-n}^{j+m-n}$ can be extended in a unique way to a $*$-homomorphism from $\mathrm{Pol}(\mathrm{Mat}_{n})_{q}$ to $\mathrm{Pol}(\mathrm{Mat}_{m})_{q}.$ If $m=2n,$ we get the $*$-homomorphism $\rho.$ From~\cite{vaksman_shilov1} (Theorem $2.2$), we have that the $*$-homomorphism $\xi$ is well defined and surjective. By~\cite{ool_preprint} (Theorem $2$), the Fock representation $\pi_{F,n}$ is equivalent to a representation of the form $\pi_{s}\circ \xi\circ \rho,$ for a $*$-representation $\pi_{s}$ of $\mathbb{C}[SU_{2n}]_{q}.$ As the Fock representation is faithful, it follows that $\rho$ must have trivial kernel.
\end{proof}
We also define the $*$-homomorphism $\zeta:\mathrm{Pol}(\mathrm{Mat}_{n})_{q}\to\mathbb{C}[SU_{2n}]_{q} $ by
\begin{equation}\label{zetaa}
\zeta:=\xi\circ \rho:z_{k}^{j}\mapsto (-q)^{k-n}t_{n+k,n+j}.
\end{equation}
As in the proof of Lemma~\ref{maps}, $\pi_{s}\circ \zeta=\pi_{s}\circ \xi\circ \rho\cong \pi_{F,n}$ gives that $\zeta$ is an injective $*$-homomorphism from $\mathrm{Pol}(\mathrm{Mat}_{n})_{q}$ to $\mathbb{C}[SU_{2n}]_{q}.$
\\

In the next subsection, following~\cite{ool_preprint} we present a method that allows one to describe explicitly the Fock representation by means of diagrams and directed paths. We then show how to modify these diagrams in order to get a similar description of any irreducible $*$-representation of $\mathrm{Pol}(\mathrm{Mat}_{n})_{q}$ up to equivalence.
\subsection{Directed Path Presentations of $*$-representations}
There is a diagrammatic way of calculating $\pi_{F,n}(z_{k}^{j})$ that was introduced in~\cite{ool_preprint}, using hooks and arrows on a $n\times n$ grid labelled in the following way
\begin{equation}\label{boxer}
\begin{tikzpicture}[thick,scale=0.5]
\draw[step=1.0,black,thick] (1,1) grid (7,7);
\node at (1.5,0.5) {$1$};
\node at (2.5,0.5) {$2$};
\node at (3.5,0.5) {$3$};
\node at (4.5,0.5) {$4$};
\node at (5.5,0.5) {$\dots$};
\node at (6.5,0.5) {$n$};
\node at (7.5,1.5) {$n$};
\node at (7.5,2.5) {$\vdots$};
\node at (7.5,3.5) {$4$};
\node at (7.5,4.5) {$3$};
\node at (7.5,5.5) {$2$};
\node at (7.5,6.5) {$1$};
\end{tikzpicture}
\end{equation}
where every square corresponds to a factor $C^{*}(S)\subseteq B(\ell^{2}(\mathbb{Z}_{+}))$ in the tensor product $C^{*}(S)^{\otimes n^{2}}.$ We order the factors by letting the lower left square corresponds to the first factor, the square directly above corresponds to the second factor, proceeding up the first column, and then once we are done with the first column we proceed to the second column etc. So for example, in  $$C^{*}(S)^{\otimes 2^{2}}=\underset{1}{C^{*}(S)}\otimes \underset{2}{C^{*}(S)}\otimes \underset{3}{C^{*}(S)}\otimes\underset{4}{C^{*}(S)},$$ the squares corresponds to the factors
$$
\begin{tikzpicture}[thick,scale=0.5]
\draw[step=1.0,black,thick] (1,1) grid (3,3);
\node at (1.5,0.5) {$1$};
\node at (2.5,0.5) {$2$};
\node at (3.5,1.5) {$2$};
\node at (3.5,2.5) {$1$};
\node at (1.5,2.5) {$2$};
\node at (1.5,1.5) {$1$};
\node at (2.5,2.5) {$4$};
\node at (2.5,1.5) {$3$};
\end{tikzpicture}
$$

We now represent the operators $T_{ij}$ and the identity $I$ graphically as
\begin{eqnarray}\label{arrow12}
T_{12}&\leadsto& \begin{tikzpicture}[line width=0.21mm,scale=0.4]
\draw (1,0) -- (0,0) -- (0,1)--(1,1)--(1,0)--(0,0);
\draw [->] (0.1,0.5)--(0.9,0.5);
\end{tikzpicture}\\\label{arrow21}
T_{21}&\leadsto&\begin{tikzpicture}[line width=0.21mm,scale=0.4]
\draw (1,0) -- (0,0) -- (0,1)--(1,1)--(1,0)--(0,0);
\draw [->] (1/2,0.1)--(1/2,0.9);
\end{tikzpicture}\\\label{11}
T_{11}&\leadsto&\begin{tikzpicture}[line width=0.21mm,scale=0.4]
\draw (1,0) -- (0,0) -- (0,1)--(1,1)--(1,0)--(0,0);
\draw [->] (0.1,0.5)--(0.5,0.5)--(0.5,0.9);
\end{tikzpicture}\\\label{22}
T_{22}&\leadsto&\begin{tikzpicture}[line width=0.21mm,scale=0.4]
\draw (1,0) -- (0,0) -- (0,1)--(1,1)--(1,0)--(0,0);
\draw [->] (0.5,0.1)--(0.5,0.5)--(0.9,0.5);
\end{tikzpicture}\\\label{noarrow}
I&\leadsto&
\begin{tikzpicture}[line width=0.21mm,scale=0.4]
\draw (1,0) -- (0,0) -- (0,1)--(1,1)--(1,0)--(0,0);
\end{tikzpicture}
\end{eqnarray}
Elementary tensors in $C^{*}(S)^{\otimes n^{2}}$ with $T_{ij}$ as factors can then be represented by $n\times n$ grids with hooks and arrows~\eqref{arrow12}-~\eqref{noarrow} placed at the corresponding boxes. So, for example, the operator $T_{21}\otimes T_{12}\otimes T_{11}\otimes I\in C^{*}(S)^{\otimes 2^{2}}$ will correspond to the diagram
$$
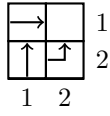
\begin{tikzpicture}[thick,scale=0.5]
\draw (1,0) -- (0,0) -- (0,1)--(1,1)--(1,0)--(0,0);
\draw (1,1) -- (0,1) -- (0,2)--(1,2)--(1,1)--(0,1);
\draw (2,0) -- (1,0) -- (1,1)--(2,1)--(2,0)--(1,0);
\draw (2,1) -- (1,1) -- (1,2)--(2,2)--(2,1)--(1,1);
\draw [->] (1/2,0.05)--(1/2,0.95);\draw [->] (0.05,1/2+1)--(0.95,1/2+1);\draw [->] (0.05+1,1/2)--(0.5+1,0.5)--(1/2+1,0.95);
\node at (0.5,-0.5) {$1$};
\node at (1.5,-0.5) {$2$};
\node at (2.5,0.5) {$2$};
\node at (2.5,1.5) {$1$};
\end{tikzpicture}
$$
To represent the images of generators of $\mathrm{Pol}(\mathrm{Mat}_{n})_{q}$ under the Fock representation, we are going to use directed connected paths from the bottom side to the right side of the $n\times n$ grid drawn with the hooks and arrows. The different terms in the image of $\pi_{F,n}(z_{k}^{j})$ can then be represented as $(-q)^{k-n}$ times the sum of elementary tensors corresponding to all possible diagrams with connected paths from the bottom side integer $k$ to the right side integer $j.$ For instance, the image of $\pi_{F,3}(z_{1}^{1})$ is $(-q)^{1-3}$ times the sum of the operators in $C^{*}(S)^{\otimes 3^{2}}$ corresponding to the diagrams
$$
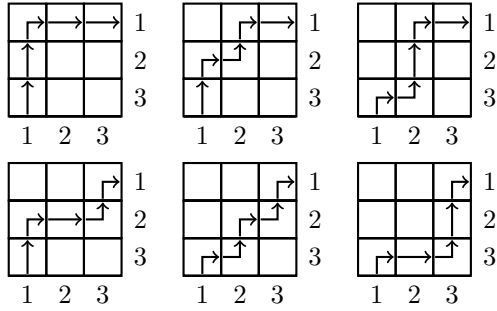
\begin{array}{cccccc}
\begin{tikzpicture}[thick,scale=0.5]
\draw (1,0) -- (0,0) -- (0,1)--(1,1)--(1,0)--(0,0);
\draw (1,1) -- (0,1) -- (0,2)--(1,2)--(1,1)--(0,1);
\draw (1,2) -- (0,2) -- (0,3)--(1,3)--(1,2)--(0,2);
\draw (2,0) -- (1,0) -- (1,1)--(2,1)--(2,0)--(1,0);
\draw (2,1) -- (1,1) -- (1,2)--(2,2)--(2,1)--(1,1);
\draw (2,2) -- (1,2) -- (1,3)--(2,3)--(2,2)--(1,2);
\draw (3,0) -- (2,0) -- (2,1)--(3,1)--(3,0)--(2,0);
\draw (3,1) -- (2,1) -- (2,2)--(3,2)--(3,1)--(2,1);
\draw (3,2) -- (2,2) -- (2,3)--(3,3)--(3,2)--(2,2);
\draw [->] (1/2,0.05)--(1/2,0.95);\draw [->] (1/2,0.05+1)--(1/2,0.95+1);\draw [->] (0.5,0.05+2)--(0.5,0.5+2)--(0.95,0.5+2);\draw [->] (0.05+1,0.5+2)--(0.95+1,0.5+2);\draw [->] (0.05+2,0.5+2)--(0.95+2,0.5+2);
\node at (0.5,-0.5) {$1$};
\node at (1.5,-0.5) {$2$};
\node at (2.5,-0.5) {$3$};
\node at (3.5,2.5) {$1$};
\node at (3.5,0.5) {$3$};
\node at (3.5,1.5) {$2$};
\end{tikzpicture}
&
\begin{tikzpicture}[thick,scale=0.5]
\draw (1,0) -- (0,0) -- (0,1)--(1,1)--(1,0)--(0,0);
\draw (1,1) -- (0,1) -- (0,2)--(1,2)--(1,1)--(0,1);
\draw (1,2) -- (0,2) -- (0,3)--(1,3)--(1,2)--(0,2);
\draw (2,0) -- (1,0) -- (1,1)--(2,1)--(2,0)--(1,0);
\draw (2,1) -- (1,1) -- (1,2)--(2,2)--(2,1)--(1,1);
\draw (2,2) -- (1,2) -- (1,3)--(2,3)--(2,2)--(1,2);
\draw (3,0) -- (2,0) -- (2,1)--(3,1)--(3,0)--(2,0);
\draw (3,1) -- (2,1) -- (2,2)--(3,2)--(3,1)--(2,1);
\draw (3,2) -- (2,2) -- (2,3)--(3,3)--(3,2)--(2,2);
\draw [->] (1/2,0.05)--(1/2,0.95);\draw [->] (0.5,0.05+1)--(0.5,0.5+1)--(0.95,0.5+1);\draw [->] (0.05+1,0.5+1)--(0.5+1,0.5+1)--(0.5+1,0.95+1);\draw [->] (0.5+1,0.05+2)--(0.5+1,0.5+2)--(0.95+1,0.5+2);\draw [->] (0.05+2,0.5+2)--(0.95+2,0.5+2);
\node at (0.5,-0.5) {$1$};
\node at (1.5,-0.5) {$2$};
\node at (2.5,-0.5) {$3$};
\node at (3.5,2.5) {$1$};
\node at (3.5,0.5) {$3$};
\node at (3.5,1.5) {$2$};
\end{tikzpicture}
&
\begin{tikzpicture}[thick,scale=0.5]
\draw (1,0) -- (0,0) -- (0,1)--(1,1)--(1,0)--(0,0);
\draw (1,1) -- (0,1) -- (0,2)--(1,2)--(1,1)--(0,1);
\draw (1,2) -- (0,2) -- (0,3)--(1,3)--(1,2)--(0,2);
\draw (2,0) -- (1,0) -- (1,1)--(2,1)--(2,0)--(1,0);
\draw (2,1) -- (1,1) -- (1,2)--(2,2)--(2,1)--(1,1);
\draw (2,2) -- (1,2) -- (1,3)--(2,3)--(2,2)--(1,2);
\draw (3,0) -- (2,0) -- (2,1)--(3,1)--(3,0)--(2,0);
\draw (3,1) -- (2,1) -- (2,2)--(3,2)--(3,1)--(2,1);
\draw (3,2) -- (2,2) -- (2,3)--(3,3)--(3,2)--(2,2);
\draw [->] (0.5,0.05)--(0.5,0.5)--(0.95,0.5);\draw [->] (0.05+1,0.5)--(0.5+1,0.5)--(0.5+1,0.95);\draw [->] (1/2+1,0.05+1)--(1/2+1,0.95+1);\draw [->] (0.5+1,0.05+2)--(0.5+1,0.5+2)--(0.95+1,0.5+2);\draw [->] (0.05+2,0.5+2)--(0.95+2,0.5+2);
\node at (0.5,-0.5) {$1$};
\node at (1.5,-0.5) {$2$};
\node at (2.5,-0.5) {$3$};
\node at (3.5,2.5) {$1$};
\node at (3.5,0.5) {$3$};
\node at (3.5,1.5) {$2$};
\end{tikzpicture}
\\
\begin{tikzpicture}[thick,scale=0.5]
\draw (1,0) -- (0,0) -- (0,1)--(1,1)--(1,0)--(0,0);
\draw (1,1) -- (0,1) -- (0,2)--(1,2)--(1,1)--(0,1);
\draw (1,2) -- (0,2) -- (0,3)--(1,3)--(1,2)--(0,2);
\draw (2,0) -- (1,0) -- (1,1)--(2,1)--(2,0)--(1,0);
\draw (2,1) -- (1,1) -- (1,2)--(2,2)--(2,1)--(1,1);
\draw (2,2) -- (1,2) -- (1,3)--(2,3)--(2,2)--(1,2);
\draw (3,0) -- (2,0) -- (2,1)--(3,1)--(3,0)--(2,0);
\draw (3,1) -- (2,1) -- (2,2)--(3,2)--(3,1)--(2,1);
\draw (3,2) -- (2,2) -- (2,3)--(3,3)--(3,2)--(2,2);
\draw [->] (1/2,0.05)--(1/2,0.95);\draw [->] (0.5,0.05+1)--(0.5,0.5+1)--(0.95,0.5+1);\draw [->] (0.05+2,0.5+1)--(0.5+2,0.5+1)--(0.5+2,0.95+1);\draw [->] (0.5+2,0.05+2)--(0.5+2,0.5+2)--(0.95+2,0.5+2);\draw [->] (0.05+1,0.5+1)--(0.95+1,0.5+1);
\node at (0.5,-0.5) {$1$};
\node at (1.5,-0.5) {$2$};
\node at (2.5,-0.5) {$3$};
\node at (3.5,2.5) {$1$};
\node at (3.5,0.5) {$3$};
\node at (3.5,1.5) {$2$};
\end{tikzpicture}
&
\begin{tikzpicture}[thick,scale=0.5]
\draw (1,0) -- (0,0) -- (0,1)--(1,1)--(1,0)--(0,0);
\draw (1,1) -- (0,1) -- (0,2)--(1,2)--(1,1)--(0,1);
\draw (1,2) -- (0,2) -- (0,3)--(1,3)--(1,2)--(0,2);
\draw (2,0) -- (1,0) -- (1,1)--(2,1)--(2,0)--(1,0);
\draw (2,1) -- (1,1) -- (1,2)--(2,2)--(2,1)--(1,1);
\draw (2,2) -- (1,2) -- (1,3)--(2,3)--(2,2)--(1,2);
\draw (3,0) -- (2,0) -- (2,1)--(3,1)--(3,0)--(2,0);
\draw (3,1) -- (2,1) -- (2,2)--(3,2)--(3,1)--(2,1);
\draw (3,2) -- (2,2) -- (2,3)--(3,3)--(3,2)--(2,2);
\draw [->] (0.5,0.05)--(0.5,0.5)--(0.95,0.5);\draw [->] (0.05+1,0.5)--(0.5+1,0.5)--(0.5+1,0.95);\draw [->] (0.5+1,0.05+1)--(0.5+1,0.5+1)--(0.95+1,0.5+1);\draw [->] (0.05+2,0.5+1)--(0.5+2,0.5+1)--(0.5+2,0.95+1);\draw [->] (0.5+2,0.05+2)--(0.5+2,0.5+2)--(0.95+2,0.5+2);
\node at (0.5,-0.5) {$1$};
\node at (1.5,-0.5) {$2$};
\node at (2.5,-0.5) {$3$};
\node at (3.5,2.5) {$1$};
\node at (3.5,0.5) {$3$};
\node at (3.5,1.5) {$2$};
\end{tikzpicture}
&
\begin{tikzpicture}[thick,scale=0.5]
\draw (1,0) -- (0,0) -- (0,1)--(1,1)--(1,0)--(0,0);
\draw (1,1) -- (0,1) -- (0,2)--(1,2)--(1,1)--(0,1);
\draw (1,2) -- (0,2) -- (0,3)--(1,3)--(1,2)--(0,2);
\draw (2,0) -- (1,0) -- (1,1)--(2,1)--(2,0)--(1,0);
\draw (2,1) -- (1,1) -- (1,2)--(2,2)--(2,1)--(1,1);
\draw (2,2) -- (1,2) -- (1,3)--(2,3)--(2,2)--(1,2);
\draw (3,0) -- (2,0) -- (2,1)--(3,1)--(3,0)--(2,0);
\draw (3,1) -- (2,1) -- (2,2)--(3,2)--(3,1)--(2,1);
\draw (3,2) -- (2,2) -- (2,3)--(3,3)--(3,2)--(2,2);
\draw [->] (1/2+2,0.05+1)--(1/2+2,0.95+1);\draw [->] (0.5,0.05)--(0.5,0.5)--(0.95,0.5);\draw [->] (0.05+2,0.5)--(0.5+2,0.5)--(0.5+2,0.95);\draw [->] (0.5+2,0.05+2)--(0.5+2,0.5+2)--(0.95+2,0.5+2);\draw [->] (0.05+1,0.5)--(0.95+1,0.5);
\node at (0.5,-0.5) {$1$};
\node at (1.5,-0.5) {$2$};
\node at (2.5,-0.5) {$3$};
\node at (3.5,2.5) {$1$};
\node at (3.5,0.5) {$3$};
\node at (3.5,1.5) {$2$};
\end{tikzpicture}\\
\end{array}
$$  
Written out, we see that the image of $\pi_{F,3}(z_{1}^{1})$ in $C^{*}(S)^{\otimes 3^{2}}$ is the operator
$$
(-q)^{-2}(T_{21}\otimes T_{21}\otimes T_{22}\otimes I\otimes I \otimes T_{12}\otimes I\otimes I\otimes T_{12}+T_{21}\otimes T_{22}\otimes I\otimes I\otimes T_{11}\otimes T_{22}\otimes I\otimes I\otimes T_{12}+
$$
$$
T_{22}\otimes I\otimes I\otimes T_{11}\otimes T_{21}\otimes T_{22}\otimes I\otimes I\otimes T_{12}+T_{21}\otimes T_{22}\otimes I\otimes I\otimes T_{12}\otimes I\otimes I\otimes T_{11}\otimes T_{22}+
$$
$$
T_{22}\otimes I\otimes I\otimes T_{11}\otimes T_{22}\otimes I \otimes I\otimes T_{11}\otimes T_{22}+T_{22}\otimes I\otimes I\otimes T_{12}\otimes I\otimes I\otimes T_{11}\otimes T_{21} \otimes T_{22}).
$$

For every $\varphi \in [0,2\pi),$ we have a $*$-homomorphism $\tau_{\varphi}:C^{*}(S)\rightarrow \mathbb{C}$ determined by
$
\tau_{\varphi}(S)=e^{i\varphi}.
$
If we apply $\tau_{\varphi}$ to any of the factors in $C^{*}(S)^{\otimes n^{2}}$ and compose this with the Fock representation, we get a new (but not necessarily irreducible!) $*$-representation of $\mathrm{Pol}(\mathrm{Mat}_{n})_{q}.$ As an example, by applying $\tau_{\varphi}$ on the $n$'th factor in $C^{*}(S)^{\otimes 3^{2}}$ we can get the so-called coherent representation, determined by the presence of a cyclic vector $\Omega$ such that
$(z_j^i)^*\Omega=0$, if $(i,j)\ne (1,1)$ and $(z_1^1)^*\Omega=e^{-i\varphi}\Omega$ . It is well known~\cite{jsw} (Proposition 1.3.3) that the coherent representation is irreducible and by~\cite{ool_preprint} (Lemma $7$) it can be obtained by applying $\tau_{\varphi}$ to one of the factors in $C^{*}(S)^{\otimes n^{2}}.$ In what follows we will prove that all irreducible $*$-representations of $\mathrm{Pol}(\mathrm{Mat}_{n})_{q}$ can be acquired by applying homomorphisms $\tau_{\varphi_{i}}$ onto a subset of the factors. In our $n\times n$ grid we represent a $C^{*}(S)$ factor that has had $\tau_{\varphi}$ applied to it by coloring the corresponding box in gray.
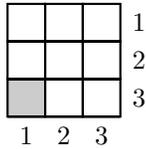
\begin{equation}\label{gray}
\begin{tikzpicture}[thick,scale=0.5]
\filldraw[fill=black!20!white, draw=black](1,0) -- (0,0) -- (0,1)--(1,1)--(1,0)--(0,0);
\draw (1,1) -- (0,1) -- (0,2)--(1,2)--(1,1)--(0,1);
\draw (1,2) -- (0,2) -- (0,3)--(1,3)--(1,2)--(0,2);
\draw (2,0) -- (1,0) -- (1,1)--(2,1)--(2,0)--(1,0);
\draw (2,1) -- (1,1) -- (1,2)--(2,2)--(2,1)--(1,1);
\draw (2,2) -- (1,2) -- (1,3)--(2,3)--(2,2)--(1,2);
\draw (3,0) -- (2,0) -- (2,1)--(3,1)--(3,0)--(2,0);
\draw (3,1) -- (2,1) -- (2,2)--(3,2)--(3,1)--(2,1);
\draw (3,2) -- (2,2) -- (2,3)--(3,3)--(3,2)--(2,2);
\node at (0.5,-0.5) {$1$};
\node at (1.5,-0.5) {$2$};
\node at (2.5,-0.5) {$3$};
\node at (3.5,2.5) {$1$};
\node at (3.5,0.5) {$3$};
\node at (3.5,1.5) {$2$};
\end{tikzpicture}
\end{equation}
(here, $\tau_{\varphi}$ has been applied to the first tensor factor). Notice that $\tau_{\varphi}(T_{12})=\tau_{\varphi}(T_{21})=0$ as they are both compact operators, while $\tau_{\varphi}(T_{11})=e^{-i\varphi}$ and $\tau_{\varphi}(T_{22})=e^{i\varphi}$ and hence paths corresponding to nonzero terms cannot contain any arrow $\begin{tikzpicture}[line width=0.21mm,scale=0.4,baseline=0.4 ex]
\draw (1,0) -- (0,0) -- (0,1)--(1,1)--(1,0)--(0,0);
\draw [->] (0.1,0.5)--(0.9,0.5);
\end{tikzpicture}$ or $\begin{tikzpicture}[line width=0.21mm,scale=0.4,baseline=0.4 ex]
\draw (1,0) -- (0,0) -- (0,1)--(1,1)--(1,0)--(0,0);
\draw [->] (0.5,0.1)--(0.5,0.9);
\end{tikzpicture}$ through the gray boxes. Thus, if $\pi:\mathrm{Pol}(\mathrm{Mat}_{3})_{q}\rightarrow C^{*}(S)^{\otimes 3^{2}-1}$ denotes the $*$-representation corresponding to the grid~\eqref{gray}, then $(-q)^{2}\pi(z_{1}^{1})$ will be the sum of the operators given by the diagrams
$$
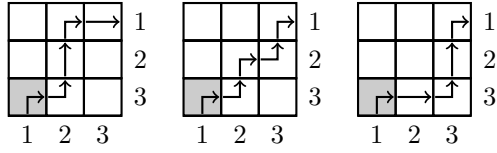
\begin{array}{cccccc}
\begin{tikzpicture}[thick,scale=0.5]
\filldraw[fill=black!20!white, draw=black](1,0) -- (0,0) -- (0,1)--(1,1)--(1,0)--(0,0);
\draw (1,1) -- (0,1) -- (0,2)--(1,2)--(1,1)--(0,1);
\draw (1,2) -- (0,2) -- (0,3)--(1,3)--(1,2)--(0,2);
\draw (2,0) -- (1,0) -- (1,1)--(2,1)--(2,0)--(1,0);
\draw (2,1) -- (1,1) -- (1,2)--(2,2)--(2,1)--(1,1);
\draw (2,2) -- (1,2) -- (1,3)--(2,3)--(2,2)--(1,2);
\draw (3,0) -- (2,0) -- (2,1)--(3,1)--(3,0)--(2,0);
\draw (3,1) -- (2,1) -- (2,2)--(3,2)--(3,1)--(2,1);
\draw (3,2) -- (2,2) -- (2,3)--(3,3)--(3,2)--(2,2);
\draw [->] (0.5,0.05)--(0.5,0.5)--(0.95,0.5);\draw [->] (0.05+1,0.5)--(0.5+1,0.5)--(0.5+1,0.95);\draw [->] (1/2+1,0.05+1)--(1/2+1,0.95+1);\draw [->] (0.5+1,0.05+2)--(0.5+1,0.5+2)--(0.95+1,0.5+2);\draw [->] (0.05+2,0.5+2)--(0.95+2,0.5+2);
\node at (0.5,-0.5) {$1$};
\node at (1.5,-0.5) {$2$};
\node at (2.5,-0.5) {$3$};
\node at (3.5,2.5) {$1$};
\node at (3.5,0.5) {$3$};
\node at (3.5,1.5) {$2$};
\end{tikzpicture}
&
\begin{tikzpicture}[thick,scale=0.5]
\filldraw[fill=black!20!white, draw=black](1,0) -- (0,0) -- (0,1)--(1,1)--(1,0)--(0,0);
\draw (1,1) -- (0,1) -- (0,2)--(1,2)--(1,1)--(0,1);
\draw (1,2) -- (0,2) -- (0,3)--(1,3)--(1,2)--(0,2);
\draw (2,0) -- (1,0) -- (1,1)--(2,1)--(2,0)--(1,0);
\draw (2,1) -- (1,1) -- (1,2)--(2,2)--(2,1)--(1,1);
\draw (2,2) -- (1,2) -- (1,3)--(2,3)--(2,2)--(1,2);
\draw (3,0) -- (2,0) -- (2,1)--(3,1)--(3,0)--(2,0);
\draw (3,1) -- (2,1) -- (2,2)--(3,2)--(3,1)--(2,1);
\draw (3,2) -- (2,2) -- (2,3)--(3,3)--(3,2)--(2,2);
\draw [->] (0.5,0.05)--(0.5,0.5)--(0.95,0.5);\draw [->] (0.05+1,0.5)--(0.5+1,0.5)--(0.5+1,0.95);\draw [->] (0.5+1,0.05+1)--(0.5+1,0.5+1)--(0.95+1,0.5+1);\draw [->] (0.05+2,0.5+1)--(0.5+2,0.5+1)--(0.5+2,0.95+1);\draw [->] (0.5+2,0.05+2)--(0.5+2,0.5+2)--(0.95+2,0.5+2);
\node at (0.5,-0.5) {$1$};
\node at (1.5,-0.5) {$2$};
\node at (2.5,-0.5) {$3$};
\node at (3.5,2.5) {$1$};
\node at (3.5,0.5) {$3$};
\node at (3.5,1.5) {$2$};
\end{tikzpicture}
&
\begin{tikzpicture}[thick,scale=0.5]
\filldraw[fill=black!20!white, draw=black](1,0) -- (0,0) -- (0,1)--(1,1)--(1,0)--(0,0);
\draw (1,1) -- (0,1) -- (0,2)--(1,2)--(1,1)--(0,1);
\draw (1,2) -- (0,2) -- (0,3)--(1,3)--(1,2)--(0,2);
\draw (2,0) -- (1,0) -- (1,1)--(2,1)--(2,0)--(1,0);
\draw (2,1) -- (1,1) -- (1,2)--(2,2)--(2,1)--(1,1);
\draw (2,2) -- (1,2) -- (1,3)--(2,3)--(2,2)--(1,2);
\draw (3,0) -- (2,0) -- (2,1)--(3,1)--(3,0)--(2,0);
\draw (3,1) -- (2,1) -- (2,2)--(3,2)--(3,1)--(2,1);
\draw (3,2) -- (2,2) -- (2,3)--(3,3)--(3,2)--(2,2);
\draw [->] (1/2+2,0.05+1)--(1/2+2,0.95+1);\draw [->] (0.5,0.05)--(0.5,0.5)--(0.95,0.5);\draw [->] (0.05+2,0.5)--(0.5+2,0.5)--(0.5+2,0.95);\draw [->] (0.5+2,0.05+2)--(0.5+2,0.5+2)--(0.95+2,0.5+2);\draw [->] (0.05+1,0.5)--(0.95+1,0.5);
\node at (0.5,-0.5) {$1$};
\node at (1.5,-0.5) {$2$};
\node at (2.5,-0.5) {$3$};
\node at (3.5,2.5) {$1$};
\node at (3.5,0.5) {$3$};
\node at (3.5,1.5) {$2$};
\end{tikzpicture}\\
\end{array}
$$
Explicitly, we get that $\pi(z_{1}^{1})$ is equal to
$$
(-q)^{-2}(e^{i\varphi}I\otimes I\otimes T_{11}\otimes T_{21}\otimes T_{22}\otimes I\otimes I \otimes T_{12}+e^{i\varphi}I\otimes I\otimes T_{11}\otimes T_{22}\otimes I\otimes I\otimes T_{11}\otimes T_{22}+
$$
$$
e^{i\varphi}I\otimes I\otimes T_{12} \otimes I\otimes I\otimes T_{11}\otimes T_{21} \otimes T_{22}).
$$
The representation $\pi$ is not irreducible as no upward pointing arrow $\begin{tikzpicture}[line width=0.21mm,scale=0.4,baseline=0.4 ex]
\draw (1,0) -- (0,0) -- (0,1)--(1,1)--(1,0)--(0,0);
\draw [->] (0.5,0.1)--(0.5,0.9);
\end{tikzpicture}$ can cross the gray square and hence for all $1\leq k,j\leq 3$ the operators $\pi(z_{k}^{j})$ will have $I$'s in the two first tensor factors. In Section $6$ we classify the grids that corresponds to irreducible $*$-representations and show that such $*$-representations exhaust all irreducible $*$-representations of $\mathrm{Pol}(\mathrm{Mat}_n)_q,$ up to equivalence. In general however, an irreducible $*$-representation can be represented as a grid in many different ways, for example, for suitable chosen $ \tau_{\vartheta}$'s, the two grids
$$
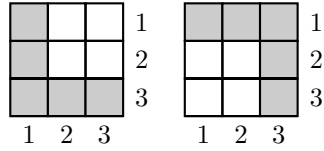
\begin{array}{cc}
\begin{tikzpicture}[thick,scale=0.5]
\filldraw[fill=black!20!white, draw=black](1,0) -- (0,0) -- (0,1)--(1,1)--(1,0)--(0,0);
\filldraw[fill=black!20!white, draw=black] (1,1) -- (0,1) -- (0,2)--(1,2)--(1,1)--(0,1);
\filldraw[fill=black!20!white, draw=black] (1,2) -- (0,2) -- (0,3)--(1,3)--(1,2)--(0,2);
\filldraw[fill=black!20!white, draw=black](2,0) -- (1,0) -- (1,1)--(2,1)--(2,0)--(1,0);
\draw (2,1) -- (1,1) -- (1,2)--(2,2)--(2,1)--(1,1);
\draw (2,2) -- (1,2) -- (1,3)--(2,3)--(2,2)--(1,2);
\filldraw[fill=black!20!white, draw=black] (3,0) -- (2,0) -- (2,1)--(3,1)--(3,0)--(2,0);
\draw (3,1) -- (2,1) -- (2,2)--(3,2)--(3,1)--(2,1);
\draw (3,2) -- (2,2) -- (2,3)--(3,3)--(3,2)--(2,2);
\node at (0.5,-0.5) {$1$};
\node at (1.5,-0.5) {$2$};
\node at (2.5,-0.5) {$3$};
\node at (3.5,2.5) {$1$};
\node at (3.5,0.5) {$3$};
\node at (3.5,1.5) {$2$};
\end{tikzpicture}
&
\begin{tikzpicture}[thick,scale=0.5]
\draw(1,0) -- (0,0) -- (0,1)--(1,1)--(1,0)--(0,0);
\draw (1,1) -- (0,1) -- (0,2)--(1,2)--(1,1)--(0,1);
\filldraw[fill=black!20!white, draw=black] (1,2) -- (0,2) -- (0,3)--(1,3)--(1,2)--(0,2);
\draw (2,0) -- (1,0) -- (1,1)--(2,1)--(2,0)--(1,0);
\draw (2,1) -- (1,1) -- (1,2)--(2,2)--(2,1)--(1,1);
\filldraw[fill=black!20!white, draw=black] (2,2) -- (1,2) -- (1,3)--(2,3)--(2,2)--(1,2);
\filldraw[fill=black!20!white, draw=black] (3,0) -- (2,0) -- (2,1)--(3,1)--(3,0)--(2,0);
\filldraw[fill=black!20!white, draw=black] (3,1) -- (2,1) -- (2,2)--(3,2)--(3,1)--(2,1);
\filldraw[fill=black!20!white, draw=black] (3,2) -- (2,2) -- (2,3)--(3,3)--(3,2)--(2,2);
\node at (0.5,-0.5) {$1$};
\node at (1.5,-0.5) {$2$};
\node at (2.5,-0.5) {$3$};
\node at (3.5,2.5) {$1$};
\node at (3.5,0.5) {$3$};
\node at (3.5,1.5) {$2$};
\end{tikzpicture}
\end{array}
$$
correspond to equivalent $*$-representations.
\begin{defn}
A sequence $[k_{n},k_{n-1},\dots,k_{1}]\in \mathbb{Z}_{+}^{n}$ is said to be \textit{admissible} if 
\begin{equation}\label{seq}
0\leq k_{j}\leq \max_{j<i\leq n}(k_{i}+j+1-i,j).
\end{equation}
For sequences of positive integers $\textbf{k}=[k_{n},\dots, k_{1}]$ and numbers $\varphi=[\varphi_{n},\dots, \varphi_{1}]\in [0,2\pi)^{n},$ we shall call the sequence of pairs
\begin{equation}\textbf{k}_{\varphi}:=\label{admissible}[(k_{n},\varphi_{n}),(k_{n-1},\varphi_{n-1}),\dots,(k_{1},\varphi_{1})]\end{equation} admissible if $\textbf{k}$ is admissible and
\begin{equation}\label{seq2}
\begin{cases}\varphi_{j}\in[0,2\pi) & \text{ if $0\leq k_{j}<\max_{j<i\leq n}(k_{i}+j+1-i,j)$}\\\varphi_{j}=0 & \text{ if $ k_{j}=\max_{j<i\leq n}(k_{i}+j+1-i,j).$}\end{cases}
\end{equation}

\end{defn}
Notice that it follows from the definitions that \begin{equation}\label{less}\begin{array}{cc} k_{j}\leq n,  &j=1,2,\dots, n\end{array}\end{equation}  for an admissible sequence $[k_{n},k_{n-1},\dots,k_{1}].$
\\

To each such admissible sequence $\textbf{k}_{\varphi}$ we associate a colored $n\times n$ grid constructed the following way:
if $$k_{j}<\max_{j<i\leq n}(k_{i}+j+1-i,j),$$ then the $j$'th row of the grid consists of $k_{j}$ non-shaded boxes to the right, one lightly shaded box to the left of the non-shaded boxes and the remaining boxes are dark shaded 
$$
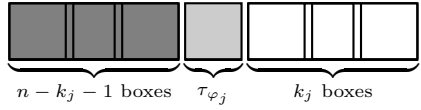
\underbrace{\begin{tikzpicture}[thick,scale=0.5]
\filldraw[fill=black!50!white, draw=black] (4.5,0) -- (0,0) -- (0,1.5)--(4.5,1.5)--(4.5,0)--(0,0);
\draw[draw=black] (1.5+1.3,0) -- (0+1.7,0) -- (0+1.7,1.5)--(1.5+1.3,1.5)--(1.5+1.3,0)--(0+1.7,0);
\filldraw[step=1.5,fill=black!50!white, draw=black] (0,0) grid (4.5,1.5);
\end{tikzpicture}}_{\text{$n-k_{j}-1$ boxes}}
\underbrace{\begin{tikzpicture}[thick,scale=0.5]
\filldraw[fill=black!20!white, draw=black] (1.5,0) -- (0,0) -- (0,1.5)--(1.5,1.5)--(1.5,0)--(0,0);
\end{tikzpicture}}_{\tau_{\varphi_{j}}}
\underbrace{\begin{tikzpicture}[thick,scale=0.5]
\filldraw[fill=white, draw=black] (4.5,0) -- (0,0) -- (0,1.5)--(4.5,1.5)--(4.5,0)--(0,0);
\draw[draw=black] (1.5+1.3,0) -- (0+1.7,0) -- (0+1.7,1.5)--(1.5+1.3,1.5)--(1.5+1.3,0)--(0+1.7,0);
\filldraw[step=1.5,fill=black!50!white, draw=black] (0,0) grid (4.5,1.5);
\end{tikzpicture}}_{\text{$k_{j}$ boxes}}
$$
In the other case, when $$k_{j}=\max_{j<i\leq n}(k_{i}+j+1-i,j),$$ we replace the lightly shaded box with a dark shaded one. If $k_{j}=n,$ then the whole $j$'th row consists of non-shaded boxes. However, if $k_{j}=n,$ then~\eqref{less} forces $k_{j}=\max_{j<i\leq n}(k_{i}+j+1-i,j)$ and hence $\varphi_{j}=0.$ So this does not give rise to any ambiguity.
\\

Now, to each $n\times n$ grid corresponding to an admissible sequence $\textbf{k}_{\varphi}$ we associate a $*$-representation $\pi_{\textbf{k}_{\varphi}}$ of $\mathrm{Pol}(\mathrm{Mat}_{n})_{q},$ $\textbf{k}_{\varphi}\to \pi_{\textbf{k}_{\varphi}}$ by starting with the white $n\times n$ grid~\eqref{boxer} associated to the Fock representation. Recall that each box corresponds to a tensor-factor of $C^{*}(S).$ If a box corresponding to a tensor-factor $C^{*}(S)$ is colored dark gray in the grid for $\textbf{k}_{\varphi},$ then we compose the Fock representation with the $*$-homomorphism $\tau_{0}$ applied to this factor, and if the box is colored light gray (and hence has a number $\varphi_{j}\in [0,2\pi)$ associated to it) we compose the Fock representation with $\tau_{\varphi_{j}}$ applied to the factor. In this way we get a $*$-representation $\pi_{\textbf{k}_{\varphi}}$ of $\mathrm{Pol}(\mathrm{Mat}_{n})_{q}.$ Moreover, it the follows that we can calculate the images of the generators $z_{k}^{j}\in \mathrm{Pol}(\mathrm{Mat}_{n})_{q}$ under $\pi_{\textbf{k}_{\varphi}}$ using the hooks-and-arrow diagrams, in the fashion that we outlined in this section.
\\

For instance, the sequence $[(0,\varphi),(2,0),(2,0)]$ corresponds to the grid
$$
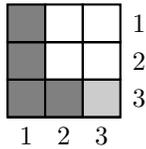
\begin{tikzpicture}[thick,scale=0.5]
\filldraw[fill=black!50!white, draw=black](1,0) -- (0,0) -- (0,1)--(1,1)--(1,0)--(0,0);
\filldraw[fill=black!50!white, draw=black] (1,1) -- (0,1) -- (0,2)--(1,2)--(1,1)--(0,1);
\filldraw[fill=black!50!white, draw=black] (1,2) -- (0,2) -- (0,3)--(1,3)--(1,2)--(0,2);
\filldraw[fill=black!50!white, draw=black](2,0) -- (1,0) -- (1,1)--(2,1)--(2,0)--(1,0);
\draw (2,1) -- (1,1) -- (1,2)--(2,2)--(2,1)--(1,1);
\draw (2,2) -- (1,2) -- (1,3)--(2,3)--(2,2)--(1,2);
\filldraw[fill=black!20!white, draw=black] (3,0) -- (2,0) -- (2,1)--(3,1)--(3,0)--(2,0);
\draw (3,1) -- (2,1) -- (2,2)--(3,2)--(3,1)--(2,1);
\draw (3,2) -- (2,2) -- (2,3)--(3,3)--(3,2)--(2,2);
\node at (0.5,-0.5) {$1$};
\node at (1.5,-0.5) {$2$};
\node at (2.5,-0.5) {$3$};
\node at (3.5,2.5) {$1$};
\node at (3.5,0.5) {$3$};
\node at (3.5,1.5) {$2$};
\end{tikzpicture}
$$
and $[(3,0),(3,0),(2,\varphi)]$ gives the grid
$$
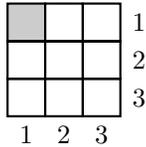
\begin{tikzpicture}[thick,scale=0.5]
\draw(1,0) -- (0,0) -- (0,1)--(1,1)--(1,0)--(0,0);
\draw(1,1) -- (0,1) -- (0,2)--(1,2)--(1,1)--(0,1);
\filldraw[fill=black!20!white, draw=black] (1,2) -- (0,2) -- (0,3)--(1,3)--(1,2)--(0,2);
\draw(2,0) -- (1,0) -- (1,1)--(2,1)--(2,0)--(1,0);
\draw (2,1) -- (1,1) -- (1,2)--(2,2)--(2,1)--(1,1);
\draw (2,2) -- (1,2) -- (1,3)--(2,3)--(2,2)--(1,2);
\draw (3,0) -- (2,0) -- (2,1)--(3,1)--(3,0)--(2,0);
\draw (3,1) -- (2,1) -- (2,2)--(3,2)--(3,1)--(2,1);
\draw (3,2) -- (2,2) -- (2,3)--(3,3)--(3,2)--(2,2);
\node at (0.5,-0.5) {$1$};
\node at (1.5,-0.5) {$2$};
\node at (2.5,-0.5) {$3$};
\node at (3.5,2.5) {$1$};
\node at (3.5,0.5) {$3$};
\node at (3.5,1.5) {$2$};
\end{tikzpicture}
$$
which corresponds to the coherent representation. 

\begin{exa}
In~\cite{Lyudmyla}, the $7$ different families of irreducible $*$-representations of $\mathrm{Pol}(\mathrm{Mat}_{2})_{q}$ was classified. They corresponds to $$[(2,0),(2,0)],[(2,0),(1,\varphi)],[(1,\varphi),(1,0)],[(2,0),(0,\varphi)],[(1,\varphi_{1}),(0,\varphi_{2})],
$$
$$
[(0,\varphi),(1,0)],[(0,\varphi_{1}),(0,\varphi_{2})]$$
and the grids
$$
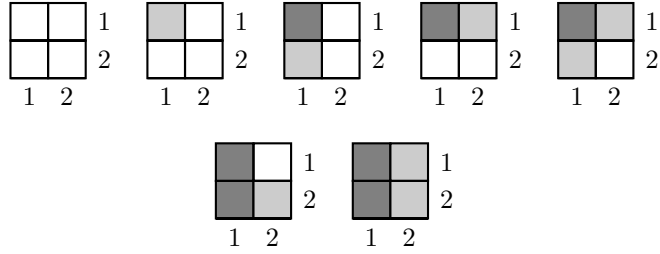
\begin{array}{ccccc}
\begin{tikzpicture}[thick,scale=0.5]
\draw (1,0) -- (0,0) -- (0,1)--(1,1)--(1,0)--(0,0);
\draw (1,1) -- (0,1) -- (0,2)--(1,2)--(1,1)--(0,1);
\draw (2,0) -- (1,0) -- (1,1)--(2,1)--(2,0)--(1,0);
\draw (2,1) -- (1,1) -- (1,2)--(2,2)--(2,1)--(1,1);
\node at (0.5,-0.5) {$1$};
\node at (1.5,-0.5) {$2$};
\node at (2.5,0.5) {$2$};
\node at (2.5,1.5) {$1$};
\end{tikzpicture}
&
\begin{tikzpicture}[thick,scale=0.5]
\draw (1,0) -- (0,0) -- (0,1)--(1,1)--(1,0)--(0,0);
\filldraw[fill=black!20!white, draw=black]  (1,1) -- (0,1) -- (0,2)--(1,2)--(1,1)--(0,1);
\draw (2,0) -- (1,0) -- (1,1)--(2,1)--(2,0)--(1,0);
\draw (2,1) -- (1,1) -- (1,2)--(2,2)--(2,1)--(1,1);
\node at (0.5,-0.5) {$1$};
\node at (1.5,-0.5) {$2$};
\node at (2.5,0.5) {$2$};
\node at (2.5,1.5) {$1$};
\end{tikzpicture}
&
\begin{tikzpicture}[thick,scale=0.5]
\filldraw[fill=black!20!white, draw=black](1,0) -- (0,0) -- (0,1)--(1,1)--(1,0)--(0,0);
\filldraw[fill=black!50!white, draw=black]  (1,1) -- (0,1) -- (0,2)--(1,2)--(1,1)--(0,1);
\draw (2,0) -- (1,0) -- (1,1)--(2,1)--(2,0)--(1,0);
\draw (2,1) -- (1,1) -- (1,2)--(2,2)--(2,1)--(1,1);
\node at (0.5,-0.5) {$1$};
\node at (1.5,-0.5) {$2$};
\node at (2.5,0.5) {$2$};
\node at (2.5,1.5) {$1$};
\end{tikzpicture}
&
\begin{tikzpicture}[thick,scale=0.5]
\draw (1,0) -- (0,0) -- (0,1)--(1,1)--(1,0)--(0,0);
\filldraw[fill=black!50!white, draw=black] (1,1) -- (0,1) -- (0,2)--(1,2)--(1,1)--(0,1);
\draw(2,0) -- (1,0) -- (1,1)--(2,1)--(2,0)--(1,0);
\filldraw[fill=black!20!white, draw=black] (2,1) -- (1,1) -- (1,2)--(2,2)--(2,1)--(1,1);
\node at (0.5,-0.5) {$1$};
\node at (1.5,-0.5) {$2$};
\node at (2.5,0.5) {$2$};
\node at (2.5,1.5) {$1$};
\end{tikzpicture}
&
\begin{tikzpicture}[thick,scale=0.5]
\filldraw[fill=black!20!white, draw=black] (1,0) -- (0,0) -- (0,1)--(1,1)--(1,0)--(0,0);
\filldraw[fill=black!50!white, draw=black] (1,1) -- (0,1) -- (0,2)--(1,2)--(1,1)--(0,1);
\draw (2,0) -- (1,0) -- (1,1)--(2,1)--(2,0)--(1,0);
\filldraw[fill=black!20!white, draw=black](2,1) -- (1,1) -- (1,2)--(2,2)--(2,1)--(1,1);
\node at (0.5,-0.5) {$1$};
\node at (1.5,-0.5) {$2$};
\node at (2.5,0.5) {$2$};
\node at (2.5,1.5) {$1$};
\end{tikzpicture}
\end{array}
$$
$$
\begin{array}{ccc}
\begin{tikzpicture}[thick,scale=0.5]
\filldraw[fill=black!50!white, draw=black] (1,0) -- (0,0) -- (0,1)--(1,1)--(1,0)--(0,0);
\filldraw[fill=black!50!white, draw=black] (1,1) -- (0,1) -- (0,2)--(1,2)--(1,1)--(0,1);
\filldraw[fill=black!20!white, draw=black] (2,0) -- (1,0) -- (1,1)--(2,1)--(2,0)--(1,0);
\draw(2,1) -- (1,1) -- (1,2)--(2,2)--(2,1)--(1,1);
\node at (0.5,-0.5) {$1$};
\node at (1.5,-0.5) {$2$};
\node at (2.5,0.5) {$2$};
\node at (2.5,1.5) {$1$};
\end{tikzpicture}
&
\begin{tikzpicture}[thick,scale=0.5]
\filldraw[fill=black!50!white, draw=black] (1,0) -- (0,0) -- (0,1)--(1,1)--(1,0)--(0,0);
\filldraw[fill=black!50!white, draw=black] (1,1) -- (0,1) -- (0,2)--(1,2)--(1,1)--(0,1);
\filldraw[fill=black!20!white, draw=black] (2,0) -- (1,0) -- (1,1)--(2,1)--(2,0)--(1,0);
\filldraw[fill=black!20!white, draw=black](2,1) -- (1,1) -- (1,2)--(2,2)--(2,1)--(1,1);
\node at (0.5,-0.5) {$1$};
\node at (1.5,-0.5) {$2$};
\node at (2.5,0.5) {$2$};
\node at (2.5,1.5) {$1$};
\end{tikzpicture}
\end{array}
$$
where $\varphi,\varphi_{1},\varphi_{2}\in [0,2\pi)$ are arbitrary. 
\end{exa}
For irreducible $*$-representation $\pi,$ let $[\pi]$ denote the set of all $*$-representations unitarily equivalent to $\pi.$ In general, we have the following. 
\begin{thm}\label{mama}
The map $\textbf{k}_{\varphi}\to [\pi_{\textbf{k}_{\varphi}}]$ gives a one to one correspondence between admissible sequences of pairs $\textbf{k}_{\varphi}=[(k_{n},\varphi_{n}),\dots,(k_{1},\varphi_{1})]$ and equivalence classes of irreducible $*$-representations of $\mathrm{Pol}(\mathrm{Mat}_{n})_{q}.$
\end{thm}

We will prove the theorem in section $6.$ There, we will also give a slightly different interpretation of an admissible sequence
$$[k_{n},k_{n-1},\dots,k_{1}].$$
Let $S$ be the subgroup of the symmetric group $S_{2n}$ that fixes the integers $n+1,\dots, 2n$ and let $\sigma\in S_{2n}$ be any element.
In Proposition~\ref{1}, we show that the orbit set
$$
O_{\sigma}:=\{g\sigma h|g,h\in S\}
$$
contains a unique element $w\in O_{\sigma}$ of minimal length $\ell(w)=\min_{t\in O_{\omega}}\ell(t).$ If we define the cycles in $S_{2n}$
$$
c_{k,j}=\begin{cases}s_{j+n-k}s_{j+n-k+1}\cdots s_{j+n-1}& \text{if $1\leq k\leq n$}\\ e & \text{if $k=0$}\end{cases}
$$ 
where $e\in S_{2n}$ is the identity, then we prove in Proposition~\ref{decomp} that there is a unique admissible sequence of integers $$[k_{n},k_{n-1},\dots,k_{1}]$$ such that
\begin{equation}\label{ww}
w=c_{k_{n},n}c_{k_{n-1},n-1}\cdots c_{k_{1},1}
\end{equation}
and the length of $w$ is
$$
\ell(w)=\sum_{j=1}^{n}k_{j}
$$
i.e. the decomposition~\eqref{ww} is minimal.
Conversely, if we start with an admissible sequence $[k_{n},k_{n-1},\dots,k_{1}],$ then we prove in Proposition~\ref{decomp1} that the group-element
$$
w=c_{k_{n},n}c_{k_{n-1},n-1}\cdots c_{k_{1},1}
$$
is of minimal length in $O_{w}$ and $\ell(w)=\sum_{j=1}^{n}k_{j}.$ Moreover, by Proposition~\ref{switch}
$$
k_{j}= \max_{j<i\leq n}(k_{i}+j+1-i,j)
$$
if and only if $1\leq w(n+j)\leq n.$
\\

For $t,r\in S_{2n},$ either $O_{t}=O_{r}$ or $O_{t}\cap O_{r}=\emptyset.$ Hence $S_{2n}$ is a union of finite number of disjoint subsets 
\begin{equation}\label{s2nsets}
S_{2n}=O_{e}\cup O_{t_{1}}\cup O_{t_{2}}\cup \dots
\end{equation}
If we let $A_{n}$ denote the cardinality of the set $\{O_{t}:t\in S_{2n}\}$, then it follows that $A_{n}$ corresponds to the number of different irreducible diagrams and hence to the number of different families of irreducible $*$-representation of $\mathrm{Pol}(\mathrm{Mat}_{n})_{q}.$
The sequence $A_{n}$ (starting with $n=1$) is
$$
\begin{array}{ccccccccccc}
2,&7,&34,&209,&1546,...&\text{$A002720$ in OEIS.}
\end{array}
$$
One can calculate the generating function for the sequence $\frac{A_{n}}{n!}$ as $$1+\sum_{n=1}^{\infty}\frac{A_{n}}{n!}x^{n}=\frac{1}{1-x}e^{\frac{x}{1-x}}.$$

\section{Main results}
\begin{thm}\label{main}
For any irreducible $*$-representation $\pi:\mathrm{Pol}(\mathrm{Mat}_{n})_{q}\to B(K),$ there is an irreducible $*$-representation $\Pi:\mathbb{C}[SU_{2n}]_{q}\to B(K)$ such that \begin{equation}\label{liftingthm}\pi=\Pi\circ \zeta\end{equation} and hence $\pi$ is equivalent to $(\pi_{w}\otimes \chi_{\varphi})\circ \zeta,$ for some $w\in S_{2n},$ $\varphi\in [0,2\pi)^{2n}.$ Moreover, if $(\pi_{w}\otimes \chi_{\varphi})\circ \zeta,$ $(\pi_{\sigma}\otimes \chi_{\psi})\circ \zeta$ are irreducible $*$-representations of $\mathrm{Pol}(\mathrm{Mat}_{n})_{q}$, then $$(\pi_{w}\otimes \chi_{\varphi})\circ \zeta\cong (\pi_{\sigma}\otimes \chi_{\psi})\circ \zeta$$ only if $w=\sigma.$
\end{thm}
The proof of the first item makes up sections $5.$ The second statement is proved in Section $6$ (Lemma~\ref{sunandmoon}), where we also specify those elements $w\in S_{2n}$ which give rise to irreducible $*$-representations $(\pi_{w}\otimes \chi_{\varphi})\circ \zeta$ of $\mathrm{Pol}(\mathrm{Mat}_{n})_{q}.$
\\

We remark here that Theorem~\ref{main} does not claim that $\Pi$ maps $\mathbb{C}[SU_{2n}]_{q}$ into the $C^{*}$-algebra generated by the image of $\mathrm{Pol}(\mathrm{Mat}_{n})_{q}$ in $B(K).$ There are many ways to choose the lift $\Pi$ such that $\Pi\circ \zeta=\pi,$ but in general, it is impossible to find a $\Pi$ such that 
\begin{equation}\label{inclusion}
\Pi(\mathbb{C}[SU_{2n}]_{q})\subseteq \overline{\pi(\mathrm{Pol}(\mathrm{Mat}_{n})_{q})}
\end{equation}
for $n\geq 2.$ 
\\

Let $A$ be a $*$-algebra. We can define a semi-norm $ |  |  \cdot|  | $ on $A,$ with values in $[0,\infty],$ by the formula $$||a||_{u}=\sup_{\phi}||\phi(a)||$$ where the supremum ranges over all $*$-representations $A.$ It follows that $|  |  \cdot|  |_{u}$ is a semi $C^{*}$-norm on $A,$ in the sense that \begin{equation}\label{semic}\begin{array}{ccc}|  | ab|  |_{u}\leq |  | a|  |_{u}|  |  b|  |_{u} , &|| a^{*}| |_{u}=|| a |  |_{u},&  || a a^{*}|  |_{u}=|  |  a|  |_{u}^{2}.\end{array}\end{equation} If $|  |  a|  |_{u}<\infty$ for all $a\in A,$ we let $I\subseteq A$ be the subset $\{a\in A;||a||_{u}=0\}.$ It follows from~\eqref{semic} that $I$ is a $*$-ideal. We can then, using $|  |  \cdot|  |_{u},$ define a $C^{*}$-norm on $A/I$ (that we again denote by $|  |  \cdot|  |_{u}$). The \textit{universal enveloping $C^{*}$-algebra} $C^{*}(A)$ of $A$ is defined to be the closure of $A/I$ under $|  |  \cdot|  |_{u}.$ We have that
\begin{itemize}
\item $C^{*}(A)$ is a $C^{*}$-algebra.
\item $\iota:A\to C^{*}(A)$ defined by $a\mapsto a+I\in A/I\subseteq C^{*}(A)$ is a $*$-homomorphism.
\item If $\phi:A\rightarrow B(H)$ is a $*$-representation, then there is a unique $*$-homomorphism $\tilde{\phi}:C^{*}(A)\to B(H),$ such that $$\phi=\tilde{\phi}\circ \iota.$$
\end{itemize}
For a $*$-representation $\pi:\mathrm{Pol}(\mathrm{Mat}_{n})_{q}\to B(H)$ and $a\in \mathrm{Pol}(\mathrm{Mat}_{n})_{q},v\in H$ we often denote 
$\pi(a)v$ as $a v.$ This is to simplify the notations and no ambiguity will arise from it.
\\

We have the following corollary to Theorem~\ref{main}.
\begin{cor}\label{maincor}
The universal enveloping $C^{*}$-algebra of $\mathrm{Pol}(\mathrm{Mat}_{n})_{q}$ exists and is isomorphic to $C_{F}(\overline{\mathbb{D}}_{n}):=\overline{\pi_{F,n}(\mathrm{Pol}(\mathrm{Mat}_{n})_{q})}$.
\end{cor}

\begin{proof}
Fix a $*$-representation $\pi$ of $\mathrm{Pol}(\mathrm{Mat}_{n})_{q}.$ It is enough to prove the inequality $$\lVert\pi(a)\rVert\leq \lVert\pi_{F,n}(a)\rVert$$ for all $a\in \mathrm{Pol}(\mathrm{Mat}_{n})_{q}.$ 
By Theorem~\ref{main}, there is a $*$-representation $\phi$ of $\mathbb{C}[SU_{2n}]_{q}$ such that $\pi=\phi \circ \zeta$ and $\zeta=\xi\circ \rho$ where $$\begin{array}{cc}\xi: \mathrm{Pol}(\mathrm{Mat}_{2n})_{q}\to \mathbb{C}[SU_{2n}]_{q} ,&\rho : \mathrm{Pol}(\mathrm{Mat}_{n})_{q}\to \mathrm{Pol}(\mathrm{Mat}_{2n})_{q}\end{array}$$ are the maps defined in Lemma~\ref{maps}. As $\phi \circ \xi$ is a $*$-representation of $\mathrm{Pol}(\mathrm{Mat}_{2n})_{q}$ that factors through $\mathbb{C}[SU_{2n}]_{q},$ it must annihilate the Shilov boundary of $\mathrm{Pol}(\mathrm{Mat}_{2n})_{q}$ and hence (by Lemma $12$ in~\cite{ool_preprint}) it will be dominated by the Fock representation $\pi_{F,2n},$ i.e. $\lVert(\phi\circ \xi) (a)\rVert\leq \lVert\pi_{F,2n}(a)\rVert$ for all $a\in \mathrm{Pol}(\mathrm{Mat}_{2n})_{q}.$
\\

Hence
$$
\lVert\pi(a)\rVert=\lVert\phi\circ \xi\circ \rho(a)\rVert\leq \lVert \pi_{F,2n}\circ\rho(a)\rVert, a\in \mathrm{Pol}(\mathrm{Mat}_{n})_{q}.
$$
Therefore, we only need to show the inequality $$\lVert (\pi_{F,2n}\circ\rho)(a)\rVert\leq \lVert \pi_{F,n}(a)\rVert,a \in \mathrm{Pol}(\mathrm{Mat}_{n})_{q}.$$ To this end, we are going to prove that the $*$-representation $ \pi_{F,2n}\circ\rho$ is a direct sum of the Fock representations $\pi_{F,n}$. To see this, we introduce an order on the generators $\{z_{k}^{j}\}_{k,j}:$ $z_{k}^{j}<z_{m}^{l}$ if either $k<m$ or if $k=m,$ then $j<l.$ We can visualize this ordering using the matrix $\textbf{Z}=(z_{k}^{j})_{k,j}$ of generators of $\mathrm{Pol}(\mathrm{Mat}_{2n})_{q}$ as
\begin{equation}\label{zmatrix}
\left(\begin{array}{ccccccc}
z_{1}^{1}\downarrow&z_{2}^{1}\downarrow&\dots&z_{2n-1}^{1}\downarrow&z_{2n}^{1}\downarrow\\
z_{1}^{2}\downarrow&z_{2}^{2}\downarrow&\dots&z_{2n-1}^{2}\downarrow&z_{2n}^{2}\downarrow\\
\vdots&\vdots& &\vdots&\vdots\\
z_{1}^{2n-1}\downarrow&z_{2}^{2n-1}\downarrow&\dots&z_{2n-1}^{2n-1}\downarrow&z_{2n}^{2n-1}\downarrow\\
z_{1}^{2n}\downarrow &z_{2}^{2n}\downarrow&\dots&z_{2n-1}^{2n-1}\downarrow&z_{2n}^{2n}\downarrow\\
\end{array}\right)
\end{equation}
where the ordering start with the upper left element, going down the first column, and then move onto the second column etc.
\\

If we write $\textbf{Z}$ as a block matrix
$$
\textbf{Z}=\left(
\begin{array}{cc}
Z_{(1,1)}& Z_{(1,2)}\\
Z_{(2,1)} & Z_{(2,2)}
\end{array}
\right)
$$
$Z_{i,j}$ being $n\times n$-blocks, then it follows from~\eqref{zaa2} that for each $1\leq k\leq n$ the elements in the $k$'th column of $Z_{(2,2)}$ commute with all elements in $Z_{(1,2)}$ that preceeds them in our ordering on $\textbf{Z}$. For $A=(a_{j,k})_{j,k}\in M_{2n}(\mathbb{Z}_{+}),$ we let, as in~\eqref{fock}, $z(A)$ denote the element
$$
(z_{2n}^{2n})^{a_{n,n}}(z_{2n}^{2n-1})^{a_{n,n-1}}\dots (z_{1}^{2})^{a_{1,2}}(z_{1}^{1})^{a_{1,1}}.
$$
Then for $A\in M_{2n}(\mathbb{Z}_{+}),$ we have $z(A)=z(A')z(A''),$ where $A=A'+A''$ and $A'$ is the matrix with the same integers as $A$ in the lower right $n\times n$ square and zero everywhere else and $A''$ is the same matrix as in $A$ but with zeros in the lower right $n\times n$ square, i.e. if 
$$
A=\left(\begin{array}{cc}A_{11}& A_{12}\\A_{21}&A_{22}\end{array}\right), A_{ij}\in M_{n}(\mathbb{Z}_{+}),
$$ then $A'=\left(\begin{array}{cc}0& 0\\0&A_{22}\end{array}\right)$ and $A''=\left(\begin{array}{cc}A_{11}& A_{12}\\A_{21}&0\end{array}\right).$  Denote these classes of matrices as $M_{2n}^{L}$ and  $M_{2n}^{U}$ respectively.
\\

We shall omit $\pi_{F,2n}$ and write simply $a$ for the image $\pi_{F,2n}(a),$ $a\in  \mathrm{Pol}(\mathrm{Mat}_{2n})_{q}.$ Let $v_{0}$ be a vacuum vector for the representation. We claim that $(z_{k}^{j})^{*}z(A'') v_{0}=0$ for every $z_{k}^{j}\in Z_{(2,2)}$ and $A''\in M_{2n}^{U}$ and hence $z(A'')v_{0}$ is a vacuum vector for the $*$-representation $\pi_{F,2n}\circ \rho.$ To see this, notice that by Proposition~\ref{fockbase} it is enough to show that $$\langle (z_{k}^{j})^{*}z(A'') v_{0}, z(B) v_{0} \rangle=\langle z(A'') v_{0},z_{k}^{j} z(B)  v_{0}\rangle=0$$
i.e. that $z(A'') v_{0}\bot z_{k}^{j} z(B)  v_{0}$
for all $B\in M_{2n}(\mathbb{Z}_{+}).$ Write $z(B)=z(B')z(B'')$ as above. It is not hard to see that the sub-algebra of $\mathbb{C}[\mathrm{Mat}_{2n}]_q\subseteq \mathrm{Pol}(\mathrm{Mat}_{2n})_{q}$ generated by the elements $z_{n+k}^{n+j}$ with $1\leq k,j\leq n$ has a vector space basis given by $\{z(C')|C'\in M_{2n}^{L}\}$ and hence, as $\text{deg}z_{k}^{j}z(B')=|B'|+1,$ we have $z_{k}^{j}z(B')=\sum_{m} c_{m}z(C'_{m})$ for constants $c_{m}\in \mathbb{C}$ and where each $C'_{m}$ has norm $|C_{m}'|=|B'|+1.$ We can then write $$z_{k}^{j} z(B)  v_{0}=z_{k}^{j} z(B') z(B'')   v_{0}=\sum_{m}c_{j}z(C_{m}'+B'') v_{0}.$$ Since $C_{m}'\neq 0$ and hence $A''\neq C_{m}'+B'',$ we get $z(A'') v_{0}\bot z(C_{m}'+B'') v_{0.}$ Thus $z(A'') v_{0}\bot z_{k}^{j} z(B) v_{0}.$
\\

It follows that for every $A''\in M_{2n}^{U},$ we have a $*$-representation equivalent to the Fock representation on the subspace $H_{A''}$ of $H_{F,2n}$ spanned by the orthogonal vectors $$\{z(A')z(A'')v_{0}|A'\in M_{2n}^{L}\}.$$ But as we know that 
$$\{z(A')z(A'')v_{0}|A\in M_{2n}^{L},A''\in M_{2n}^{U}\}=$$
$$\{z(A)v_{0}|A\in M_{2n}(\mathbb{Z}_{+})\}$$
is an orthogonal basis for $H_{F,2n},$ it follows that $H_{F,2n}=\oplus_{A''\in M_{2n}^{U}}H_{A''}$ and that $\mathrm{Pol}(\mathrm{Mat}_{n})_{q}$- $\pi_{F,2n}\circ \rho$ is a direct sum of  $*$-representations equivalent to $\pi_{F,n}.$
\end{proof}
\section{Proof of the Main Result: Existence of Lifting}
In this section we will prove the first part of Theorem~\ref{main}, the existence of the lift~\eqref{liftingthm}. The second part follows from Lemma~\ref{sunandmoon}. The proof is by induction on $n.$ The result is well known in the case of $n=1,$ as in this case it follows from~\eqref{zaa1}-\eqref{zaa4} that $\mathrm{Pol}(\mathrm{Mat}_{n})_{q}\cong Pol(\mathbb{C})_{q},$ the quantum disc (see~\cite{vaksman-book}). This is the unital $*$-algebra over $\mathbb{C}$ generated by a single element $z$ subject to the relation
$$
z^{*}z=q^{2}z z^{*}+(1-q^{2})I.
$$
By~\cite{vaksman-book} (Proposition $1.10$), any irreducible $*$-representations of $Pol(\mathbb{C})_{q},$ up to unitary equivalence, is either 
\begin{itemize}
\item $\pi_{F,1}:Pol(\mathbb{C})_{q}\rightarrow B(\ell^{2}(\mathbb{Z}_{+}))$ determined by $z\mapsto T_{22}$
with $T_{22}$ as in~\eqref{fan}, or
\item $\chi_{\varphi}:Pol(\mathbb{C})_{q}\rightarrow \mathbb{C}$ for $\varphi\in [0,2\pi),$ determined by $z\mapsto e^{i\varphi}.$
\end{itemize}
It is not hard to show that these $*$-representations can be lifted in the way claimed in Theorem~\ref{main}. Indeed, with $\pi_{F,1},$ we let $\phi$ be the $*$-representation of $\mathbb{C}[SU_{2}]_{q}$ determined by $t_{ij}\mapsto T_{ij}.$ Then $\phi\circ \zeta (z)=\phi(t_{22})=T_{22}$ and hence $\phi\circ \zeta=\pi_{F,1}.$
\\

If we let $\phi_{\varphi}:\mathbb{C}[SU_{2}]_{q}\rightarrow \mathbb{C}$ be the $*$-representation of $\mathbb{C}[SU_{2}]_{q}$ determined by 
$$
\left(
\begin{array}{ccc}
t_{11}&  t_{12}\\
t_{21}& t_{22}
\end{array}
\right)
\mapsto
\left(
\begin{array}{ccc}
e^{-i\varphi}& 0\\
0&e^{i\varphi}
\end{array}
\right)
$$
Then $\phi_{\varphi}\circ \zeta(z)=\phi_{\varphi}(t_{22})=e^{i\varphi}$ and therefore $\phi_{\varphi}\circ \zeta =\chi_{\varphi}.$
\\

We also note that the representation theory of $Pol(\mathbb{C})_{q}$ implies the following lemma.
\begin{lem}\label{polclem}
If $X\in B(H)$ satisfies the equation 
\begin{equation}\label{polc}
X^{*}X=q^{2}XX^{*}+(1-q^{2})I
\end{equation}
then $H$ can be written as a direct sum $H=H_{1}\oplus H_{2}$ of subspaces $H_{1},H_{2},$ both reducing $X,$ such that $X|_{H_{1}}$ is isomorphic to a direct sum of the operator $T_{22}$ and $X|_{H_{2}}$ is a unitary isometry. In particular, if $\ker X^{*}=\{0\},$ then $X$ is a unitary isometry. 
\end{lem}

 In the general case, when $n>1,$ the proof is constructed around the fact that every irreducible $*$-representation $\pi$ of $\mathrm{Pol}(\mathrm{Mat}_{n})_{q}$ falls into one of two classes:
\begin{enumerate}
\item[\textbf{A}] There is $1\leq k\leq n$ such that either $\ker \pi(z_{k}^{n})^{*}=\{0\}$ or $\ker \pi(z_{n}^{k})^{*}=\{0\};$
\item[\textbf{B}]  $\ker\pi (z_{n}^{k})^{*}\neq \{0\}$ and $\ker \pi(z_{k}^{n})^{*}\neq \{0\}$ for all $k.$ 
 \end{enumerate}
 The induction will be applied slightly different in these two cases.
 \\
 
Notice that the map $z_{j}^{i}\mapsto z_{i}^{j}$ is a $*$-automorphism of $\mathrm{Pol}(\mathrm{Mat}_{n})_{q}$ and hence in the case $\textbf{A},$ we can, and will, assume that $\ker \pi(z_{k}^{n})^{*}=\{0\}.$
\subsection{Outline of The Proof By Use of Examples}
We shall first outline the structure of the proof in two particular examples representing the two different cases of $\textbf{A}$ and $\textbf{B}$. 
\\

Let $\pi$ be the $*$-representation of $\mathrm{Pol}(\mathrm{Mat}_{3})_{q}$ corresponding to the string $$[(1,\varphi_{3}),(2,0),(1,\varphi_{1})]$$ with the associated square diagram given by
$$
\begin{tikzpicture}[thick,scale=0.5]
\filldraw[fill=black!50!white, draw=black](1,0) -- (0,0) -- (0,1)--(1,1)--(1,0)--(0,0);
\filldraw[fill=black!50!white, draw=black](1,1) -- (0,1) -- (0,2)--(1,2)--(1,1)--(0,1);
\filldraw[fill=black!50!white, draw=black] (1,2) -- (0,2) -- (0,3)--(1,3)--(1,2)--(0,2);
\filldraw[fill=black!20!white, draw=black](2,0) -- (1,0) -- (1,1)--(2,1)--(2,0)--(1,0);
\draw (2,1) -- (1,1) -- (1,2)--(2,2)--(2,1)--(1,1);
\filldraw[fill=black!20!white, draw=black](2,2) -- (1,2) -- (1,3)--(2,3)--(2,2)--(1,2);
\draw (3,0) -- (2,0) -- (2,1)--(3,1)--(3,0)--(2,0);
\draw (3,1) -- (2,1) -- (2,2)--(3,2)--(3,1)--(2,1);
\draw (3,2) -- (2,2) -- (2,3)--(3,3)--(3,2)--(2,2);
\node at (0.5,-0.5) {$1$};
\node at (1.5,-0.5) {$2$};
\node at (2.5,-0.5) {$3$};
\node at (3.5,2.5) {$1$};
\node at (3.5,0.5) {$3$};
\node at (3.5,1.5) {$2$};
\end{tikzpicture}
$$
i.e. $\pi=\Pi\circ \zeta,$ where 
$$
\Pi=(\tau_{0}\otimes \tau_{0}\otimes \tau_{0}\otimes \tau_{\phi_{3}}\otimes \id \otimes \tau_{\phi_{1}}\otimes \id\otimes \id\otimes \id)\circ \pi_{s}.
$$
So, by considering all possible routes from $(3,j)$ to $(i,3)$ constructed out of arrows and hooks, we can recover the action of $\pi(z_{j}^{i}).$ In particular, we have 
$$
\begin{array}{ccc}
\begin{tikzpicture}[thick,scale=0.5]
\filldraw[fill=black!50!white, draw=black](1,0) -- (0,0) -- (0,1)--(1,1)--(1,0)--(0,0);
\filldraw[fill=black!50!white, draw=black](1,1) -- (0,1) -- (0,2)--(1,2)--(1,1)--(0,1);
\filldraw[fill=black!50!white, draw=black] (1,2) -- (0,2) -- (0,3)--(1,3)--(1,2)--(0,2);
\filldraw[fill=black!20!white, draw=black](2,0) -- (1,0) -- (1,1)--(2,1)--(2,0)--(1,0);
\draw (2,1) -- (1,1) -- (1,2)--(2,2)--(2,1)--(1,1);
\filldraw[fill=black!20!white, draw=black](2,2) -- (1,2) -- (1,3)--(2,3)--(2,2)--(1,2);
\draw (3,0) -- (2,0) -- (2,1)--(3,1)--(3,0)--(2,0);
\draw (3,1) -- (2,1) -- (2,2)--(3,2)--(3,1)--(2,1);
\draw (3,2) -- (2,2) -- (2,3)--(3,3)--(3,2)--(2,2);
\node at (0.5,-0.5) {$1$};
\node at (1.5,-0.5) {$2$};
\node at (2.5,-0.5) {$3$};
\node at (3.5,2.5) {$1$};
\node at (3.5,0.5) {$3$};
\node at (3.5,1.5) {$2$};
\draw [->] (0.5,0.05)--(0.5,0.5)--(0.95,0.5);\draw [->] (0.05+1,0.5)--(0.95+1,0.5);\draw [->] (0.05+2,0.5)--(0.95+2,0.5);
\node at (-2.9,1.5) {$\pi(z_{1}^{3})=(-q)^{-2}\times$};
\end{tikzpicture}
&
\begin{tikzpicture}[thick,scale=0.5]
\filldraw[fill=black!50!white, draw=black](1,0) -- (0,0) -- (0,1)--(1,1)--(1,0)--(0,0);
\filldraw[fill=black!50!white, draw=black](1,1) -- (0,1) -- (0,2)--(1,2)--(1,1)--(0,1);
\filldraw[fill=black!50!white, draw=black] (1,2) -- (0,2) -- (0,3)--(1,3)--(1,2)--(0,2);
\filldraw[fill=black!20!white, draw=black](2,0) -- (1,0) -- (1,1)--(2,1)--(2,0)--(1,0);
\draw (2,1) -- (1,1) -- (1,2)--(2,2)--(2,1)--(1,1);
\filldraw[fill=black!20!white, draw=black](2,2) -- (1,2) -- (1,3)--(2,3)--(2,2)--(1,2);
\draw (3,0) -- (2,0) -- (2,1)--(3,1)--(3,0)--(2,0);
\draw (3,1) -- (2,1) -- (2,2)--(3,2)--(3,1)--(2,1);
\draw (3,2) -- (2,2) -- (2,3)--(3,3)--(3,2)--(2,2);
\node at (0.5,-0.5) {$1$};
\node at (1.5,-0.5) {$2$};
\node at (2.5,-0.5) {$3$};
\node at (3.5,2.5) {$1$};
\node at (3.5,0.5) {$3$};
\node at (3.5,1.5) {$2$};
\draw [->] (0.5+1,0.05)--(0.5+1,0.5)--(0.95+1,0.5);\draw [->] (0.05+2,0.5)--(0.95+2,0.5);
\node at (-2.9,1.5) {$\pi(z_{2}^{3})=(-q)^{-1}\times$};
\end{tikzpicture}
&
\begin{tikzpicture}[thick,scale=0.5]
\filldraw[fill=black!50!white, draw=black](1,0) -- (0,0) -- (0,1)--(1,1)--(1,0)--(0,0);
\filldraw[fill=black!50!white, draw=black](1,1) -- (0,1) -- (0,2)--(1,2)--(1,1)--(0,1);
\filldraw[fill=black!50!white, draw=black] (1,2) -- (0,2) -- (0,3)--(1,3)--(1,2)--(0,2);
\filldraw[fill=black!20!white, draw=black](2,0) -- (1,0) -- (1,1)--(2,1)--(2,0)--(1,0);
\draw (2,1) -- (1,1) -- (1,2)--(2,2)--(2,1)--(1,1);
\filldraw[fill=black!20!white, draw=black](2,2) -- (1,2) -- (1,3)--(2,3)--(2,2)--(1,2);
\draw (3,0) -- (2,0) -- (2,1)--(3,1)--(3,0)--(2,0);
\draw (3,1) -- (2,1) -- (2,2)--(3,2)--(3,1)--(2,1);
\draw (3,2) -- (2,2) -- (2,3)--(3,3)--(3,2)--(2,2);
\node at (0.5,-0.5) {$1$};
\node at (1.5,-0.5) {$2$};
\node at (2.5,-0.5) {$3$};
\node at (3.5,2.5) {$1$};
\node at (3.5,0.5) {$3$};
\node at (3.5,1.5) {$2$};
\draw [->] (0.5+2,0.05)--(0.5+2,0.5)--(0.95+2,0.5);
\node at (-1.5,1.5) {$\pi(z_{3}^{3})=$};
\end{tikzpicture}
\end{array}
$$
However, observe that if the arrow $\begin{tikzpicture}[line width=0.21mm,scale=0.4]
\draw (1,0) -- (0,0) -- (0,1)--(1,1)--(1,0)--(0,0);
\draw [->] (0.1,0.5)--(0.9,0.5);
\end{tikzpicture}$ or $\begin{tikzpicture}[line width=0.21mm,scale=0.4]
\draw (1,0) -- (0,0) -- (0,1)--(1,1)--(1,0)--(0,0);
\draw [->] (0.5,0.1)--(0.5,0.9);
\end{tikzpicture}$ happen to fall in a shaded box, then the corresponding operator is zero. This follows from the fact that $\begin{tikzpicture}[line width=0.21mm,scale=0.4]
\draw (1,0) -- (0,0) -- (0,1)--(1,1)--(1,0)--(0,0);
\draw [->] (0.1,0.5)--(0.9,0.5);
\end{tikzpicture}$ and $\begin{tikzpicture}[line width=0.21mm,scale=0.4]
\draw (1,0) -- (0,0) -- (0,1)--(1,1)--(1,0)--(0,0);
\draw [->] (0.5,0.1)--(0.5,0.9);
\end{tikzpicture}$ corresponds to the compact operators $T_{12}$ and $T_{21}$ respectively, and hence annihilated by each $*$-representation $\tau_{\varphi}.$ In particular, we have $\pi(z_{1}^{3})=0.$ 
\\

As $\pi(z_{2}^{3})=e^{i\phi}I\otimes T_{12}\otimes I\otimes I,$ we see that $\ker \pi(z_{2}^{3})^{*}=\{0\},$ and this shows that $\pi$ falls into the class $(A).$
\\

Let $\pi'$ be the $*$-representation of $\mathrm{Pol}(\mathrm{Mat}_{2})_{q}$ corresponding to the upper right $2\times 2$ sub-square

$$
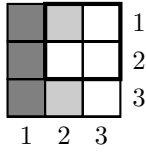
\begin{tikzpicture}[thick,scale=0.5]
\filldraw[fill=black!50!white, draw=black](1,0) -- (0,0) -- (0,1)--(1,1)--(1,0)--(0,0);
\filldraw[fill=black!50!white, draw=black](1,1) -- (0,1) -- (0,2)--(1,2)--(1,1)--(0,1);
\filldraw[fill=black!50!white, draw=black] (1,2) -- (0,2) -- (0,3)--(1,3)--(1,2)--(0,2);
\filldraw[fill=black!20!white, draw=black](2,0) -- (1,0) -- (1,1)--(2,1)--(2,0)--(1,0);
\draw (2,1) -- (1,1) -- (1,2)--(2,2)--(2,1)--(1,1);
\filldraw[fill=black!20!white, draw=black](2,2) -- (1,2) -- (1,3)--(2,3)--(2,2)--(1,2);
\draw (3,0) -- (2,0) -- (2,1)--(3,1)--(3,0)--(2,0);
\draw (3,1) -- (2,1) -- (2,2)--(3,2)--(3,1)--(2,1);
\draw (3,2) -- (2,2) -- (2,3)--(3,3)--(3,2)--(2,2);
\draw[black,ultra thick]  (1,1) -- (1,3)--(3,3)--(3,1)--(1,1);
\node at (0.5,-0.5) {$1$};
\node at (1.5,-0.5) {$2$};
\node at (2.5,-0.5) {$3$};
\node at (3.5,2.5) {$1$};
\node at (3.5,0.5) {$3$};
\node at (3.5,1.5) {$2$};
\end{tikzpicture}
$$
The string associated to $\pi'$ is $[(2,0),(1,\phi_{1})].$ Next we will relate $\pi'$ to the $*$-representation $\pi.$ To do this, consider all possible routes from $(3,i)$ to $(j,3)$ such that $i\neq 2$ and $j\neq 3.$ Using the above observations about straight arrows on dark squares, we can easily see that the only routes corresponding to non-zero operators are those having in the last row either the sub-path 
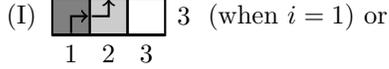
\begin{enumerate}[(I)]
\item
$\begin{tikzpicture}[thick,scale=0.5,baseline=1ex]
\filldraw[fill=black!50!white, draw=black](1,0) -- (0,0) -- (0,1)--(1,1)--(1,0)--(0,0);
\filldraw[fill=black!20!white, draw=black](2,0) -- (1,0) -- (1,1)--(2,1)--(2,0)--(1,0);
\draw (3,0) -- (2,0) -- (2,1)--(3,1)--(3,0)--(2,0);
\draw [->] (0.5,0.05)--(0.5,0.5)--(0.95,0.5);
\draw [->] (0.05+1,0.5)--(0.5+1,0.5)--(0.5+1,0.95);
\node at (0.5,-0.5) {$1$};
\node at (1.5,-0.5) {$2$};
\node at (2.5,-0.5) {$3$};
\node at (3.5,0.5) {$3$};
\end{tikzpicture}$
 (when $i=1$) or
\item $
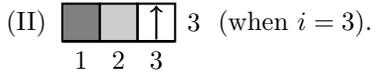
\begin{tikzpicture}[thick,scale=0.5,baseline=1ex]
\filldraw[fill=black!50!white, draw=black](1,0) -- (0,0) -- (0,1)--(1,1)--(1,0)--(0,0);

\filldraw[fill=black!20!white, draw=black](2,0) -- (1,0) -- (1,1)--(2,1)--(2,0)--(1,0);

\draw (3,0) -- (2,0) -- (2,1)--(3,1)--(3,0)--(2,0);

\draw [->] (1/2+2,0.05)--(1/2+2,0.95);
\node at (0.5,-0.5) {$1$};
\node at (1.5,-0.5) {$2$};
\node at (2.5,-0.5) {$3$};
\node at (3.5,0.5) {$3$};
\end{tikzpicture}$ (when $i=3$).
\end{enumerate}
Hence all routes corresponding to non-zero summands in $\pi(z_{1}^{j})$ (resp $\pi(z_{3}^{j})$) can be obtained by attaching to sub-path (I) (resp (II)) a path from the positions $(2,1)$ to $(j,2)$ (resp $(2,2)$ to $(j,2)$) in the upper right sub-square
$$
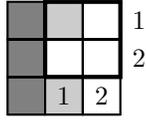
\begin{tikzpicture}[thick,scale=0.5]
\filldraw[fill=black!50!white, draw=black](1,0) -- (0,0) -- (0,1)--(1,1)--(1,0)--(0,0);
\filldraw[fill=black!50!white, draw=black](1,1) -- (0,1) -- (0,2)--(1,2)--(1,1)--(0,1);
\filldraw[fill=black!50!white, draw=black] (1,2) -- (0,2) -- (0,3)--(1,3)--(1,2)--(0,2);
\filldraw[fill=black!20!white, draw=black](2,0) -- (1,0) -- (1,1)--(2,1)--(2,0)--(1,0);
\draw (2,1) -- (1,1) -- (1,2)--(2,2)--(2,1)--(1,1);
\filldraw[fill=black!20!white, draw=black](2,2) -- (1,2) -- (1,3)--(2,3)--(2,2)--(1,2);
\draw (3,0) -- (2,0) -- (2,1)--(3,1)--(3,0)--(2,0);
\draw (3,1) -- (2,1) -- (2,2)--(3,2)--(3,1)--(2,1);
\draw (3,2) -- (2,2) -- (2,3)--(3,3)--(3,2)--(2,2);
\draw[black,ultra thick]  (1,1) -- (1,3)--(3,3)--(3,1)--(1,1);
\node at (1.5,0.5) {$1$};
\node at (2.5,0.5) {$2$};

\node at (3.5,2.5) {$1$};

\node at (3.5,1.5) {$2$};
\end{tikzpicture}
$$
Next we observe that the operator $$
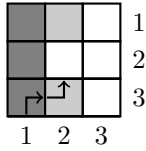
\begin{tikzpicture}[thick,scale=0.5]
\filldraw[fill=black!50!white, draw=black](1,0) -- (0,0) -- (0,1)--(1,1)--(1,0)--(0,0);
\filldraw[fill=black!50!white, draw=black](1,1) -- (0,1) -- (0,2)--(1,2)--(1,1)--(0,1);
\filldraw[fill=black!50!white, draw=black] (1,2) -- (0,2) -- (0,3)--(1,3)--(1,2)--(0,2);
\filldraw[fill=black!20!white, draw=black](2,0) -- (1,0) -- (1,1)--(2,1)--(2,0)--(1,0);
\draw (2,1) -- (1,1) -- (1,2)--(2,2)--(2,1)--(1,1);
\filldraw[fill=black!20!white, draw=black](2,2) -- (1,2) -- (1,3)--(2,3)--(2,2)--(1,2);
\draw (3,0) -- (2,0) -- (2,1)--(3,1)--(3,0)--(2,0);
\draw (3,1) -- (2,1) -- (2,2)--(3,2)--(3,1)--(2,1);
\draw (3,2) -- (2,2) -- (2,3)--(3,3)--(3,2)--(2,2);
\draw [->] (0.5,0.05)--(0.5,0.5)--(0.95,0.5);
\draw [->] (0.05+1,0.5)--(0.5+1,0.5)--(0.5+1,0.95);
\node at (0.5,-0.5) {$1$};
\node at (1.5,-0.5) {$2$};
\node at (2.5,-0.5) {$3$};
\node at (3.5,2.5) {$1$};
\node at (3.5,0.5) {$3$};
\node at (3.5,1.5) {$2$};
\end{tikzpicture}$$ is $e^{-\phi_{3}}I,$ while the operator $$
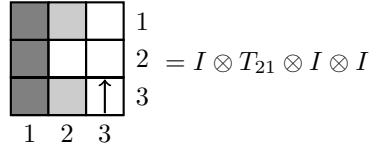
\begin{tikzpicture}[thick,scale=0.5,baseline=4ex]
\filldraw[fill=black!50!white, draw=black](1,0) -- (0,0) -- (0,1)--(1,1)--(1,0)--(0,0);
\filldraw[fill=black!50!white, draw=black](1,1) -- (0,1) -- (0,2)--(1,2)--(1,1)--(0,1);
\filldraw[fill=black!50!white, draw=black] (1,2) -- (0,2) -- (0,3)--(1,3)--(1,2)--(0,2);
\filldraw[fill=black!20!white, draw=black](2,0) -- (1,0) -- (1,1)--(2,1)--(2,0)--(1,0);
\draw (2,1) -- (1,1) -- (1,2)--(2,2)--(2,1)--(1,1);
\filldraw[fill=black!20!white, draw=black](2,2) -- (1,2) -- (1,3)--(2,3)--(2,2)--(1,2);
\draw (3,0) -- (2,0) -- (2,1)--(3,1)--(3,0)--(2,0);
\draw (3,1) -- (2,1) -- (2,2)--(3,2)--(3,1)--(2,1);
\draw (3,2) -- (2,2) -- (2,3)--(3,3)--(3,2)--(2,2);
\draw [->] (1/2+2,0.05)--(1/2+2,0.95);
\node at (0.5,-0.5) {$1$};
\node at (1.5,-0.5) {$2$};
\node at (2.5,-0.5) {$3$};
\node at (3.5,2.5) {$1$};
\node at (3.5,0.5) {$3$};
\node at (3.5,1.5) {$2$};
\end{tikzpicture}=I\otimes T_{21}\otimes I\otimes I$$ is a multiple of the identity (actually the identity) after restricting it to the subspace 
$
H:=\ker \pi(z_{2}^{3})^{*}=\spann \{e_{k}\otimes e_{0}\otimes e_{m}\otimes e_{l};(k,m,l)\in \mathbb{Z}_{+}^{3}\}\cong \ell^{2}(\mathbb{Z}_{+})\otimes \langle e_{0}\rangle \otimes \ell^{2}(\mathbb{Z}_{+})\otimes \ell^{2}(\mathbb{Z}_{+}).$  Therefore, we have
$$
\pi(z_{1}^{j})|_{H}=(-q)^{1-3}e^{-i\varphi_{3}}((-q)^{-(1-2)}\pi'(z_{1}^{j}))=(-q)^{-1}e^{-i\varphi_{3}}\pi'(z_{1}^{j})
$$
and
$$
\pi(z_{3}^{j})|_{H}=(-q)^{3-3}((-q)^{2-2}\pi'(z_{2}^{j}))=\pi'(z_{2}^{j}).
$$
Moreover, it is easy to see that $\pi(z_{2}^{j})|_{H}=0$ for $j\neq 3$ and $\pi(z_{2}^{3})|_{H}=e^{i\varphi_{3}}I.$ In general, being in the case $\textbf{A}$ with an irreducible $*$-representation $\pi$ of $\mathrm{Pol}(\mathrm{Mat}_{n})_{q},$ such that $\ker\pi(z_{k}^{n})^{*}=\{0\},$ we shall prove that the subspace 
$$
H:=\cap_{i=k+1}^{n}\ker \pi(z_{i}^{n})^{*}
$$
is invariant with respect to all the operators $\pi(z_{i}^{j})$ and $\pi(z_{k}^{n})=e^{i\varphi}I$ for some $\varphi\in [0,2\pi).$ Furthermore, we can define an irreducible $*$-representation $\pi':\mathrm{Pol}(\mathrm{Mat}_{n-1})_{q}\to B(H)$ by the formula
$$
\pi'(z_{i}^{j})=
\begin{cases}
-e^{-i\varphi}q\pi(z_{i}^{j})|_{H} & \text{ for $1\leq i <k$ and $1\leq j\leq n-1$}\\
\pi(z_{i+1}^{j}) & \text{ for $k\leq i\leq n-1$ and $1\leq j\leq n-1$}.
\end{cases}
$$
Returning to the example, we can by induction lift the $*$-representation $\pi'$ of $\mathrm{Pol}(\mathrm{Mat}_{2})_{q}$ to a $*$-representation of $\Pi':\mathbb{C}[SU_{4}]_{q}\to B(H).$ If we let $\Psi:\mathbb{C}[SU_{6}]_{q}\to \mathbb{C}[SU_{4}]_{q}$ be the $*$-homomorphism determined by 
$$\Psi(t_{kj})=\begin{cases}t_{k-1,j-1}&\text{ if $2\leq k,j\leq 5$}\\ \delta_{k,j}I& \text{ otherwise}\end{cases}$$ 
then $\Pi'\circ \Psi \circ \zeta$ corresponds to the grid
\begin{equation}\label{lifted}
\begin{tikzpicture}[thick,scale=0.5,baseline=4ex]
\filldraw[fill=black!50!white, draw=black](1,0) -- (0,0) -- (0,1)--(1,1)--(1,0)--(0,0);
\filldraw[fill=black!50!white, draw=black](1,1) -- (0,1) -- (0,2)--(1,2)--(1,1)--(0,1);
\filldraw[fill=black!50!white, draw=black] (1,2) -- (0,2) -- (0,3)--(1,3)--(1,2)--(0,2);
\filldraw[fill=black!50!white, draw=black](2,0) -- (1,0) -- (1,1)--(2,1)--(2,0)--(1,0);
\draw (2,1) -- (1,1) -- (1,2)--(2,2)--(2,1)--(1,1);
\filldraw[fill=black!20!white, draw=black](2,2) -- (1,2) -- (1,3)--(2,3)--(2,2)--(1,2);
\filldraw[fill=black!50!white, draw=black] (3,0) -- (2,0) -- (2,1)--(3,1)--(3,0)--(2,0);
\draw (3,1) -- (2,1) -- (2,2)--(3,2)--(3,1)--(2,1);
\draw (3,2) -- (2,2) -- (2,3)--(3,3)--(3,2)--(2,2);
\node at (0.5,-0.5) {$1$};
\node at (1.5,-0.5) {$2$};
\node at (2.5,-0.5) {$3$};
\node at (3.5,2.5) {$1$};
\node at (3.5,0.5) {$3$};
\node at (3.5,1.5) {$2$};
\end{tikzpicture}
\end{equation}
In fact, $$\Pi'\circ\Psi\circ\zeta=(\tau_{0}\otimes\tau_{0}\otimes \tau_{0}\otimes \tau_{0}\otimes \id\otimes\tau_{\varphi_{1}}\otimes \tau_{0}\otimes \id\otimes \id)\circ\pi_{s}$$
%$$
%(\tau_{0}\otimes\tau_{0}\otimes \tau_{0}\otimes \tau_{0}\otimes \id\otimes\tau_{\varphi_{1}}\otimes \tau_{0}\otimes \id\otimes \id)\circ(\pi_{3}\otimes\pi_{2}\otimes \pi_{1}\otimes \pi_{4}\otimes \pi_{3}\otimes \pi_{2}\otimes \pi_{5}\otimes \pi_{4}\otimes \pi_{3})=
%$$

 The replacement of the last row on~\eqref{lifted} by  $\begin{tikzpicture}[thick,scale=0.5,baseline=1ex]
\filldraw[fill=black!50!white, draw=black](1,0) -- (0,0) -- (0,1)--(1,1)--(1,0)--(0,0);

\filldraw[fill=black!20!white, draw=black](2,0) -- (1,0) -- (1,1)--(2,1)--(2,0)--(1,0);

\draw (3,0) -- (2,0) -- (2,1)--(3,1)--(3,0)--(2,0);
\node at (0.5,-0.5) {$1$};
\node at (1.5,-0.5) {$2$};
\node at (2.5,-0.5) {$3$};
\node at (3.5,0.5) {$3$};
\end{tikzpicture}$ will correspond to the tensoring the $*$-representation $\Pi'\circ\Psi$ of $\mathbb{C}[SU_{6}]_{q}$ by a suitable $*$-representation $\lambda$ of $\mathbb{C}[SU_{6}]_{q}$ and taking the composition $(\lambda\otimes (\Pi'\circ \Psi))\circ \zeta,$ i.e. $\pi\cong (\lambda\otimes (\Pi'\circ \Psi))\circ \zeta.$
More precisely, $\lambda$ is given by $(\tau_{0}\otimes\tau_{-\varphi}\otimes \id)\circ(\pi_{3}\otimes \pi_{4}\otimes \pi_{5})=
(\tau_{-\varphi}\otimes\id)\circ (\pi_{4}\otimes \pi_{5}).$
\\

Assume now that $\pi$ is a $*$-representation of $\mathrm{Pol}(\mathrm{Mat}_{3})_{q}$ such that $\ker\pi(z_{k}^{3})\neq \{0\}$ and $\ker \pi(z_{3}^{k})^{*}\neq \{0\}$ for any $k=1,2,3.$ If $\pi$ is the $*$-representation corresponding to a string $[(k_{3},\varphi_{3}),(k_{2},\varphi_{2}),(k_{1},\varphi_{1})],$ then the kernel condition is satisfied if and only if the last column and last row of the corresponding grid consist of white boxes. A typical example of such $*$-representation is the one given by the string $[(3,0),(2,\phi),(2,0)]$ with the following grid
$$
\begin{tikzpicture}[thick,scale=0.5]
\draw(1,0) -- (0,0) -- (0,1)--(1,1)--(1,0)--(0,0);
\filldraw[fill=black!20!white, draw=black](1,1) -- (0,1) -- (0,2)--(1,2)--(1,1)--(0,1);
\filldraw[fill=black!50!white, draw=black] (1,2) -- (0,2) -- (0,3)--(1,3)--(1,2)--(0,2);
\draw(2,0) -- (1,0) -- (1,1)--(2,1)--(2,0)--(1,0);
\draw (2,1) -- (1,1) -- (1,2)--(2,2)--(2,1)--(1,1);
\draw(2,2) -- (1,2) -- (1,3)--(2,3)--(2,2)--(1,2);
\draw (3,0) -- (2,0) -- (2,1)--(3,1)--(3,0)--(2,0);
\draw (3,1) -- (2,1) -- (2,2)--(3,2)--(3,1)--(2,1);
\draw (3,2) -- (2,2) -- (2,3)--(3,3)--(3,2)--(2,2);
\node at (0.5,-0.5) {$1$};
\node at (1.5,-0.5) {$2$};
\node at (2.5,-0.5) {$3$};
\node at (3.5,2.5) {$1$};
\node at (3.5,0.5) {$3$};
\node at (3.5,1.5) {$2$};
\end{tikzpicture}
$$
Here the approach from the case $\textbf{A}$ fails. Instead we shall construct a $*$-representation of $\mathrm{Pol}(\mathrm{Mat}_{2})_{q}$ corresponding to the left upper $2\times 2$ sub-grid

$$
\begin{tikzpicture}[thick,scale=0.5]
\draw(1,0) -- (0,0) -- (0,1)--(1,1)--(1,0)--(0,0);
\filldraw[fill=black!20!white, draw=black](1,1) -- (0,1) -- (0,2)--(1,2)--(1,1)--(0,1);
\filldraw[fill=black!50!white, draw=black] (1,2) -- (0,2) -- (0,3)--(1,3)--(1,2)--(0,2);
\draw(2,0) -- (1,0) -- (1,1)--(2,1)--(2,0)--(1,0);
\draw (2,1) -- (1,1) -- (1,2)--(2,2)--(2,1)--(1,1);
\draw(2,2) -- (1,2) -- (1,3)--(2,3)--(2,2)--(1,2);
\draw (3,0) -- (2,0) -- (2,1)--(3,1)--(3,0)--(2,0);
\draw (3,1) -- (2,1) -- (2,2)--(3,2)--(3,1)--(2,1);
\draw (3,2) -- (2,2) -- (2,3)--(3,3)--(3,2)--(2,2);
\node at (0.5,-0.5) {$1$};
\node at (1.5,-0.5) {$2$};
\node at (2.5,-0.5) {$3$};
\node at (3.5,2.5) {$1$};
\node at (3.5,0.5) {$3$};
\node at (3.5,1.5) {$2$};
\draw[black,ultra thick]  (0,1) -- (0,3)--(2,3)--(2,1)--(0,1);
\end{tikzpicture}
$$
and to show that it can be obtained out of $\pi$ by restricting the latter to an appropriate subspace, namely 
$$
H:=\left(\cap_{k=1}^{3}\ker \pi(z_{k}^{3})^{*}\right)\cap\left(\cap_{k=1}^{3}\ker\pi(z_{3}^{k})^{*}\right).
$$ 
We have 
$$
\begin{array}{ccc}
\pi(z_{1}^{3})=(-q)^{-2}\times\begin{tikzpicture}[thick,scale=0.5,baseline=4ex] 
\draw(1,0) -- (0,0) -- (0,1)--(1,1)--(1,0)--(0,0);
\filldraw[fill=black!20!white, draw=black](1,1) -- (0,1) -- (0,2)--(1,2)--(1,1)--(0,1);
\filldraw[fill=black!50!white, draw=black] (1,2) -- (0,2) -- (0,3)--(1,3)--(1,2)--(0,2);
\draw(2,0) -- (1,0) -- (1,1)--(2,1)--(2,0)--(1,0);
\draw (2,1) -- (1,1) -- (1,2)--(2,2)--(2,1)--(1,1);
\draw(2,2) -- (1,2) -- (1,3)--(2,3)--(2,2)--(1,2);
\draw (3,0) -- (2,0) -- (2,1)--(3,1)--(3,0)--(2,0);
\draw (3,1) -- (2,1) -- (2,2)--(3,2)--(3,1)--(2,1);
\draw (3,2) -- (2,2) -- (2,3)--(3,3)--(3,2)--(2,2);
\node at (0.5,-0.5) {$1$};
\node at (1.5,-0.5) {$2$};
\node at (2.5,-0.5) {$3$};
\node at (3.5,2.5) {$1$};
\node at (3.5,0.5) {$3$};
\node at (3.5,1.5) {$2$};
\draw [->] (0.5,0.05)--(0.5,0.5)--(0.95,0.5);\draw [->] (0.05+1,0.5)--(0.95+1,0.5);\draw [->] (0.05+2,0.5)--(0.95+2,0.5);
\end{tikzpicture}
& 
\pi(z_{2}^{3})=(-q)^{-1}\times\begin{tikzpicture}[thick,scale=0.5,baseline=4ex] 
\draw(1,0) -- (0,0) -- (0,1)--(1,1)--(1,0)--(0,0);
\filldraw[fill=black!20!white, draw=black](1,1) -- (0,1) -- (0,2)--(1,2)--(1,1)--(0,1);
\filldraw[fill=black!50!white, draw=black] (1,2) -- (0,2) -- (0,3)--(1,3)--(1,2)--(0,2);
\draw(2,0) -- (1,0) -- (1,1)--(2,1)--(2,0)--(1,0);
\draw (2,1) -- (1,1) -- (1,2)--(2,2)--(2,1)--(1,1);
\draw(2,2) -- (1,2) -- (1,3)--(2,3)--(2,2)--(1,2);
\draw (3,0) -- (2,0) -- (2,1)--(3,1)--(3,0)--(2,0);
\draw (3,1) -- (2,1) -- (2,2)--(3,2)--(3,1)--(2,1);
\draw (3,2) -- (2,2) -- (2,3)--(3,3)--(3,2)--(2,2);
\node at (0.5,-0.5) {$1$};
\node at (1.5,-0.5) {$2$};
\node at (2.5,-0.5) {$3$};
\node at (3.5,2.5) {$1$};
\node at (3.5,0.5) {$3$};
\node at (3.5,1.5) {$2$};
\draw [->] (0.5+1,0.05)--(0.5+1,0.5)--(0.95+1,0.5);\draw [->] (0.05+2,0.5)--(0.95+2,0.5);
\end{tikzpicture}\\
\pi(z_{3}^{3})=\begin{tikzpicture}[thick,scale=0.5,baseline=4ex] 
\draw(1,0) -- (0,0) -- (0,1)--(1,1)--(1,0)--(0,0);
\filldraw[fill=black!20!white, draw=black](1,1) -- (0,1) -- (0,2)--(1,2)--(1,1)--(0,1);
\filldraw[fill=black!50!white, draw=black] (1,2) -- (0,2) -- (0,3)--(1,3)--(1,2)--(0,2);
\draw(2,0) -- (1,0) -- (1,1)--(2,1)--(2,0)--(1,0);
\draw (2,1) -- (1,1) -- (1,2)--(2,2)--(2,1)--(1,1);
\draw(2,2) -- (1,2) -- (1,3)--(2,3)--(2,2)--(1,2);
\draw (3,0) -- (2,0) -- (2,1)--(3,1)--(3,0)--(2,0);
\draw (3,1) -- (2,1) -- (2,2)--(3,2)--(3,1)--(2,1);
\draw (3,2) -- (2,2) -- (2,3)--(3,3)--(3,2)--(2,2);
\node at (0.5,-0.5) {$1$};
\node at (1.5,-0.5) {$2$};
\node at (2.5,-0.5) {$3$};
\node at (3.5,2.5) {$1$};
\node at (3.5,0.5) {$3$};
\node at (3.5,1.5) {$2$};
\draw [->] (0.5+2,0.05)--(0.5+2,0.5)--(0.95+2,0.5);
\end{tikzpicture}
& 
\pi(z_{3}^{1})=\begin{tikzpicture}[thick,scale=0.5,baseline=4ex] 
\draw(1,0) -- (0,0) -- (0,1)--(1,1)--(1,0)--(0,0);
\filldraw[fill=black!20!white, draw=black](1,1) -- (0,1) -- (0,2)--(1,2)--(1,1)--(0,1);
\filldraw[fill=black!50!white, draw=black] (1,2) -- (0,2) -- (0,3)--(1,3)--(1,2)--(0,2);
\draw(2,0) -- (1,0) -- (1,1)--(2,1)--(2,0)--(1,0);
\draw (2,1) -- (1,1) -- (1,2)--(2,2)--(2,1)--(1,1);
\draw(2,2) -- (1,2) -- (1,3)--(2,3)--(2,2)--(1,2);
\draw (3,0) -- (2,0) -- (2,1)--(3,1)--(3,0)--(2,0);
\draw (3,1) -- (2,1) -- (2,2)--(3,2)--(3,1)--(2,1);
\draw (3,2) -- (2,2) -- (2,3)--(3,3)--(3,2)--(2,2);
\node at (0.5,-0.5) {$1$};
\node at (1.5,-0.5) {$2$};
\node at (2.5,-0.5) {$3$};
\node at (3.5,2.5) {$1$};
\node at (3.5,0.5) {$3$};
\node at (3.5,1.5) {$2$};
\draw [->] (1/2+2,0.1)--(1/2+2,0.9);\draw [->] (1/2+2,0.1+1)--(1/2+2,0.9+1);\draw [->] (0.5+2,0.1+2)--(0.5+2,0.5+2)--(0.9+2,0.5+2);
\end{tikzpicture}
\end{array}
$$
$$
\pi(z_{3}^{2})=\begin{tikzpicture}[thick,scale=0.5,baseline=4ex] 
\draw(1,0) -- (0,0) -- (0,1)--(1,1)--(1,0)--(0,0);
\filldraw[fill=black!20!white, draw=black](1,1) -- (0,1) -- (0,2)--(1,2)--(1,1)--(0,1);
\filldraw[fill=black!50!white, draw=black] (1,2) -- (0,2) -- (0,3)--(1,3)--(1,2)--(0,2);
\draw(2,0) -- (1,0) -- (1,1)--(2,1)--(2,0)--(1,0);
\draw (2,1) -- (1,1) -- (1,2)--(2,2)--(2,1)--(1,1);
\draw(2,2) -- (1,2) -- (1,3)--(2,3)--(2,2)--(1,2);
\draw (3,0) -- (2,0) -- (2,1)--(3,1)--(3,0)--(2,0);
\draw (3,1) -- (2,1) -- (2,2)--(3,2)--(3,1)--(2,1);
\draw (3,2) -- (2,2) -- (2,3)--(3,3)--(3,2)--(2,2);
\node at (0.5,-0.5) {$1$};
\node at (1.5,-0.5) {$2$};
\node at (2.5,-0.5) {$3$};
\node at (3.5,2.5) {$1$};
\node at (3.5,0.5) {$3$};
\node at (3.5,1.5) {$2$};
\draw [->] (0.5+2,0.1+1)--(0.5+2,0.5+1)--(0.9+2,0.5+1);\draw [->] (1/2+2,0.1)--(1/2+2,0.9);
\end{tikzpicture}.
$$
Using the fact that the kernel of the adjoint of $T_{22}=C_{q}S$ (recall that $T_{22}$ corresponds to the $\begin{tikzpicture}[thick,scale=0.4,baseline=0.4ex]\draw (1,0) -- (0,0) -- (0,1)--(1,1)--(1,0)--(0,0);
\draw [->] (0.5,0.1)--(0.5,0.5)--(0.9,0.5); \end{tikzpicture}$ arrow) is $\mathbb{C}\langle e_{0}\rangle$ and that the kernels of the operators corresponding to the arrows $\begin{tikzpicture}[thick,scale=0.4,baseline=0.4ex]\draw (1,0) -- (0,0) -- (0,1)--(1,1)--(1,0)--(0,0);
\draw [->] (0.1,0.5)--(0.9,0.5); \end{tikzpicture}$ and $\begin{tikzpicture}[thick,scale=0.4,baseline=0.4ex]\draw (1,0) -- (0,0) -- (0,1)--(1,1)--(1,0)--(0,0);
\draw [->] (1/2,0.1)--(1/2,0.9); \end{tikzpicture}$ are zero, we can easily see that $$H=\spann\{e_{0}\otimes e_{0}\otimes e_{k}\otimes e_{j}\otimes e_{0}\otimes e_{0}\otimes e_{0};(k,j)\in \mathbb{Z}_{+}^{2}\}\cong
$$
$$
\mathbb{C}\langle e_{0}\rangle\otimes \mathbb{C}\langle e_{0}\rangle\otimes \ell^{2}(\mathbb{Z}_{+})\otimes \ell^{2}(\mathbb{Z}_{+})\otimes \mathbb{C}\langle e_{0}\rangle\otimes \mathbb{C}\langle e_{0}\rangle\otimes \mathbb{C}\langle e_{0}\rangle.
$$
By looking at paths, it is not hard to see that $H$ is invariant with respect to $\pi(z_{j}^{k})$ for $k,j\neq 3$ (it can also be seen directly from the relations defining $\mathrm{Pol}(\mathrm{Mat}_{3})_{q}$). Consider now the routes representing the terms in $\pi(z_{1}^{1}).$ They will either start and end at the arrows of the diagram 
$$
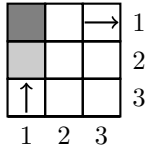
\begin{tikzpicture}[thick,scale=0.5]
\draw(1,0) -- (0,0) -- (0,1)--(1,1)--(1,0)--(0,0);
\filldraw[fill=black!20!white, draw=black](1,1) -- (0,1) -- (0,2)--(1,2)--(1,1)--(0,1);
\filldraw[fill=black!50!white, draw=black] (1,2) -- (0,2) -- (0,3)--(1,3)--(1,2)--(0,2);
\draw(2,0) -- (1,0) -- (1,1)--(2,1)--(2,0)--(1,0);
\draw (2,1) -- (1,1) -- (1,2)--(2,2)--(2,1)--(1,1);
\draw(2,2) -- (1,2) -- (1,3)--(2,3)--(2,2)--(1,2);
\draw (3,0) -- (2,0) -- (2,1)--(3,1)--(3,0)--(2,0);
\draw (3,1) -- (2,1) -- (2,2)--(3,2)--(3,1)--(2,1);
\draw (3,2) -- (2,2) -- (2,3)--(3,3)--(3,2)--(2,2);
\node at (0.5,-0.5) {$1$};
\node at (1.5,-0.5) {$2$};
\node at (2.5,-0.5) {$3$};
\node at (3.5,2.5) {$1$};
\node at (3.5,0.5) {$3$};
\node at (3.5,1.5) {$2$};
\draw [->] (0.05+2,0.5+2)--(0.95+2,0.5+2);\draw [->] (1/2,0.1)--(1/2,0.9);
\end{tikzpicture}
$$
or they will have to contain a right-up arrow $\begin{tikzpicture}[thick,scale=0.4,baseline=0.4 ex]
\draw (1,0) -- (0,0) -- (0,1)--(1,1)--(1,0)--(0,0);
\draw [->] (0.1,0.5)--(0.5,0.5)--(0.5,0.9);
\end{tikzpicture}$ in the last row or the last column. The operators corresponding to the latter routes will vanish when restricted to the subspace $H$ (as $\begin{tikzpicture}[thick,scale=0.4,baseline=0.4 ex]
\draw (1,0) -- (0,0) -- (0,1)--(1,1)--(1,0)--(0,0);
\draw [->] (0.1,0.5)--(0.5,0.5)--(0.5,0.9);
\end{tikzpicture}$ corresponds to $S^{*}C_{q}$ and $S^{*}C_{q}e_{0}=0$). Moreover, the operator corresponding to $$
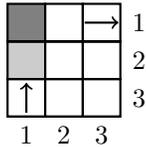
\begin{tikzpicture}[thick,scale=0.5]
\draw(1,0) -- (0,0) -- (0,1)--(1,1)--(1,0)--(0,0);
\filldraw[fill=black!20!white, draw=black](1,1) -- (0,1) -- (0,2)--(1,2)--(1,1)--(0,1);
\filldraw[fill=black!50!white, draw=black] (1,2) -- (0,2) -- (0,3)--(1,3)--(1,2)--(0,2);
\draw(2,0) -- (1,0) -- (1,1)--(2,1)--(2,0)--(1,0);
\draw (2,1) -- (1,1) -- (1,2)--(2,2)--(2,1)--(1,1);
\draw(2,2) -- (1,2) -- (1,3)--(2,3)--(2,2)--(1,2);
\draw (3,0) -- (2,0) -- (2,1)--(3,1)--(3,0)--(2,0);
\draw (3,1) -- (2,1) -- (2,2)--(3,2)--(3,1)--(2,1);
\draw (3,2) -- (2,2) -- (2,3)--(3,3)--(3,2)--(2,2);
\node at (0.5,-0.5) {$1$};
\node at (1.5,-0.5) {$2$};
\node at (2.5,-0.5) {$3$};
\node at (3.5,2.5) {$1$};
\node at (3.5,0.5) {$3$};
\node at (3.5,1.5) {$2$};
\draw [->] (0.05+2,0.5+2)--(0.95+2,0.5+2);\draw [->] (1/2,0.1)--(1/2,0.9);
\end{tikzpicture}$$ becomes a multiple (in fact $-q$) of the identity when restricted to $H.$ Therefore, by letting $\pi$ to be the $*$-representation of $\mathrm{Pol}(\mathrm{Mat}_{n})_{q}$ corresponding to the upper left $2\times 2$ sub-grid, we obtain $\pi(z_{1}^{1})|_{H}=\pi'(z_{1}^{1}).$
A similar argument shows that
$$
\pi(z_{j}^{k})|_{H}=\pi'(z_{j}^{k})\text{ for all $1\leq j,k \leq 2$}. 
$$
Hence, by induction, we can lift $\pi'$ to a $*$-representation $\Pi'$ of $\mathbb{C}[SU_{4}]_{q}.$ Consider now the $*$-homomorphism
$
\Phi:\mathbb{C}[SU_{6}]_{q}\to \mathbb{C}[SU_{4}]_{q}
$
given by
$$
\Phi(t_{ij})=\begin{cases}t_{ij} & \text{ if $1\leq i,j\leq 4$}\\ \delta_{ij}I & \text{ otherwise}.  \end{cases}
$$
One can show that $\Pi'\circ \Phi\circ \zeta $ is a $*$-representation of $\mathrm{Pol}(\mathrm{Mat}_{3})_{q}$ with the corresponding  grid as follows
$$
\begin{tikzpicture}[thick,scale=0.5]
\filldraw[fill=black!50!white, draw=black](1,0) -- (0,0) -- (0,1)--(1,1)--(1,0)--(0,0);
\filldraw[fill=black!20!white, draw=black](1,1) -- (0,1) -- (0,2)--(1,2)--(1,1)--(0,1);
\filldraw[fill=black!50!white, draw=black] (1,2) -- (0,2) -- (0,3)--(1,3)--(1,2)--(0,2);
\filldraw[fill=black!50!white, draw=black](2,0) -- (1,0) -- (1,1)--(2,1)--(2,0)--(1,0);
\draw (2,1) -- (1,1) -- (1,2)--(2,2)--(2,1)--(1,1);
\draw(2,2) -- (1,2) -- (1,3)--(2,3)--(2,2)--(1,2);
\filldraw[fill=black!50!white, draw=black] (3,0) -- (2,0) -- (2,1)--(3,1)--(3,0)--(2,0);
\filldraw[fill=black!50!white, draw=black](3,1) -- (2,1) -- (2,2)--(3,2)--(3,1)--(2,1);
\filldraw[fill=black!50!white, draw=black] (3,2) -- (2,2) -- (2,3)--(3,3)--(3,2)--(2,2);
\node at (0.5,-0.5) {$1$};
\node at (1.5,-0.5) {$2$};
\node at (2.5,-0.5) {$3$};
\node at (3.5,2.5) {$1$};
\node at (3.5,0.5) {$3$};
\node at (3.5,1.5) {$2$};
\end{tikzpicture}
$$
The original grid can be obtained by tensoring the $*$-representation $\Pi'\circ \Phi$ (of $\mathbb{C}[SU_{6}]_{q}$) with the $*$-representations 
$\lambda_{a}=\pi_{3}\otimes \pi_{4}\otimes \pi_{5}$ and $\lambda_{b}=\pi_{3}\otimes\pi_{4}$ so that 
$$
(\lambda_{a}\otimes(\Pi'\circ \Phi)\otimes \lambda_{b} )\circ \zeta\cong \pi.
$$
\subsection{Auxiliary Lemmas}
Let $\textbf{Z},$ as above, be the set of generators $\{z_{k}^{j}\}_{1\leq k,j\leq n}$ and $\textbf{Z}^{*}$ the set $\{(z_{k}^{j})^{*}\}_{1\leq k,j\leq n}.$ For $1\leq l,m\leq n,$ we let $\textbf{Z}_{l}^{m}=\{z_{k}^{j}\in \textbf{Z}|k\neq l,j\neq m\},$ and similar we define $\textbf{Z}_{l}^{m*}=\{(z_{k}^{j})^{*}\in \textbf{Z}^{*}|k\neq l,j\neq m\}.$
\\

\begin{lem}\label{linj}
Let $\textbf{I}=\{z_{k}^{n},z_{k+1}^{n},\dots,z_{n}^{n},z_{n}^{n-1},\dots, z_{n}^{m}\}\subseteq \textbf{Z}$ for $1\leq k,m\leq n,$ and let $\mathcal{C}(\textbf{I})\subseteq\mathrm{Pol}(\mathrm{Mat}_{n})_{q}$ be the unital $*$-sub-algebra generated by $\bf{I}.$ Then there exists a unique linear functional $\gamma_{\textbf{I}}$ on $\mathcal{C}(\textbf{I})$ with the property that for any $*$-representation $\pi:\mathrm{Pol}(\mathrm{Mat}_{n})_{q}\rightarrow B(K),$ with a subspace $H\subseteq K$ such that $\ker \pi(z_{j}^{l})^{*}\cap H=H$ for all $z_{j}^{l}\in \textbf{I},$ we have 
$$
\langle \pi(a)  u,w\rangle=\langle u,w\rangle\gamma_{\textbf{I}}(a)
$$
for all $u,w\in H.$
\end{lem}
\begin{proof}
We let
$$\gamma_{\textbf{I}}(a):=\langle \pi_{F,n}(a)v_{0},v_{0}\rangle$$
where $\pi_{F,n}$ is the Fock representation of $\mathrm{Pol}(\mathrm{Mat}_{n})_{q}$ and $v_{0}$ a unit vacuum vector.
From the relations in $\mathrm{Pol}(\mathrm{Mat}_{n})_{q}$, it is easy to see that the sub-algebra $\mathcal{C}(\textbf{I})$ is generated as a vector space by elements of the form $c b^{*}$ where $b,c$ are either monomials of the generators in $\textbf{I},$ or the identity $I.$ It follows from the properties of the Fock representation that for monomials $c,b$ we have
$$\begin{array}{ccc}\gamma_{\textbf{I}}(cb^{*})=\langle \pi_{F,n}(cb^{*})v_{0},v_{0}\rangle=\langle \pi_{F,n}(b)^{*}v_{0},\pi_{F,n}(c)^{*}v_{0}\rangle=0 & \text{if either $b$ or $c\neq I$}\end{array}$$
and also that $\gamma_{\textbf{I}}(I)=\langle v_{0},v_{0}\rangle=1.$
Thus, the linear functional $\gamma_{\textbf{I}}$ projects onto the vector space $V$ that is the quotient of $\mathcal{C}(\textbf{I})$ by the subspace generated by elements $ab^{*},$ where $a,b$ are monomials in $\textbf{I}$, not both equal to $I.$ The vector space $V$ is of dimension at most $1,$ and since $\gamma_{\textbf{I}}\neq 0,$ we see that $\dim V=1.$
\\

For any $u,w\in H,$ consider the linear functional $\gamma_{uw}$ defined as $a\mapsto\langle  \pi(a) u,w\rangle.$ From the properties of $H,$ we get for any two monomials $b,c\in \mathcal{C}(\textbf{I}),$ not both multiples of $I,$ that
$$
\gamma_{uw}(cb^{*})=\langle \pi(cb^{*})u,w\rangle=\langle \pi(b)^{*}u,\pi(c)^{*} w\rangle=0.
$$
So $\gamma_{uw}$ factors through the subspace $V$ and thus it must be a multiple of $\gamma_{\textbf{I}}.$ Evaluating $\gamma_{uw}(I)=\langle u,v\rangle,$ we get that $\gamma_{uw}=\langle u,w\rangle \gamma_{\textbf{I}}.$
\end{proof}
\begin{lem}\label{linj2}
If we let $\mathcal{A}(\textbf{I})$ be the unital algebra generated by the elements in $\textbf{I},$ then viewing $\mathcal{A}(\textbf{I})$ as a vector space over $\mathbb{C},$ the functional $\gamma_{\textbf{I}}$ gives a non-degenerate inner product $\langle\cdot,\cdot\rangle_{\textbf{I}}$ on $\mathcal{A}(\textbf{I})$
by $\langle a,b\rangle_{\textbf{I}}=\gamma_{\textbf{I}}(b^{*}a)$ for $a,b\in \mathcal{A}(\textbf{I}).$ Moreover, the monomials in $\mathcal{A}(\textbf{I})$ form an orthogonal basis of $\mathcal{A}(\textbf{I})$ with respect to the inner product $\langle\cdot,\cdot\rangle_{\textbf{I}}.$
\end{lem}

\begin{proof}
The monomials obviously span $\mathcal{A}(\textbf{I})$ as a vector space, so we only need to check that they form an orthogonal basis.
Notice that by~\eqref{zaa1} and~\eqref{zaa2}, the elements in $\textbf{I}$ commute or $q$-commute. Thus, every monomial $a$ can be written in the form 
\begin{equation}\label{aww}
a=\beta (z_{k}^{n})^{\alpha_{k}}(z_{k+1}^{n})^{\alpha_{k+1}}\cdots(z_{n-1}^{n})^{\alpha_{n-1}}(z_{n}^{m})^{\beta_{m}}(z_{n}^{m+1})^{\beta_{m+1}}\cdots (z_{n}^{n})^{\beta_{n}}
\end{equation}
for some $\beta\in \mathbb{C},\alpha_{i},\beta_{i}\in \mathbb{Z}_{+}.$ Using the terminology from Corollary~\ref{maincor}, we have $a=\beta z(A),$ where $A=(a_{ij})_{ij}\in M_{n}(\mathbb{Z}_{+})$ with $a_{in}=\alpha_{i}$ if $k\leq i\leq n,$ $a_{nj}=\beta_{j}$ if $m\leq j\leq n$ and $a_{ij}=0$ otherwise. Under the Fock representation, we have $\pi_{F,n}(z(A)) v_{0}\bot \pi_{F,n}(z(B)) v_{0}$ for any $B\in M_{n}(\mathbb{Z}_{+})$ not equal to $A,$ it follows that $\pi_{F,n}(z(A))v_{0}\bot \pi_{F,n}(b) v_{0}$ for any monomial $b\in \mathcal{A}(\textbf{I})$ that is not a multiple of $a.$ So if $a\neq b$ are monomials in $\mathcal{A}(\textbf{I}),$ then we have
$$
\langle a,b\rangle_{I}=\gamma_{\textbf{I}}(b^{*}a)=\langle \pi_{F,n}(b^{*}a) v_{0},v_{0}\rangle =\langle \pi_{F,n}(a) v_{0},\pi_{F,n}(b) v_{0}\rangle=0
$$
and as $\pi_{F,n}(z(B)) v_{0}\neq 0$ for any $B\in M_{n}(\mathbb{Z}_{+}),$ we also get $\langle a,a\rangle_{I}=||\pi_{F,n}(a) v_{0}||^{2}>0.$
\end{proof}
\begin{lem}\label{span1}
Let $|\textbf{I}|$ be the cardinality of $\textbf{I}.$ For any multi-index $$\textbf{m}=(m_{1},m_{2},\dots, m_{|\textbf{I}|})\in \mathbb{Z}_{+}^{|\textbf{I}|},$$ let 
\begin{equation}\label{orderbasis}
z(\textbf{m}):=(z_{k}^{n})^{m_{1}}(z_{k+1}^{n})^{m_{2}}\cdots(z_{n-1}^{n})^{m_{n-k}}(z_{n}^{m})^{m_{n-k+1}}(z_{n}^{m+1})^{m_{n-k+2}}\cdots (z_{n}^{n})^{m_{|\textbf{I}|}}.
\end{equation}
Then the elements of the form $z(\textbf{m}) b,$ where $b$ is a monomial in $\textbf{Z}\backslash \textbf{I},$ form a vector space basis for $\mathbb{C}[\mathrm{Mat}_{n}]_{q}.$
\end{lem}
\begin{proof}
As the generators $z_{i}^{n},z_{n}^{j}$ either commute or $q$-commute, it is enough to prove the statement for 
$$
\textbf{I}=\{z^{n}_{1},\dots,z_{n}^{n},z_{n}^{n-1},\dots, z_{n}^{1}\}.
$$
\\

Firstly, using the ordering of the generators $z_{k}^{j},$ we can write any $z(A)$ with $A\in M_{n}(\mathbb{Z}_{+})$ as $z(A')z(A''),$ where $A'$ has only non-zero values in the last column, $A''$ is equal to zero in the last column and $A=A'+A''.$ Notice that for any $1\leq k\leq n-1,$ we can use~\eqref{zaa1}-\eqref{zaa3} to see that in the ordering of the generators, we have that $z_{k}^{n}$ commutes with all the generators proceeding it that is not in the form $z_{m}^{n}.$ Hence we can write every $z(B)$ as $z(B')z(B'')z(B'''),$ with $B'$ only nonzero in the last column, $B''$ only nonzero in the last row $,B'''$ equal to zero in the last row and column and $B=B'+B''+B'''.$ We can then write $z(B')z(B'')=z(\textbf{m})$ for some $\textbf{m}\in \mathbb{Z}_{+}^{2n-1}.$
\end{proof}
\begin{lem}\label{pol}
Given a $*$-representation $\pi$ of $\mathrm{Pol}(\mathrm{Mat}_{n})_{q} $
the operators $$\pi(z_{1}^{n}),\dots, \pi(z_{n}^{n}),\pi(z_{n}^{n-1}),\dots,\pi(z_{n}^{1})$$ are all contractions. %Also, the kernel of the image of $z_{k}^{n}$ is trivial unless $z_{k}^{n}$ is mapped to zero by $\pi$, and the same is true for $z_{n}^{k}$
\end{lem}
\begin{proof}
Let $\textbf{I}=\{z_{1}^{n},z_{2}^{n},\dots,z_{n}^{n}\}.$ Then by~\eqref{zaa4} we have $\mathcal{C}(\textbf{I})\cong \mathrm{Pol}(\mathbb{C}^{n})$ with a $*$-isomorphism given by $$z_{j}^{n}\mapsto z_{j},1\leq j\leq n.$$ The statement now follows from the fact that $||\phi(z_{j})||\leq 1$ for $1\leq j\leq n$ and any bounded $*$-representation $\phi$ of $\mathrm{Pol}(\mathbb{C}^{n})$ (see~\cite{pet}, though notice that they defined $\mathrm{Pol}(\mathbb{C}^{n})$ using the generators $a_{i}=\frac{1}{\sqrt{1-q^{2}}}z_{i}$). The same also holds for $\textbf{I}=\{z_{n}^{1},z_{n}^{2},\dots,z_{n}^{n}\}.$
 %and that the kernel of a generator is trivial unless that generator is mapped to zero. The same argument also works with $\textbf{I}=\{z_{n}^{1},z_{n}^{2},\dots, z_{n}^{n}\}$. 
\end{proof}
\subsection{The Case $\textbf{A}$}
In this section we prove the existence of the lifting $\Pi:\mathbb{C}[SU_{2n}]_{q}\to B(K)$ for a fixed irreducible $*$-representation $\pi:\mathrm{Pol}(\mathrm{Mat}_{n})_{q}\rightarrow B(K)$ such that $\ker \pi(z_{k}^{n})^{*}=\{0\}$ for some $1\leq k\leq n.$ Let us assume that $k$ is the largest integer with this property.
\\

Throughout this section write $$H:=\ker (\pi(z_{k}^{n})^{*}\pi(z_{k}^{n})-I)$$ and fix the set of generators
$$
\textbf{I}:=\{z_{k+1}^{n},\dots, z_{n}^{n}\}.
$$
\begin{lem}\label{polly}
If $k=n,$ then $H=K.$
If $1\leq k<n,$ then $H$ is a non-trivial proper subspace of $K$ and $H=\cap_{j=k+1}^{n}\ker \pi(z_{j}^{n})^{*}.$
\end{lem}
\begin{proof}
If $\ker \pi(z_{n}^{n})^{*}=0,$ then by~\eqref{zaa44} $z_{n}^{n}$ satisfies the equation $$(z_{n}^{n})^{*}z_{n}^{n}=q^{2}z_{n}^{n}(z_{n}^{n})^{*}+(1-q^{2})I.$$  By Lemma~\ref{polclem} $\pi(z_{n}^{n})$ is a unitary operator and hence 
$\ker (\pi(z_{n}^{n})^{*}\pi(z_{n}^{n})-I)=K.$ 
\\

Assume now $1\leq k<n$ and set  $H':=\cap_{j=k+1}^{n}\ker \pi(z_{j}^{n})^{*}.$ For the simplicity, we shall write $z_{j}^{l}$ for the images $\pi(z_{j}^{l}),1\leq j,l\leq n.$ As $q z_{k}^{n}(z_{j}^{n})^{*}=(z_{j}^{n})^{*}z_{k}^{n}$ and $z_{k}^{n}z_{j}^{n}=q z_{j}^{n}z_{k}^{n}$ for $k<j,$ we see that $H'$ reduces $z_{k}^{n}.$ Moreover, as $$(z_{k}^{n})^{*}z_{k}^{n}=q^{2}z_{k}^{n}(z_{k}^{n})^{*}+(1-q^{2})(I-\sum_{j=k+1}z_{j}^{n}(z_{j}^{n})^{*})$$ and $(z_{j}^{n})^{*} H'=0$ for $k+1\leq j\leq n,$  we obtain
$$
(z_{k}^{n})^{*}z_{k}^{n}|_{H'}=q^{2}z_{k}^{n}(z_{k}^{n})^{*}|_{H'}+(1-q^{2})I
$$
and, as $\ker (z_{k}^{n})^{*}=0,$ we can use Lemma~\ref{polc} to see that $z_{k}^{n}|_{H'}$ is a unitary operator and hence $H'\subseteq H.$
\\

To see the other inclusion, notice that $(z_{j}^{n})^{*}((z_{k}^{n})^{*}z_{k}^{n})=q^{2}((z_{k}^{n})^{*}z_{k}^{n})(z_{j}^{n})^{*}$ for $k<j\leq n$ and hence by Lemma~\ref{pol}, we have for $v\in H$ 
$$||(z_{j}^{n})^{*} v||=||(z_{j}^{n})^{*}((z_{k}^{n})^{*}z_{k}^{n}) v||=q^{2}||((z_{k}^{n})^{*}z_{k}^{n})(z_{j}^{n})^{*}v||\leq 
$$
$$
q^{2}||(z_{k}^{n})^{*}z_{k}^{n}||\cdot||(z_{j}^{n})^{*} v||\leq q^{2}||(z_{j}^{n})^{*} v||$$ 
 giving $||(z_{j}^{n})^{*} v||=0$ and $v\in \ker (z_{j}^{n})^{*}.$ Thus $H= H'.$ 
\\

We now need to prove that $H\neq 0$ if $1\leq k<n.$ Assume contrary that $$H=\cap_{j=k+1}^{n}\ker (z_{j}^{n})^{*}=0$$ and let $1\leq m\leq n$ be the first integer such that $\cap_{j=m}^{n}\ker (z_{j}^{n})^{*}\neq 0.$ Notice that this integer exists since we assumed that $\ker (z_{n}^{n})^{*}\neq \{0\}.$
The subspace $L:=\cap_{j=m}^{n}\ker (z_{j}^{n})^{*}$ reduces the operators $z_{1}^{n},\dots ,z_{m-1}^{n}$ by the same argument that we used to prove the similar statement for $H.$ Now, the restriction of $z_{m-1}^{n}$ to $L$ must again satisfy the equation 
$$
(z_{m-1}^{n})^{*}z_{m-1}^{n}|_{L}=q^{2}z_{m-1}^{n}(z_{m-1}^{n})^{*}|_{L}+(1-q^{2})I
$$
and as $\ker (z_{m-1}^{n})^{*}|_{L}=0,$ it follows that $z_{m-1}^{n}|_{L}$ is unitary. 
\\

We then claim that $(z_{l}^{n})^{*}|_{L}=0$ for $1\leq l< m-1$ and hence $$L=\cap_{j=m}^{n}\ker (z_{j}^{n})^{*}\subseteq \ker (z_{k}^{n})^{*}.$$ As $k<m-1,$ this gives a contradiction. To see the claim, observe first that for $m-2$ we have
$$
(z_{m-2}^{n})^{*}z_{m-2}^{n}|_{L}=q^{2}z_{m-2}^{n}(z_{m-2}^{n})^{*}|_{L}+(1-q^{2})(I-\sum_{r=m-1}^{n}z_{r}^{n}(z_{r}^{n})^{*}|_{L})=
$$
$$
q^{2}z_{m-2}^{n}(z_{m-2}^{n})^{*}|_{L}+(1-q^{2})(I-z_{m-1}^{n}(z_{m-1}^{n})^{*}|_{L})=q^{2}z_{m-2}^{n}(z_{m-2}^{n})^{*}|_{L}
$$
as $z_{m-1}^{n}(z_{m-1}^{n})^{*}|_{L}=I.$ This gives $||z_{m-2}^{n}(z_{m-2}^{n})^{*}|_{L}||=||(z_{m-2}^{n})^{*}z_{m-2}^{n}|_{L}||=q^{2}||z_{m-2}^{n}(z_{m-2}^{n})^{*}|_{L}||$ and hence $(z_{m-2}^{n})^{*}|_{L}=0.$ We now use an induction type argument on $1\leq j<m-2;$ if the equation $I-\sum_{r=j-1}^{n}z_{r}^{n}(z_{r}^{n})^{*}|_{L}=0$ holds, then we can again deduce that $$||z_{j}^{n}(z_{j}^{n})^{*}|_{L}||=||(z_{j}^{n})^{*}z_{j}^{n}|_{L}||=q^{2}||z_{j}^{n}(z_{j}^{n})^{*}|_{L}||$$ giving $(z_{j}^{n})^{*}|_{L}=0$ and hence also that $I-\sum_{r=j-2}^{n}z_{r}^{n}(z_{r}^{n})^{*}|_{L}=0.$
\end{proof}
\begin{lem}\label{miss}
For $1\leq m\leq n,$  let
$$
A_{m}:=I-\sum_{j=m}^{n}z_{j}^{n}(z_{j}^{n})^{*}.
$$
Then the following relations hold in $\mathrm{Pol}(\mathrm{Mat}_{n})_{q}$
\begin{equation}\label{an}A_{m}=\frac{1}{1-q^{2}}((z_{m}^{n})^{*}z_{m}^{n}-z_{m}^{n}(z_{m}^{n})^{*}) \end{equation}
\begin{equation}\label{123}z_{j}^{n}A_{m}=A_{m}z_{j}^{n},1\leq j<m\end{equation} \begin{equation}\label{1234}q^{2}z_{j}^{n}A_{m}=A_{m}z_{j}^{n},m\leq j\leq n.\end{equation} Furthermore, if $\Pi $ is a $*$-representation of $\mathrm{Pol}(\mathrm{Mat}_{n})_{q}$ such that $\ker \Pi(z_{m}^{n})^{*}=0,$ then also $\Pi(A_{m})=0$ and 
\begin{enumerate}
\item $\pi(z_{m}^{n})$ is normal,
\item $\pi(z_{j}^{n})=0$ for $1\leq j<m.$
\end{enumerate}
\end{lem}
\begin{proof}
The equation~\eqref{an} follows by subtracting $z_{m}^{n}(z_{m}^{n})^{*}$ from both side of the equation
\begin{equation}\label{qqqq}
(z_{m}^{n})^{*}z_{m}^{n}=q^{2}z_{m}^{n}(z_{m}^{n})^{*}+(1-q^{2})(I-\sum_{m<j}z_{j}^{n}(z_{j}^{n})^{*})
\end{equation}
and dividing by $(1-q^{2}).$ When $j\neq m,$~\eqref{123} and~\eqref{1234} follow directly from equations~\eqref{zaa1} and~\eqref{zaa43}. When $j=m,$ we have (with $A_{n+1}=I$)
$$A_{m}z_{m}^{n}=-z_{m}^{n}(z_{m}^{n})^{*}z_{m}^{n}+A_{m+1}z_{m}^{n}=$$
$$
-z_{m}^{n}(q^{2}z_{m}^{n}(z_{m}^{n})^{*}+(1-q^{2})A_{m+1})+z_{m}^{n}A_{m+1}=
$$
$$
-q^{2}z_{m}^{n}z_{m}^{n}(z_{m}^{n})^{*}+q^{2}z_{m}^{n}A_{m+1}=q^{2}z_{m}^{n}A_{m}.
$$
Let $\Pi$ be a $*$-representation satisfying the conditions of the lemma and write simply $z_{i}^{j}$ for the image $\Pi(z_{i}^{j}),i,j=1,\dots,n.$ Let $(z_{m}^{n})^{*}=U |(z_{m}^{n})^{*}|$ be the polar decomposition of $(z_{m}^{n})^{*},$ here $|(z_{m}^{n})^{*}|=(z_{m}^{n}(z_{m}^{n})^{*})^{\frac{1}{2}}$ and $U$ is an isometry. As $$(z_{m}^{n}(z_{m}^{n})^{*})A_{m}=A_{m}(z_{m}^{n}(z_{m}^{n})^{*}),$$ we can easily deduce from~\eqref{1234} that 
$$
UA_{m}=q^{2}A_{m}U
$$ 
giving $q^{2}U^{*}A_{m}U=A_{m}$ and hence $A_{m}=0.$ From equation~\eqref{qqqq}, it now follows that $z_{m}^{n}$ is normal. As we also have $z_{j}^{n}(z_{m}^{n}(z_{m}^{n})^{*})=(z_{m}^{n}(z_{m}^{n})^{*})z_{j}^{n}$ for $1\leq j<m,$ we get similarly that $ q z_{j}^{n}U= U z_{j}^{n}$ giving $U^{*}z_{j}^{n}U=q^{-1}  z_{j}^{n}$ and hence $z_{j}^{n}=0.$
\end{proof}
In section $3,$ during our informal discussion, we indicated how to approach the problem of lifting a $*$-representation $\pi$ of $\mathrm{Pol}(\mathrm{Mat}_{n})_{q}$ to a $*$-representations of $\mathbb{C}[SU_{2n}]_{q}$ by reducing $\pi$ to a $*$-representation of $\mathrm{Pol}(\mathrm{Mat}_{n-1})_{q}.$ In the case $\textbf{A},$ this was done by isolating the paths not starting on the integer $k,$ satisfying the condition $\ker \pi(z_{k}^{n})^{*}=\{0\}.$ The next proposition is the general statement of this. However, while the heuristic picture is quite clear, the proof turns out to be somewhat arduous as we must check that we actually end up with a representation of $\mathrm{Pol}(\mathrm{Mat}_{n-1})_{q}$ and hence must verify that the equations~\eqref{zaa41}-~\eqref{zaa44} hold.
\begin{prop}\label{l2}
For the irreducible $*$-representation $\pi :\mathrm{Pol}(\mathrm{Mat}_{n})_{q}\rightarrow B(K)$ the following holds
\begin{enumerate}
\item we have $\pi(z_{k}^{n})|_{H}=e^{i \varphi} I |_{H}$ for some $\varphi \in [0,2\pi),$ 
\item
there is an irreducible $*$-representation $\pi' :\mathrm{Pol}(\mathrm{Mat}_{n-1})_{q}\rightarrow B(H)$ such that 
\begin{equation}\label{pii}
\pi'(z_{j}^{m})=
\begin{cases}
-q e^{-i\varphi}\pi(z_{j}^{m})|_{H} & \text{If $1\leq j< k$}\\           
\pi(z_{j+1}^{m})|_{H}  & \text{If $k\leq j\leq n-1 $}\\
\end{cases}
\end{equation}.
\end{enumerate}
\end{prop}
\begin{proof}
The proof of the irreducibility of $\pi'$ is postponed until after Lemma~\ref{molly}.
Let again $\textbf{I}=\{z_{k+1}^{n},z_{k+1}^{n},\dots,z_{n}^{n}\}.$ By Lemma~\ref{polly}, $H=\cap_{j=k+1}^{n}\ker \pi(z_{j}^{n})^{*}.$ By Lemma~\ref{linj}, there is a linear functional $\gamma_{\textbf{I}}$ on $\mathcal{A}(\textbf{I}),$ such that for all $ v,u\in H,$ and $a\in\mathcal{A}(\textbf{I}),$ we have $\langle a v, u\rangle=\gamma_{\textbf{I}}(a)\langle v,u\rangle.$ Again, we suppress $\pi$ and write $z_{i}^{j}$ instead of $\pi(z_{i}^{j})$ for $i,j=1,\dots,n.$
\\

It follows from the proof of Lemma~\ref{polly}, that the restriction $z_{k}^{n}|_{H}$ is unitary. Now, from Lemma~\ref{miss}, we know that the elements $z_{1}^{n},\dots, z_{k-1}^{n}\in \mathrm{Pol}(\mathrm{Mat}_{n})_{q}$ are all mapped to zero by $\pi$ and this gives that relation~\eqref{zaa3} reduces to $z_{j}^{m}z_{k}^{n}=z_{k}^{n}z_{j}^{m}$ for $j<k$ and $m<n$. As also $z_{j}^{m}z_{k}^{n}=z_{k}^{n}z_{j}^{m}$ for $k<j$ and $m<n$ holds by~\eqref{zaa3}, we get that the elements in $\textbf{Z}_{k}^{n}$ commute with both $z_{k}^{n}$ and $(z_{k}^{n})^{*}$ (the latter follows from~\eqref{zaa4}) and hence $H$ is a reducing subspace for the $\pi$-images of elements in $\textbf{Z}_{k}^{n}.$
\\

By~\eqref{zaa4}, we have $z_{k}^{j}(z_{m}^{n})^{*}=(z_{m}^{n})^{*}z_{k}^{j}$ for $1\leq j\leq n-1$ and $k+1\leq m \leq n$ and hence $H$ is invariant under the operators $z_{k}^{1},z_{k}^{2},\dots,z_{k}^{n-1}.$ As also $z_{k}^{n}z_{k}^{j}=q z_{k}^{j}z_{k}^{n}$ and $z_{k}^{n}|_{H}$ is unitary, we must have
\begin{equation}\label{zero}
z_{k}^{j}|_{H}=0 \text{  for $1\leq j\leq n-1.$}
\end{equation}

Let $U$ be the unitary operator in the polar decomposition of $z_{k}^{n}$ (that $U$ actually is unitary follows from the fact that $z_{k}^{n}$ is normal and $\ker (z_{k}^{n})^{*}=\{0\}$). Then we claim that $U$ commutes with the image of $\mathrm{Pol}(\mathrm{Mat}_{n})_{q}$ under $\pi.$ As we already know that $z_{k}^{n}$ and $(z_{k}^{n})^{*}$ commutes with the elements in $\textbf{Z}_{k}^{n},$ it follows that so does also $U.$ So what is left to prove is that $U$ commutes with $z_{k}^{1},z_{k}^{2},\dots z_{k}^{n-1}$ and the elements in $\textbf{I}.$ However, as $z_{k}^{n}$ is normal and $q$-commutes with these operators, we can apply the Fuglede-Putnam Theorem (if $T,N,M\in B(\mathcal{H})$ and $N,M$ are normal, then $NT=TM$ implies $N^{*}T=T M^{*}$) with $N=z_{k}^{n},M=q z_{k}^{n}$ and $T=z_{k}^{j}$ to get
$$
(z_{k}^{n})^{*}z_{k}^{j}=q z_{k}^{j}(z_{k}^{n})^{*}
$$
and hence
$$
((z_{k}^{n})^{*}z_{k}^{n})z_{k}^{j}=q^{2} z_{k}^{j}((z_{k}^{n})^{*}z_{k}^{n}).
$$
This gives $|z_{k}^{n}|z_{k}^{j}=q z_{k}^{j}|z_{k}^{n}|$ and $$U z_{k}^{j}|z_{k}^{n}|= z_{k}^{j}U| z_{k}^{n}|U.$$ As $\ker|z_{k}^{n}|=\{0\}$ and hence $\overline{|z_{k}^{n}|K}=K,$ this implies $U z_{k}^{j}=z_{k}^{j}U.$
By Schur's Lemma, $U=e^{i\varphi}I.$
\\

So it is left to prove that $\pi'$ actually defines a $*$-representation of $\mathrm{Pol}(\mathrm{Mat}_{n-1})_{q}.$ It is easy to see that relations~\eqref{zaa1}-\eqref{zaa3*} hold for the $\pi$-images of the generators, so what is left to verify is~\eqref{zaa4}. We do this by splitting~\eqref{zaa4} into cases~\eqref{zaa41}-\eqref{zaa44}.
\\

It is clear that relation~\eqref{zaa41} holds.
\\

For $j,m\neq k,$ we have by~\eqref{zaa43}
$$
(z_{j}^{\alpha})^{*}z_{m}^{\alpha}|_{H}=q z_{m}^{a}(z_{j}^{a})^{*}|_{H}-(q^{-1}-q)\sum_{c=\alpha+1}^{n}z_{m}^{c}(z_{j}^{c})^{*}|_{H}=
$$
$$
q z_{m}^{\alpha}(z_{j}^{\alpha})^{*}|_{H}-(q^{-1}-q)\sum_{c=\alpha+1}^{n-1}z_{m}^{c}(z_{j}^{c})^{*}|_{H}
$$
as $z_{m}^{n}(z_{j}^{n})^{*}|_{H}=0$ when $m,j\neq k.$ Hence $\pi'(z_{j}^{\alpha}), \alpha,j=1,\dots, n-1,$ satisfy relation~\eqref{zaa43}.
\\

If $j\geq k$ then it is easy see that~\eqref{zaa42} and~\eqref{zaa44} hold for the images $$\pi'(z_{j}^{\beta})^{*}\pi'(z_{j}^{\alpha})=(z_{j+1}^{\beta})^{*}z_{j+1}^{\alpha}|_{H},a,b=1,\dots n-1$$ as they hold for $\pi$ and as we have $z_{k+1}^{n}(z_{k+1}^{n})^{*}|_{H},z_{k+2}^{n}(z_{k+2}^{n})^{*}|_{H},\dots,z_{n}^{n}(z_{n}^{n})^{*}|_{H}=0.$
\\

In order to verify the remaining relations, which are~\eqref{zaa42} and~\eqref{zaa44} for $\pi'(z_{j}^{\beta})^{*}\pi'(z_{j}^{\alpha})$ when $j<k,$ we observe first that as $z_{k}^{\beta}|_{H}=0$ for $1\leq \beta\leq n-1,$ we have $(z_{k}^{\alpha})^{*}z_{k}^{\beta}|_{H}=0$ and by \eqref{zaa42}
\begin{equation}\label{etic1}
0=q z_{k}^{\beta}(z_{k}^{\alpha})^{*}|_{H}-(q^{-1}-q)\sum_{k<l}^{n}z_{l}^{\beta}(z_{l}^{\alpha})^{*}|_{H} 
\end{equation}
 if $\alpha\neq \beta$ and by~\eqref{zaa44}
\begin{multline}\label{epic1}
0=q^{2} z_{k}^{\alpha}(z_{k}^{\alpha})^{*}|_{H}-(1-q^{2})\sum_{k<l}^{n}z_{l}^{\alpha}(z_{l}^{\alpha})^{*}|_{H}-
(1-q^{2})\sum_{\alpha<\beta}^{n}z_{k}^{\beta}(z_{k}^{\beta})^{*}|_{H}+\\
q^{-2}(1-q^{2})^{2}\sum_{k<l,\alpha<\beta}^{n}z_{l}^{\beta}(z_{l}^{\beta})^{*}|_{H}+(1-q^{2})I  
\end{multline}
if $\alpha=\beta.$ Assume now that $m=j<k$ and $\alpha\neq \beta$ then we have
$$
(z_{m}^{\alpha})^{*}z_{m}^{\beta}|_{H}=q z_{m}^{\beta}(z_{m}^{\alpha})^{*}|_{H}-(q^{-1}-q)\sum_{m<l}^{n}z_{l}^{\beta}(z_{l}^{\alpha})^{*}|_{H}=
$$
$$
=(z_{m}^{\alpha})^{*}z_{m}^{\beta}|_{H}=q z_{m}^{\beta}(z_{m}^{\alpha})^{*}|_{H}-(q^{-1}-q)\sum_{k<l}^{n}z_{l}^{\beta}(z_{l}^{\alpha})^{*}|_{H}-(q^{-1}-q)z_{k}^{\beta}(z_{k}^{\alpha})^{*}|_{H}-
$$
$$
-(q^{-1}-q)\sum_{m<l}^{k-1}z_{l}^{\beta}(z_{l}^{\alpha})^{*}|_{H}.
$$
By~\eqref{etic1}, we get

$$
(z_{m}^{\alpha})^{*}z_{m}^{\beta}|_{H}=q z_{m}^{\beta}(z_{m}^{\alpha})^{*}|_{H}-(q^{-1}-q)\sum_{k<l}^{n}z_{l}^{\beta}(z_{l}^{\alpha})^{*}|_{H}-
$$
$$
-(q^{-2}-1)(q-q^{-1})\sum_{k<l}^{n}z_{l}^{\beta}(z_{l}^{\alpha})^{*}|_{H}-(q^{-1}-q)\sum_{m<l}^{k-1}z_{l}^{\beta}(z_{l}^{\alpha})^{*}|_{H}=
$$
$$
=q z_{m}^{\beta}(z_{m}^{\alpha})^{*}|_{H}-(q^{-1}-q)\sum_{k<l}^{n}(-q)^{-2}z_{l}^{\beta}(z_{l}^{\alpha})^{*}|_{H}-
(q^{-1}-q)\sum_{m<l}^{k-1}z_{l}^{\beta}(z_{l}^{\alpha})^{*}|_{H}.
$$
Finally if $m=j<k$ and $\alpha=\beta,$ then
$$
(z_{m}^{\alpha})^{*}z_{m}^{\alpha}|_{H}=q^{2}z_{m}^{\alpha}(z_{m}^{\alpha})^{*}|_{H}-(1-q^{2})\sum_{m<l}^{n}z_{l}^{\alpha}(z_{l}^{\alpha})^{*}|_{H}-
(1-q^{2})\sum_{a<l}^{n}z_{k}^{l}(z_{k}^{l})^{*}|_{H}+
$$
$$
+q^{-2}(1-q^{2})^{2}\sum_{m<l,a<b}^{n}z_{l}^{b}(z_{l}^{b})^{*}|_{H}+(1-q^{2})I=
$$
$$
=q^{2}z_{m}^{a}(z_{m}^{a})^{*}|_{H}-(1-q^{2})\sum_{m<l}^{k-1}z_{l}^{a}(z_{l}^{a})^{*}|_{H}-(1-q^{2})z_{k}^{a}(z_{k}^{a})^{*}|_{H}-(1-q^{2})\sum_{k<l}^{n}z_{l}^{a}(z_{l}^{a})^{*}|_{H}+
$$
$$
+q^{-2}(1-q^{2})^{2}\sum_{m<l}^{k-1}\sum_{a<b}^{n-1}z_{l}^{b}(z_{l}^{b})^{*}|_{H}+q^{-2}(1-q^{2})^{2}\sum_{\alpha<\beta}^{n}z_{k}^{\beta}(z_{k}^{\alpha})^{*}|_{H}+
$$
$$
+q^{-2}(1-q^{2})^{2}\sum_{k<l}^{n}\sum_{\alpha<\beta}^{n}z_{l}^{\beta}(z_{l}^{\beta})^{*}|_{H}+(1-q^{2})I.
$$
If, in this last sum, we use~\eqref{epic1} to substitute $z_{k}^{\alpha}(z_{k}^{\alpha})^{*}|_{H},$ then a similar calculation as in the case $\alpha\neq \beta$ yields the final case.
\end{proof}
%The map
%$$
%z_{m}^{j}\mapsto
%\begin{cases}
%e^{-i\varphi}z_{m}^{j}& \text{if $m=k$}\\
%z_{m}^{j}& \text{if $m\neq k$}
%\end{cases}
%$$
%can be extended to a $*$-automorphism $\mathrm{Pol}(\mathrm{Mat}_{n})_{q}\rightarrow\mathrm{Pol}(\mathrm{Mat}_{n})_{q}.$
%By composing this map with the representation $\pi,$ we can assume without loss of generality that $z_{k}^{n}|_{H}=I.$
%Once we lift to $\mathbb{C}[SU(2n)]_{q},$ we can then get a lift of the original $*$-homomorphism by using the $*$-automorphism of $\mathbb{C}[SU(2n)]_{q}$ given by 
%$$
%t_{m,l}\mapsto
%\begin{cases}
%e^{-i\varphi}t_{m,l}& \text{if $m=k$}\\
%e^{i\varphi}t_{m,l}& \text{if $m=n+k$}\\
%t_{m,l}& \text{otherwise.}
%\end{cases}
%%$$

\begin{lem}\label{molly}
Assume $k<n.$ For any multi-index $\textbf{m}=(m_{1},m_{2},\dots,m_{n-k})\in \mathbb{Z}_{+}^{n-k},$ let $$z(\textbf{m}):=(z_{k+1}^{n})^{m_{1}}(z_{k+2}^{n})^{m_{2}}\dots (z_{n}^{n})^{m_{n-k}}$$ and $$H_{\textbf{m}}=\overline{\pi(z(\textbf{m}))H}.$$ Then $H_{\textbf{m}}\bot H_{\textbf{n}}$ for $\textbf{m}\neq \textbf{n}$ and 
\begin{equation}\label{kfull}
K=\bigvee_{\textbf{m}\in \mathbb{Z}_{+}^{n-k}}H_{\textbf{m}}.
\end{equation}
\end{lem}
\begin{proof}
Let $u',v'\in H$ and $v=\pi(z(\textbf{n})) v'\in H_{\textbf{n}},u=\pi(z(\textbf{m})) u'\in H_{\textbf{m}}$ for $\textbf{n}\neq \textbf{m}.$
By Lemma~\ref{linj} and Lemma~\ref{linj2} applied to $\textbf{I}=\{z_{k+1}^{n},z_{k+1}^{n},\dots,z_{n}^{n}\}$ and $H,$ we have 
$$\langle v,u\rangle=\langle \pi(z(\textbf{n})) v',\pi(z(\textbf{m})) u'\rangle=\gamma_{\textbf{I}}(z(\textbf{m})^{*}z(\textbf{n}))\langle v',u'\rangle=0$$ and hence $H_{\textbf{m}}\bot H_{\textbf{n}}$ whenever $\textbf{m}\neq \textbf{n}$. 
\\

To prove~\eqref{kfull}, it is enough to prove that the right hand side is reducing $\textbf{Z},$ as then the equality follows from the irreducibility of $\pi.$ By Lemma~\ref{span1}, the elements of the form $z(\textbf{m})b,$ where $b$ is a monomial of generators in $\textbf{Z}\backslash \textbf{I},$ is a basis of $\mathbb{C}[\mathrm{Mat}_{n}]_{q}$ and thus for any $z_{j}^{l}\in\textbf{Z},$ we have $z_{j}^{l}z(\textbf{m})=\sum_{i}z(\textbf{m}_{i})b_{i},$ where $b_{i}$ is in the unital algebra generated by $\textbf{Z}\backslash \textbf{I}.$ As $H$ is invariant under the operators in $\textbf{Z}\backslash \mathbf{I},$ we then get
$$\pi(z_{j}^{l}) H_{\textbf{m}}=\pi(z_{j}^{l})\overline{\pi(z(\textbf{m}))H}\subseteq$$
$$
\bigvee_{i}\overline{\pi(z(\mathbf{m}_{i})b_{i}) H}\subseteq \bigvee_{i} \overline{\pi(z(\mathbf{m}_{i})) H}\subseteq \bigvee_{\textbf{m}\in \mathbb{Z}_{+}^{n-k}}H_{\textbf{m}}.
$$
Let us now prove the invariance under the operators in $\textbf{Z}^{*}.$
If $1\leq j\leq k-1$ and $k+1\leq l\leq n$ then $(z_{j}^{r})^{*}z_{l}^{n}=z_{l}^{n}(z_{j}^{r})^{*}$ if $1\leq r\leq n-1$ and $z_{l}^{n}(z_{j}^{n})^{*}=q (z_{j}^{n})^{*}z_{l}^{n}$ if $r=n,$ and hence the right-hand side of~\eqref{kfull} is invariant under these operators. Now let $k+1\leq j\leq n,$ if we define the norm $|\textbf{m}|=m_{1}+m_{2}+\dots+m_{n-k},$ then from the equations
$$
(z_{j}^{r})^{*}z_{l}^{n}=z_{l}^{n}(z_{j}^{r})^{*}\text{  if $j\neq l$ and $r\neq n$}
$$
$$
(z_{j}^{r})^{*}z_{l}^{n}=q z_{}^{n}(z_{j}^{r})^{*}\text{  if $j\neq l$ and $r= n$}
$$
$$
(z_{j}^{r})^{*}z_{j}^{n}=q z_{j}^{n}(z_{j}^{r})^{*}-(q^{-1}-q)\sum_{j<s}z_{s}^{n}(z_{s}^{r})^{*}\text{  if $r\neq n$}
$$
$$
(z_{j}^{n})^{*}z_{j}^{n}=q^{2} z_{j}^{n}(z_{j}^{n})^{*}+(1-q^{2})(I-\sum_{j<s}z_{s}^{n}(z_{s}^{n})^{*})
$$
and the fact that $z_{s}^{n}\in \textbf{I}$ if $j<s,$ it is easy to see by induction on $m\in \mathbb{Z}_{+},$ that the spaces $K_{m}:=\bigvee_{|\textbf{m}|\leq m} H_{\textbf{m}}$ are all invariant under the operators $\pi(z_{j}^{r})^{*}$, as $K_{0}=H$ is invariant with respect to them (hence establishing the case $m=0$). 
\\

The final case to consider is the column $ (z_{k}^{j})^{*}$ for $1\leq j\leq n.$ Notice that as $z_{k}^{n}z_{j}^{n}=q z_{j}^{n}z_{k}^{n}$ and we assumed $\pi(z_{k}^{n})|_{H}=e^{i\varphi}I,$ we see that $\pi(z_{k}^{n})|_{H_{\textbf{m}}}=e^{i\varphi}q^{|\textbf{m}|}I.$ As the set of operators $\pi(z_{k}^{1})^{*},\pi(z_{k}^{2}),\dots, \pi(z_{k}^{(n-1)})^{*}$ commute with the operators $\pi(z_{k+1}^{n}),\pi(z_{k+2}),\dots,\pi(z_{n}^{n}),$ (by~\eqref{zaa41}) we only need to prove that $\pi(z_{k}^{j})H\subset \bigvee_{\mathbb{Z}_{+}^{n-k}}H_{\textbf{m}},$ for $j=1,\dots,n-1.$ Let $P$ be the orthogonal projection onto the orthogonal complement of $\bigvee_{\textbf{m}\in \mathbb{Z}_{+}^{n-k}}H_{\textbf{m}}$ and let $v\in H.$ Then
$$||P  \pi(z_{k}^{j})^{*}v||=||P \pi(z_{k}^{j})^{*}\pi(z_{k}^{n}) v||=$$ 
$$
=||P(q\pi(z_{k}^{n})\pi(z_{k}^{j})^{*}-(q^{-1}-q)\sum_{j<s}\pi(z_{s}^{n})\pi(z_{s}^{j})^{*}) v||=q||P \pi(z_{k}^{n})\pi(z_{k}^{j})^{*} v||
$$
by what has already been proven. By induction, we have $$||P \pi(z_{k}^{j})^{*} v||=q^{r}||P \pi(z_{k}^{n})^{r}\pi(z_{k}^{j})^{*} v||.$$ 
By Lemma~\ref{pol}, $\pi(z_{k}^{j})^{*}$ is a contraction, and so we get

$$
\begin{array}{ccc}||P \pi(z_{k}^{j})^{*}v||\leq q^{r}||\pi(z_{k}^{j})^{*}||\cdot ||v|| & \text{for all $r\in \mathbb{Z}_{+}$}\end{array}
$$
and thus $||P \pi(z_{k}^{j})^{*} v||=0.$
\end{proof}
\begin{rem}
Notice that this lemma shows that by applying Lemma~\ref{linj} and Lemma~\ref{linj2} to $\textbf{I}=\{z_{k+1}^{n},\dots,z_{n}^{n}\}$ and $H,$ we get a natural unitary isometry $  \overline{\mathcal{A}(\textbf{I})\otimes H} \to K$, where $\mathcal{A}(\textbf{I})$ is equipped with the inner product $\langle\cdot,\cdot\rangle_{\textbf{I}}.$
This unitary isometry is the closure of the linear map determined by $a\otimes v\in\mathcal{A}(\textbf{I})\otimes H \mapsto \pi(a) v\in K.$ 
\end{rem}
We can now prove that $\pi'$ is irreducible. Assume that $H=H_{1}\oplus H_{2},$ where $H_{1}\bot H_{2}$ and both subspaces reducing $\pi'.$ Then $H_{1},H_{2}$ are actually invariant under $\textbf{Z}\backslash \textbf{I}$ as $\pi(z_{k}^{j})|_{H}=0$ for $1\leq j\leq n-1$ and $\pi(z_{k}^{n})|_{H}=e^{i\varphi}I$ and the invariance under the remaining operators follows from the way $\pi'$ is defined. Now, the same arguments that were used in Lemma~\ref{molly} to show invariance under $\textbf{Z}$ of the right hand side of~\eqref{kfull} can be used to show that the subspaces $\bigvee_{\textbf{m}\in\mathbb{Z}_{+}^{n-k}}\overline{\pi(z(\textbf{m}))H_{1}}$ and $\bigvee_{\textbf{m}\in\mathbb{Z}_{+}^{n-k}}\overline{\pi(z(\textbf{m})) H_{2}}$  are both invariant under $\textbf{Z}.$ We also have $\overline{\pi(z(\textbf{m}_{1}))H_{1}}\bot \overline{\pi(z(\textbf{m}_{2})) H_{2}}$ for all $\textbf{m}_{1},\textbf{m}_{2}\in \mathbb{Z}_{+}^{n-k}$ by similar arguments as in Lemma~\ref{molly}. By Lemma~\ref{molly}, we get $$K=\bigvee_{\textbf{m}\in\mathbb{Z}_{+}^{n-k}}\overline{\pi(z(\textbf{m})) H_{1}}\oplus\bigvee_{\textbf{m}\in\mathbb{Z}_{+}^{n-k}}\overline{\pi(z(\textbf{m})) H_{2}}$$ and therefore
$\left(\bigvee_{\textbf{m}\in\mathbb{Z}_{+}^{n-k}}\overline{\pi(z(\textbf{m})) H_{1}}\right)^{\bot}=\bigvee_{\textbf{m}\in\mathbb{Z}_{+}^{n-k}}\overline{\pi(z(\textbf{m})) H_{2}}.$
So both subspaces are actually invariant under $\textbf{Z}^{*}$ too, contradicting irreducibility of $\pi$.
\\

For $j+i\leq m-1,$ let $c_{ji}$ be the cycle $s_{j}s_{j+1}\dots s_{j+i}\in S_{m}.$ Let us denote the $*$-representation of $\mathbb{C}[SU_{m}]_{q}$ corresponding to $c_{ji}$ by $\pi_{ji}$ (see Definition~\ref{surep1}). In next part of the proof we are going to reconstruct the original $*$-representation $\pi$ using the lift of $\pi',$ so now we are working in $\mathbb{C}[SU_{2n}]_{q}$ instead. We are going to tensor our reduced and lifted representation $\pi'$  with a suitable $\pi_{ji},$ and the proof that this $*$-representation is isomorphic to $\pi$ relies strongly on being able to explicitly calculate the image $\pi_{ji}(t_{kl}),$ so we will now explain how to do this easily. By the definition of tensor product of representations of $\mathbb{C}[SU_{m}]_{q}$ we have
\begin{equation}\label{copout}
\pi_{ji}(t_{kl})=\sum_{k_{1},k_{2},\dots, k_{i}=1 }^{m}\pi_{j}(t_{k,k_{1}})\otimes \pi_{j+1}(t_{k_{1},k_{2}})\otimes \dots \otimes \pi_{j+i}(t_{k_{i},l})
\end{equation}
Notice that since $\pi_{a}(t_{rs})=\delta_{rs}I$ unless $(r,s)$ is one of the four pairs of integers $(a,a),(a+1,a),(a,a+1),(a+1,a+1),$ we have that a non-zero term in~\eqref{copout} corresponds to a sequence $k,k_{1},\dots,k_{i-1},l$ that can be visualized as paths on the grid
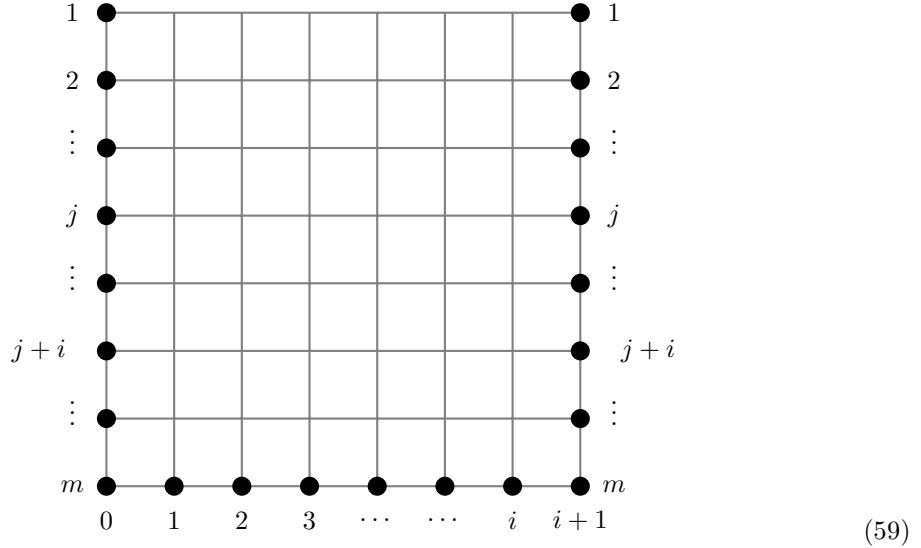
\begin{equation}\label{picture}
\begin{tikzpicture}[thick,scale=0.9]
\draw[step=1.0,black!50!white,thick] (1,1) grid (8,8);
\node at (0.5,1) {$m$};
\node at (0.5,2.2) {$\vdots$};
\node at (0,3) {$j+i$};
\node at (0.5,4.2) {$\vdots$};
\node at (0.5,5) {$j$};
\node at (0.5,6.2) {$\vdots$};
\node at (0.5,7) {$2$};
\node at (0.5,8) {$1$};
\node at (8.5,1) {$m$};
\node at (8.5,2.2) {$\vdots$};
\node at (9,3) {$j+i$};
\node at (8.5,4.2) {$\vdots$};
\node at (8.5,5) {$j$};
\node at (8.5,6.2) {$\vdots$};
\node at (8.5,7) {$2$};
\node at (1,0.5) {$0$};
\node at (8.5,8) {$1$};
\node at (2,0.5) {$1$};
\node at (3,0.5) {$2$};
\node at (4,0.5) {$3$};
\node at (5,0.5) {$\dots$};
\node at (6,0.5) {$\dots$};
\node at (7,0.5) {$i$};
\node at (8,0.5) {$i+1$};
\fill (1,1)  circle[radius=4pt];
\fill (2,1)  circle[radius=4pt];
\fill (3,1)  circle[radius=4pt];
\fill (4,1)  circle[radius=4pt];
\fill (5,1)  circle[radius=4pt];
\fill (6,1)  circle[radius=4pt];
\fill (7,1)  circle[radius=4pt];
\fill (8,1)  circle[radius=4pt];
\fill (1,2)  circle[radius=4pt];
\fill (1,3)  circle[radius=4pt];
\fill (1,4)  circle[radius=4pt];
\fill (1,5)  circle[radius=4pt];
\fill (1,6)  circle[radius=4pt];
\fill (1,7)  circle[radius=4pt];
\fill (1,8)  circle[radius=4pt];
\fill (8,2)  circle[radius=4pt];
\fill (8,3)  circle[radius=4pt];
\fill (8,4)  circle[radius=4pt];
\fill (8,5)  circle[radius=4pt];
\fill (8,6)  circle[radius=4pt];
\fill (8,7)  circle[radius=4pt];
\fill (8,8)  circle[radius=4pt];
\end{tikzpicture}
\end{equation}
starting at vertex $(k,0)$ on the left-hand side proceeding to $(k_{1},1)$ etc, and then ending at vertex $(l,i+1).$ These paths have the property that $k_{r}\neq k_{r+1}$ if and only if $k_{r}=j+r$ or $k_{r}=j+r+1,$ and in the first case, then $k_{r+1}=j+r+1$ and in the second case $k_{r+1}=j+r.$
As an example, when $m=8,j=2,i=4,$ a path for $t_{46}$ can be visualized as
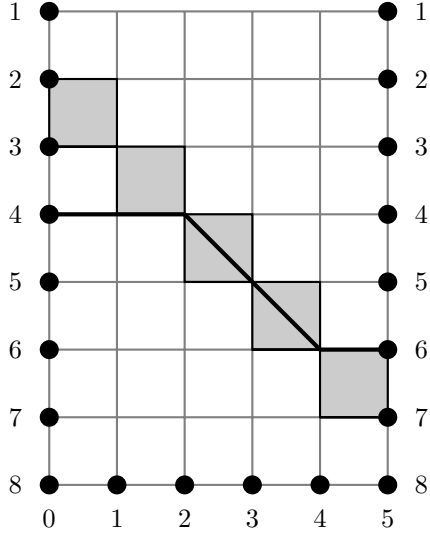
\begin{equation}\label{exam}
\begin{tikzpicture}[thick,scale=0.9]
\draw[step=1.0,black!50!white,thick] (1,1) grid (6,8);
\node at (0.5,1) {$8$};
\node at (0.5,2) {$7$};
\node at (0.5,3) {$6$};
\node at (0.5,4) {$5$};
\node at (0.5,5) {$4$};
\node at (0.5,6) {$3$};
\node at (0.5,7) {$2$};
\node at (0.5,8) {$1$};
\node at (6.5,1) {$8$};
\node at (6.5,2) {$7$};
\node at (6.5,3) {$6$};
\node at (6.5,4) {$5$};
\node at (6.5,5) {$4$};
\node at (6.5,6) {$3$};
\node at (6.5,7) {$2$};
\node at (6.5,8) {$1$};
\node at (1,0.5) {$0$};
\node at (2,0.5) {$1$};
\node at (3,0.5) {$2$};
\node at (4,0.5) {$3$};
\node at (5,0.5) {$4$};
\node at (6,0.5) {$5$};
\filldraw[fill=black!20!white, draw=black](2,6) -- (1,6) -- (1,7)--(2,7)--(2,6)--(1,6);
\filldraw[fill=black!20!white, draw=black](3,5) -- (2,5) -- (2,6)--(3,6)--(3,5)--(2,5);
\filldraw[fill=black!20!white, draw=black](4,4) -- (3,4) -- (3,5)--(4,5)--(4,4)--(3,4);
\filldraw[fill=black!20!white, draw=black](5,3) -- (4,3) -- (4,4)--(5,4)--(5,3)--(4,3);
\filldraw[fill=black!20!white, draw=black](6,2) -- (5,2) -- (5,3)--(6,3)--(6,2)--(5,2);
\fill (1,1)  circle[radius=4pt];
\fill (2,1)  circle[radius=4pt];
\fill (3,1)  circle[radius=4pt];
\fill (4,1)  circle[radius=4pt];
\fill (5,1)  circle[radius=4pt];
\fill (6,1)  circle[radius=4pt];
\fill (1,2)  circle[radius=4pt];
\fill (1,3)  circle[radius=4pt];
\fill (1,4)  circle[radius=4pt];
\fill (1,5)  circle[radius=4pt];
\fill (1,6)  circle[radius=4pt];
\fill (1,7)  circle[radius=4pt];
\fill (1,8)  circle[radius=4pt];
\fill (6,2)  circle[radius=4pt];
\fill (6,3)  circle[radius=4pt];
\fill (6,4)  circle[radius=4pt];
\fill (6,5)  circle[radius=4pt];
\fill (6,6)  circle[radius=4pt];
\fill (6,7)  circle[radius=4pt];
\fill (6,8)  circle[radius=4pt];
\draw[ultra thick] (1,5) -- (2,5) -- (3,5) -- (4,4) -- (5,3) -- (6,3);
\end{tikzpicture}
\end{equation}
Here the light shadowed boxes indicate the $*$-representations $\pi_{j},\dots, \pi_{j+i}.$ They are drawn so that, if $\pi_{l}$ is the $m$'th tensor-factor in $\pi_{ji},$ then the box, that is $l$'th from the top and $m$'th from the left is shaded gray. As
$$
\pi_{24}(t_{46})=\sum_{k_{1},\dots, k_{4}}^{8}\pi_{2}(t_{2k_{1}})\otimes \pi_{3}(t_{k_{1}k_{2}})\otimes \pi_{4}(t_{k_{2}k_{3}})\otimes \pi_{5}(t_{k_{3}k_{4}})\otimes \pi_{6}(t_{k_{4}6})
$$
the path chosen in~\eqref{exam} corresponds to $k_{1}=4,k_{2}=4,k_{3}=5,k_{4}=6$ and to the operator 

$$ \pi_{2}(t_{44})\otimes \pi_{3}(t_{44})\otimes \pi_{4}(t_{45})\otimes \pi_{5}(t_{56})\otimes \pi_{6}(t_{66})=$$
$$I\otimes T_{22}\otimes T_{12}\otimes T_{12}\otimes T_{11}.$$ 
We define an admissible path in $\pi_{ji}(t_{ml})$ as a sequence of integers $\{m,k_{1},k_{2},\dots, k_{i},l\}$ such that
$$
\pi_{i}(t_{mk_{1}})\otimes \pi_{i+1}(t_{k_{1}k_{2}})\otimes \cdots \otimes \pi_{i+j}(t_{k_{i}l})\neq 0.
$$
In $\pi_{24},$ it is not hard to see that the path in~\eqref{exam} corresponding to $\{4,4,4,5,6,6\}$ is the only admissible path. We also see that there is no possible way to join $m$ with $m-2,$ as we can traverse at most once upwards in a diagonal. This holds also in general.
\begin{lem}\label{coopout}
For $j+i\leq m-1,$ let $\pi_{ji}$ (resp. $\pi_{ji}^{op}$) be the $*$-representation of $\mathbb{C}[SU_{m}]_{q}$ corresponding to $s_{j}s_{j+1}\dots s_{j+i}\in S_{m+1}$ (resp. $s_{j+i}s_{j+i-1}\cdots s_{j}\in S_{m}$). Then for any $1\leq k,l \leq m,$ there is at most one non-zero term in the sum $$\pi_{ji}(t_{ml})=\sum_{k_{1}k_{2},\dots, k_{i} =1}^{m}\pi_{j}(t_{mk_{1}})\otimes \pi_{j+1}(t_{k_{1}k_{2}})\otimes \dots \otimes \pi_{j+i}(t_{k_{i-1}l})$$
(resp. $$\pi_{ji}^{op}(t_{ml})=\sum_{k_{i},k_{i-1},\dots, k_{1}=1 }^{m}\pi_{j+i}(t_{mk_{i}})\otimes \pi_{j+i-1}(t_{k_{i}k_{i-1}})\otimes \dots \otimes \pi_{j}(t_{k_{i-1}l}) )$$ and we have $\pi_{ji}(t_{ml})=0$ if $m-l\geq 2$ (resp. $\pi_{ji}^{op}(t_{ml})=0$ if $l-m\geq 2$).
\end{lem}

The above lemma can be easily seen to be true by drawing diagrams such as~\eqref{exam}, and it is not hard to make the pictorial intuition into a formal proof by, for instance, using induction on $i.$
\\

Hence by this lemma, it is very easy to calculate the action of $\pi_{ji}(t_{ml})$ explicitly as one only needs to draw the obvious path.
Another way to illustrate the action of $\pi_{ji}(t_{ml})$ is by using diagrams with arrows and hooks similar to the one that describes the action of the Fock representation of $\mathrm{Pol}(\mathrm{Mat}_{n})_{q}.$ Consider the row
$$
\begin{tikzpicture}[thick,scale=0.9]
\draw[step=1.0,black,thick] (1,1) grid (8,2);
\node at (1.5,0.5) {$2$};
\node at (2.5,0.5) {$\dots$};
\node at (3.5,0.5) {$j+1$};
\node at (4.5,0.5) {$\dots$};
\node at (5.5,0.5) {$j+i+1$};
\node at (6.5,0.5) {$\dots$};
\node at (7.5,0.5) {$m$};
\node at (0.5,1.5) {$1$};

\node at (8.5,1.5) {$m$};
\node at (1.5,2.4) {$1$};
\node at (2.5,2.4) {$\dots$};
\node at (3.5,2.4) {$j$};
\node at (4.5,2.4) {$\dots$};
\node at (5.5,2.4) {$j+i$};
\node at (6.5,2.4) {$\dots$};
\node at (7.5,2.4) {$m-1$};
\filldraw[fill=black!50!white, draw=black](1,2) -- (1,1) -- (2,1)--(2,2)--(1,2)--(1,1);
\filldraw[fill=black!50!white, draw=black](2,2) -- (2,1) -- (3,1)--(3,2)--(2,2)--(2,1);
\filldraw[fill=black!50!white, draw=black](6,2) -- (6,1) -- (7,1)--(7,2)--(6,2)--(6,1);
\filldraw[fill=black!50!white, draw=black](7,2) -- (7,1) -- (8,1)--(8,2)--(7,2)--(7,1);
\end{tikzpicture}
$$
where the dark gray squares again correspond to the trivial $*$-representation of $\mathbb{C}[SU_{2}]_{q}$ onto $\mathbb{C}$ and on the white squares the $*$-representation of $\mathbb{C}[SU_{2}]_{q}$ into $C^{*}(S)$ given by~\eqref{basrep}. Here the tensor factors are ordered from left to right, and we determine the image $\pi_{ji}(t_{kl})$ by connecting $k$ on the left or bottom side to the top or right side with the hooks and arrows.  So, in the previous example, the term corresponding to the admissible path $\pi_{24}(t_{t_{46}}),$ can be represented now as
$$
\begin{tikzpicture}[thick,scale=0.9]
\draw[step=1.0,black,thick] (1,1) grid (8,2);
\node at (1.5,0.5) {$2$};
\node at (2.5,0.5) {$3$};
\node at (3.5,0.5) {$4$};
\node at (4.5,0.5) {$5$};
\node at (5.5,0.5) {$6$};
\node at (6.5,0.5) {$7$};
\node at (7.5,0.5) {$8$};
\node at (0.5,1.5) {$1$};

\node at (8.5,1.5) {$8$};
\node at (1.5,2.4) {$1$};
\node at (2.5,2.4) {$2$};
\node at (3.5,2.4) {$3$};
\node at (4.5,2.4) {$4$};
\node at (5.5,2.4) {$5$};
\node at (6.5,2.4) {$6$};
\node at (7.5,2.4) {$7$};
\filldraw[fill=black!50!white, draw=black](1,2) -- (1,1) -- (2,1)--(2,2)--(1,2)--(1,1);
\filldraw[fill=black!50!white, draw=black](7,2) -- (7,1) -- (8,1)--(8,2)--(7,2)--(7,1);
\draw [->] (0.5+3,0.05+1)--(0.5+3,0.5+1)--(0.95+3,0.5+1);\draw [->] (0.05+4,0.5+1)--(0.95+4,0.5+1);\draw [->] (0.05+5,0.5+1)--(0.95+5,0.5+1);\draw [->] (0.05+6,0.5+1)--(0.5+6,0.5+1)--(0.5+6,0.95+1);
\end{tikzpicture}
$$
If we instead consider $s_{j+i}s_{j+i-1}\dots s_{j}\in S_{m},$ we can represent this similarly, but with a vertical block instead, and where the ordering of factors is made from top to bottom. If we apply $\tau_{\varphi}$ to one of the $C^{*}(S)$ factors, then as before, we can calculate the new $*$-representation by putting in a light gray square in that factor's place.
We leave it up to the readers to convince themselves that these two ways of doing calculations actually yield the same result.
\\

Some remarks about notations regarding tensor products: the Hilbert space $\ell^{2}(\mathbb{Z}_{+})^{\otimes j}$ will be given the basis $\{e_{\textbf{m}}\}_{\textbf{m}\in\mathbb{Z}_{+}^{j}},$ where $\textbf{m}=\{m_{1},m_{2},\dots,m_{j}\}\in\mathbb{Z}_{+}^{j}$ is a multi-index and $e_{\textbf{m}}=e_{m_{1}}\otimes e_{m_{2}}\otimes \dots\otimes e_{m_{j}}\in \ell^{2}(\mathbb{Z}_{+})^{\otimes j},$ where $\{e_{m}\}_{m=0}^{\infty}$ is the standard orthonormal basis of $\ell^{2}(\mathbb{Z}_{+}).$
For $m\in \mathbb{Z}_{+}$ and $1\leq l\leq j,$ let $(m)_{l}\in \mathbb{Z}_{+}^{k}$ be the multi-index with value $m$ at the $l$'th index and equal to zero otherwise. Let also $\textbf{0}=\{0,0,\dots,0\}\in \mathbb{Z}^{j}_{+},$ so that $e_{\textbf{0}}=e_{0}\otimes e_{0}\otimes \dots \otimes e_{0}\in\ell^{2}(\mathbb{Z}_{+})^{\otimes j}.$
\\

Let $\beta_{\varphi}:\mathrm{Pol}(\mathrm{Mat}_{n-1})_{q}\rightarrow \mathrm{Pol}(\mathrm{Mat}_{n-1})_{q}$ be the automorphism defined by
\begin{equation}\label{X}\beta_{\varphi}(z_{j}^{l})=\begin{cases} 
z_{j}^{i} & \text{ if $j\neq k-1$}\\
e^{i\varphi}z_{j}^{l} & \text{ if $j=k-1$}
\end{cases}
\end{equation}
By induction, we can now assume that $\pi'\circ \beta_{\varphi},$ where $\pi'$ is defined by~\eqref{pii}, can be lifted to a $*$-representation $\mu:\mathbb{C}[SU_{2(n-1)}]_{q}\rightarrow B(H)$ such that \begin{equation}\label{Y}\mu\circ \zeta=\pi'\circ \beta_{\varphi}.\end{equation}
We can then lift $\mu$ to a $*$-representation $\tilde{\pi}$ of $\mathbb{C}[SU_{2n}]_{q}$ using the map $\psi:\mathbb{C}[SU_{2n}]_{q}\rightarrow \mathbb{C}[SU_{2(n-1)}]_{q}$ given as
\begin{equation}\label{mappp}
\psi(t_{k,j})=
\begin{cases}
t_{k-1,j-1}& \text{ if $2\leq k,j\leq 2n-1$}\\
\delta_{k,j}& \text{otherwise}
\end{cases}
\end{equation}
so that \begin{equation}\label{Z}\tilde{\pi}:=\mu\circ \psi.\end{equation}
\begin{rem}
A sudden involvement of $\beta_{\varphi}$ may seem a bit unmotivated, but it is logical, given our informal discussion in section $3.$ If we look at the paths starting like
$$
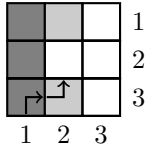
\begin{tikzpicture}[thick,scale=0.5]
\filldraw[fill=black!50!white, draw=black](1,0) -- (0,0) -- (0,1)--(1,1)--(1,0)--(0,0);
\filldraw[fill=black!50!white, draw=black](1,1) -- (0,1) -- (0,2)--(1,2)--(1,1)--(0,1);
\filldraw[fill=black!50!white, draw=black] (1,2) -- (0,2) -- (0,3)--(1,3)--(1,2)--(0,2);
\filldraw[fill=black!20!white, draw=black](2,0) -- (1,0) -- (1,1)--(2,1)--(2,0)--(1,0);
\draw (2,1) -- (1,1) -- (1,2)--(2,2)--(2,1)--(1,1);
\filldraw[fill=black!20!white, draw=black](2,2) -- (1,2) -- (1,3)--(2,3)--(2,2)--(1,2);
\draw (3,0) -- (2,0) -- (2,1)--(3,1)--(3,0)--(2,0);
\draw (3,1) -- (2,1) -- (2,2)--(3,2)--(3,1)--(2,1);
\draw (3,2) -- (2,2) -- (2,3)--(3,3)--(3,2)--(2,2);
\draw [->] (0.5,0.05)--(0.5,0.5)--(0.95,0.5);
\draw [->] (0.05+1,0.5)--(0.5+1,0.5)--(0.5+1,0.95);
\node at (0.5,-0.5) {$1$};
\node at (1.5,-0.5) {$2$};
\node at (2.5,-0.5) {$3$};
\node at (3.5,2.5) {$1$};
\node at (3.5,0.5) {$3$};
\node at (3.5,1.5) {$2$};
\end{tikzpicture}
$$
then when we use the scheme in Section $5.1$ to express the operators $\pi'(z_{1}^{1}),\pi'(z_{1}^{2})$ ($\pi'$ corresponding to the upper-right $2\times 2$ square) they get multiplied by $e^{-i\varphi}$ coming from they gray square at $(2,3)$. So if we want to get the $*$-representation corresponding to the upper right $2\times 2$ square, then we must multiply the first column in $\mathrm{Pol}(\mathrm{Mat}_{2})_{q}$ by $e^{i\varphi}$ in order to cancel this factor. 
\end{rem}
Now define a $*$-representation of $\mathbb{C}[\mathrm{SU}_{2n}]_{q}$ by the formula
\begin{equation}\label{diddi}
\lambda=(\tau_{\varphi}\circ \pi_{n+k-1})\otimes \pi_{n+k}\otimes\dots\otimes \pi_{2n-1}:\mathbb{C}[SU_{2n}]_{q}\rightarrow B(\ell^{2}(\mathbb{Z}_{+}))^{\otimes n-k}
\end{equation}
corresponding to the diagram
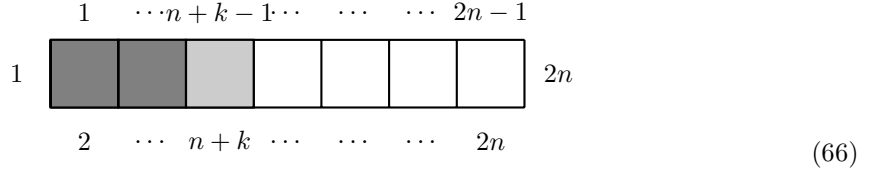
\begin{equation}\label{ulrik2}
\begin{tikzpicture}[thick,scale=0.9]
\draw[step=1.0,black,thick] (1,1) grid (8,2);
\node at (1.5,0.5) {$2$};
\node at (2.5,0.5) {$\dots$};
\node at (3.5,0.5) {$n+k$};
\node at (4.5,0.5) {$\dots$};
\node at (5.5,0.5) {$\dots$};
\node at (6.5,0.5) {$\dots$};
\node at (7.5,0.5) {$2n$};
\node at (0.5,1.5) {$1$};

\node at (8.5,1.5) {$2n$};
\node at (1.5,2.4) {$1$};
\node at (2.5,2.4) {$\dots$};
\node at (3.5,2.4) {$n+k-1$};
\node at (4.5,2.4) {$\dots$};
\node at (5.5,2.4) {$\dots$};
\node at (6.5,2.4) {$\dots$};
\node at (7.5,2.4) {$2n-1$};
\filldraw[fill=black!50!white, draw=black](1,2) -- (1,1) -- (2,1)--(2,2)--(1,2)--(1,1);
\filldraw[fill=black!50!white, draw=black](2,2) -- (2,1) -- (3,1)--(3,2)--(2,2)--(2,1);
\filldraw[fill=black!20!white, draw=black](3,2) -- (3,1) -- (4,1)--(4,2)--(3,2)--(3,1);
\end{tikzpicture}
\end{equation}
and consider the representation $$\lambda\otimes \tilde{\pi}:\mathbb{C}[SU_{2n}]_{q}\rightarrow B(\ell^{2}(\mathbb{Z}_{+}))^{\otimes n-k}\otimes B(H),$$
If we let $\Delta:=(\lambda\otimes \tilde{\pi})\circ\zeta,$ then the aim is to show that $\Delta\cong \pi.$
\begin{lem}\label{did}
We have $\ker \Delta(z_{k}^{n})^{*}=0$ and 
\begin{equation}\label{qqq}
\cap_{j=k+1}^{n}\ker \Delta(z_{j}^{n})^{*}=\langle e_{\textbf{0}}\rangle\otimes H
\end{equation}
and for $z_{m}^{j}\in \textbf{Z}\backslash\textbf{I}$ (recall $\textbf{I}=\{z_{k+1}^{n},\dots, z_{n}^{n}\}$) we have
\begin{equation}\label{qq}
\Delta(z_{m}^{j}) (e_{\textbf{0}}\otimes v)= e_{\textbf{0}}\otimes \pi(z_{m}^{j})v.
\end{equation}
Moreover, $\ell^{2}(\mathbb{Z}_{+})^{\otimes n-k}\otimes H$ is the closed linear span of vectors
$$
\begin{array}{ccc}\Delta(z(\textbf{m}))(e_{\textbf{0}}\otimes v),& v\in H, & \textbf{m}\in \mathbb{Z}_{+}^{n-k}.\end{array}  
$$
\end{lem}
\begin{proof}
We start with~\eqref{qq}. We have $\zeta(z_{m}^{j})=(-q)^{m-n}t_{m+n,j+n}$ (see~\eqref{zetaa}) and hence 
\begin{equation}\label{ccc}
\Delta(z_{m}^{j})=(-q)^{m-n}\sum_{r=1}^{2n}\lambda(t_{m+n,r})\otimes \tilde{\pi}(t_{r,j+n})
\end{equation}
If $1\leq m\leq k-2,$ then $\lambda(t_{m+n,r})=\delta_{m+n,r}I$ so that, by~\eqref{X},~\eqref{Y} and~\eqref{Z}, the right hand side of~\eqref{ccc} becomes
$$
\Delta(z_{m}^{j})=(-q)^{m-n}I\otimes \tilde{\pi}(t_{m+n,j+n})=(-q)^{m-n}I\otimes (\mu\circ \psi)(t_{m+n,j+n})=
$$
$$
=(-q)^{m-n}I\otimes \mu(t_{m+n-1,j+n-1})=(-q)^{-1}I\otimes (\pi'\circ \beta_{\varphi})(z_{m}^{j})=
$$
$$
=(-q)^{-1}I\otimes \pi'(z_{m}^{j})=I\otimes \pi(z_{m}^{j})
$$
proving~\eqref{qq} in this case. If $m=k-1,$ then similar arguments give $\lambda(t_{k-1+n,r})=e^{-i\varphi}\delta_{k-1+n,r}I ,$ and so
$$
\Delta(z_{k-1}^{j})=e^{-i\varphi}(-q)^{k-1-n}I\otimes \tilde{\pi}(t_{k-1+n,j+n})=
$$
$$
=e^{-i\varphi}(-q)^{k-1-n}I\otimes (\mu\circ \psi)(t_{k-1+n,j+n})=e^{-i\varphi}(-q)^{k-1-n}I\otimes \mu(t_{k-1+n-1,j+n-1})=
$$
$$
=e^{-i\varphi}(-q)^{-1}I\otimes (\pi'\circ \beta_{\varphi})(z_{k-1}^{j})=(-q)^{-1}I\otimes \pi'(z_{k-1}^{j})=I\otimes \pi(z_{k-1}^{j})
$$
 proving~\eqref{qq}.
\\

When $m=k,$ then as $\tilde{\pi}(t_{r,2n})=\delta_{r,2n}I,$ we obtain by~\eqref{ccc} that
$$
\Delta(z_{k}^{n})=(-q)^{k-n}\lambda(t_{k+n,2n})\otimes I=e^{i\varphi}D_{q}\otimes D_{q}\otimes \dots \otimes D_{q}\otimes I
$$
(where $D_{q}$ is given by) proving~\eqref{qq} and that $\ker \Delta(z_{k}^{n})^{*}=0$. By Lemma~\ref{miss} we get $\Delta(z_{j}^{n})=0$ for $1\leq j \leq k-1.$ As $\pi(z_{j}^{n})=0$ for those values of $j,$ we obtain 
$$\Delta(z_{j}^{n})(e_{\textbf{0}}\otimes v) =0=e_{\textbf{0}}\otimes \pi(z_{j}^{n})v$$
and so~\eqref{qq} holds in this case.
\\

If $1\leq j<n$ then 
$$
\Delta(z_{k}^{j})=(-q)^{k-n}\sum_{r=1}^{2n}\lambda(t_{k+n,r})\otimes \tilde{\pi}(t_{r,j+n})=
$$
$$
(-q)^{k-n}\sum_{r=1}^{2n-1}\lambda(t_{k+n,r})\otimes \tilde{\pi}(t_{r,j+n})+\lambda(t_{k+n,2n})\otimes \tilde{\pi}(t_{2n,j+n})=
$$
$$
=(-q)^{k-n}\sum_{r=1}^{2n-1}\lambda(t_{k+n,r})\otimes \tilde{\pi}(t_{r,j+n}).
$$
A calculation using~\eqref{ulrik2} gives that unless $r=2n,$ the operator $\lambda(t_{k+n,r}),$ if non-zero, will contain $T_{11}$ as a factor, thus annihilates $e_{\textbf{0}}.$ Hence
$$
(-q)^{k-n}\sum_{r=2}^{n+k-1}\lambda(t_{k+n,r})\otimes \tilde{\pi}(t_{r,j+n})(e_{\textbf{0}}\otimes v)=0
$$
for all $v\in H,$ and thus equal to $e_{\textbf{0}}\otimes \pi(z_{k}^{j})v$ by~\eqref{zero}.
\\

Finally, if $z_{m}^{j}\in \textbf{Z}\backslash \{z_{k+1}^{n},\dots,z_{n}^{n}\} $ with $ m \geq k+1,$ then we have
$$
\Delta(z_{m}^{j})=(-q)^{k-n}\sum_{r=1}^{2n}\lambda(t_{m+n,r})\otimes \tilde{\pi}(t_{r,j+n})=
$$
$$
=(-q)^{m-n}\sum_{r=1}^{2n-1}\lambda(t_{m+n,r})\otimes \tilde{\pi}(t_{r,j+n}).
$$
Again, we see from~\eqref{ulrik2} that if $r\neq m+n-1,$ then we have a $S^{*}C_{q}$ factor in $\lambda(t_{m+n,r})$ if it is non-zero, and in the case $r=m+n-1,$ we have $\lambda(t_{m+n,r})e_{\textbf{0}}=e_{\textbf{0}}.$ We obtain
$$
(-q)^{m-n}\sum_{r=1}^{2n-1}\lambda(t_{m+n,r})\otimes \tilde{\pi}(t_{r,j+n})(e_{\textbf{0}}\otimes v)=
$$
$$
=(-q)^{m-n}\lambda(t_{m+n,m+n-1})\otimes\tilde{\pi}(t_{m+n-1,j+n})(e_{\textbf{0}}\otimes v)=
$$
$$
=e_{\textbf{0}}\otimes ((-q)^{m-n}\tilde{\pi}(t_{m+n-1,j+n}))v=e_{\textbf{0}}\otimes ((-q)^{m-n}\mu\circ\psi(t_{m+n-1,j+n}))v=
$$
$$
=e_{\textbf{0}}\otimes ((-q)^{m-n}\mu(t_{m+n-2,j+n-1})v)=e_{\textbf{0}}\otimes ((-q)^{(m-1)-(n-1)}\mu(t_{(m-1)+(n-1),j+(n-1)})v)=
$$
$$
=e_{\textbf{0}}\otimes (\pi\circ \beta_{\varphi})(z_{m-1}^{j}))v=e_{\textbf{0}}\otimes (\pi(z_{m-1}^{j}))v=
$$
$$
=e_{\textbf{0}}\otimes \pi(z_{m}^{j})v.
$$
and hence~\eqref{qq} is proven for all $z_{j}^{l}\in \textbf{Z}\backslash \textbf{I}.$
\\

Now we prove the second claim. We have
$$\Delta(z_{m}^{n})=(-q)^{m-n}\sum_{r=1}^{2n}\lambda(t_{m+n,r})\otimes \tilde{\pi}(t_{r,2n})=$$
$$
=(-q)^{m-n}\lambda(t_{m+n,2n})\otimes I
$$
as $\tilde{\pi}(t_{r,2n})=\delta_{r,2n}I.$ For $k+1\leq m\leq n$ we use~\eqref{ulrik2} to get $$(-q)^{m-n}\lambda(t_{m+n,2n})\otimes I=\underbrace{I\otimes \dots \otimes I}_{\text{$m-k+1$ times}}\otimes C_{q}S\otimes \underbrace{D_{q}\otimes  \dots \otimes D_{q}}_{\text{$n-m$ times}}\otimes I$$
and hence~\eqref{qqq}. As $\Delta(z(\textbf{m})) (e_{\textbf{0}}\otimes v)$ is a nonzero multiple of $e_{\textbf{m}}\otimes v,$ and these vectors obviously span $\ell^{2}(\mathbb{Z}_{+})^{\otimes n-k}\otimes H,$ we obtain the last statement.  
\end{proof}
\begin{prop}\label{robert}
 The $*$-representations $\Delta$ and $\pi$ are equivalent by the unitary isometry $U:K\rightarrow \ell^{2}(\mathbb{Z}_{+})^{\otimes n-k}\otimes H$ given by
\begin{equation}\label{bosch}
\pi(z(\textbf{m}))v\mapsto \Delta(z(\textbf{m}))(e_{\textbf{0}}\otimes v)
\end{equation}
for $v\in H$ and $\textbf{m}\in \mathbb{Z}_{+}^{n-k}.$
\end{prop}
\begin{proof}
By Lemma~\ref{linj} and Lemma~\ref{linj2}, for all $v\in H$ $$||\pi(z(\textbf{m})) v||^{2}=\langle z(\textbf{m}),z(\textbf{m}) \rangle_{\textbf{I}}||v||^{2}$$ and $\langle z(\textbf{m}),z(\textbf{m}) \rangle_{\textbf{I}}\neq 0.$ Hence $\pi(z(\textbf{m})) v=\pi(z(\textbf{m}))u$ $u,v\in H$ if and only if $v=u.$ It then follows from Proposition~\ref{molly} that $U$ is well defined on the dense subspace, and hence if we can show that $U$ is an isometry on this dense subspace, then it can be extended to an isometry on the whole of $K.$ By applying Lemma~\ref{linj} again to both $(H,\textbf{I})$ and $(e_{\textbf{0}}\otimes H,\textbf{I}),$ we see that
$$
\langle \pi(z(\textbf{m}))v,\pi(z(\textbf{m}')) u\rangle=\gamma_{\textbf{I}}(z(\textbf{m}')^{*}z(\textbf{m}))\langle v,u\rangle=
$$
$$
=\gamma_{\textbf{I}}(z(\textbf{m}')^{*}z(\textbf{m}))\langle e_{\textbf{0}}\otimes v,e_{\textbf{0}}\otimes u\rangle=
$$
$$
=\langle \Delta(z(\textbf{m}))(e_{\textbf{0}}\otimes v),\Delta(z(\textbf{m}')) (e_{\textbf{0}}\otimes u)\rangle=
$$
$$
=\langle U (\pi(z(\textbf{m}))v),U(\pi(z(\textbf{m}')) u)\rangle.
$$
Hence, by linearity, we can extend $U$ to an isometry. By Lemma~\ref{did}, we can now conclude that $U$ is surjective and thus it follows that $U$ is a unitary isometry.
\\

To prove that $U$ induces an isomorphism we only need to prove that $U\pi(z_{m}^{j})= \Delta(z_{m}^{j})U$ for all $z_{m}^{j}\in \textbf{Z}.$ Note that this is true by the construction of $U$ for the elements in $\textbf{I}.$ By Lemma~\ref{did}, for $z_{m}^{j}\in(\textbf{Z}\backslash\textbf{I})\cap \{z_{1}^{n},\dots z_{k-1}^{n}\}$ we have $\pi(z_{m}^{j})=0$ and $\Delta(z_{m}^{j})=0,$ so obviously in this case $U\pi(z_{m}^{j})=\Delta(z_{m}^{j})U.$ If $z_{m}^{j}\in\textbf{Z}\backslash(\textbf{I}\cup \{z_{1}^{n},\dots z_{k-1}^{n}\}),$ then first we apply the same arguments that were used in the proof of Lemma~\ref{molly} to see that $$z_{m}^{j}z(\textbf{m})=\sum_{i} z(\textbf{m}_{i})b_{i}$$ for some $b_{i}$ in the unital algebra generated by $\textbf{Z}\backslash\textbf{I}.$ Then
$$
U \pi(z_{m}^{j})\pi(z(\textbf{m})) v=U\sum_{i} \pi(z(\textbf{m}_{i}))\pi(b_{i}) v=
$$
$$
=\sum_{i} \Delta(z(\textbf{m}_{i})) (U\pi(b_{i}) v)=
$$
$$
=\sum_{i} \Delta(z(\textbf{m}_{i}))(e_{\textbf{0}}\otimes(\pi(b_{i})v))=\left(\sum_{i} \Delta(z(\textbf{m}_{i})b_{i})\right)\cdot(e_{\textbf{0}}\otimes v)=
$$
$$
=\Delta(z_{m}^{j} (z(\textbf{m})) (e_{\textbf{0}}\otimes v))= \Delta(z_{m}^{j}) U(\pi(z(\textbf{m}))v)
$$
where in the fourth equality we used~\eqref{qq}. This concludes the proof of case $\textbf{A}$.
\end{proof}
\subsection{The Case $\textbf{B}$}
Assume now that the irreducible $*$-representation $\pi:\mathrm{Pol}(\mathrm{Mat}_{n})_{q}\to B(K)$ has the property that $\ker \pi(z_{n}^{k})^{*}\neq \{0\}$ and $\ker \pi(z_{k}^{n})^{*}\neq \{0\}$ for $1\leq k\leq n.$ In this case the reduction from $n$ to $n-1$ is much more straightforward. We will also be re-using many of the arguments from the case $\textbf{A},$ so at some places in the proofs, we will simply refer to Section $5.3.$ 
\\

Throughout this section, we let 
$$
\textbf{I}=\{z_{1}^{n},\dots , z_{n-1}^{n},z_{n}^{n},\dots, z_{n}^{1}\}
$$
and, as before, we use the multi-index notation: for $\textbf{m}=(m_{1},\dots, m_{2n-1})\in \mathbb{Z}^{2n-1}$
$$
z(\textbf{m})=(z_{1}^{n})^{m_{1}}\cdot\dots \cdot (z_{n-1}^{n})^{m_{n-1}}(z^{n}_{n})^{m_{n}}\cdot \dots \cdot (z_{n}^{1})^{m_{2n-1}}.
$$
We start by proving the following:

\begin{lem}
The subspace $H:=(\cap_{k=1}^{n}\ker \pi(z_{n}^{k})^{*})\cap (\cap_{k=1}^{n}\ker \pi(z_{k}^{n})^{*})$ is non-trivial.
\end{lem}
\begin{proof}
For the simplicity, we suppress $\pi$ and write simply $a$ for the image $\pi(a),a\in \mathrm{Pol}(\mathrm{Mat}_{n})_{q}.$
Assume that $(\cap_{k=1}^{n}\ker (z_{n}^{k})^{*})\cap (\cap_{k=1}^{n}\ker (z_{k}^{n})^{*})=\{0\}$ and let $0\leq m,j < n$ be integers such that
$$(\cap_{k=m+1}^{n}\ker (z_{n}^{k})^{*})\cap (\cap_{k=j+1}^{n}\ker (z_{k}^{n})^{*})\neq \{0\}$$
but $$(\cap_{k=m}^{n}\ker (z_{n}^{k})^{*})\cap (\cap_{k=j+1}^{n}\ker (z_{k}^{n})^{*})=\{0\}$$ if $m>0$ and $$(\cap_{k=m+1}^{n}\ker (z_{n}^{k})^{*})\cap (\cap_{k=j}^{n}\ker (z_{k}^{n})^{*})=\{0\}$$ if $j>0.$
To be more precise, we assume that $m>0$ and let $$J:=(\cap_{k=m+1}^{n}\ker (z_{n}^{k})^{*})\cap (\cap_{k=j+1}^{n}\ker (z_{k}^{n})^{*}).$$
Then as $\ker (z_{m}^{n})^{*}\cap J=0$ we get, by the same argument as in Lemma~\ref{polly}, that $J$ reduces $z_{m}^{n}$ and $z_{m}^{n}|_{J}$ is unitary. This shows that the operator $$R:=z_{m}^{n}(z_{m}^{n})^{*}-(z_{m}^{n})^{*}z_{m}^{n}$$ has a non-trivial kernel, as it clearly contains $J.$ 
\\

\textit{Claim: $\ker R$ is reducing for $\pi.$} \\
As $\pi$ is irreducible, the claim would give $R=0$ and hence $z_{m}^{n}$ is normal. By using Fuglede-Putnam's theorem in a similar way as in the proof of Proposition~\ref{l2}, we can deduce that the normal partial isometry $U$ in the polar decomposition of $z_{m}^{n}$ is commuting with the image of $\mathrm{Pol}(\mathrm{Mat}_{n})_{q}$ under $\pi.$ This gives that $K=\ker U \oplus U K$ is a decomposition of $K$ into orthogonal subspaces, both reducing $\pi.$ Hence one of the subspaces must be trivial. But if $\ker U$ is trivial, then the kernel of $(z_{m}^{n})^{*}$ must be trivial, contradicting our assumption on $\pi.$ If $UK$ is trivial, then $U=0$ and hence $(z_{m}^{n})^{*}=0,$ but this contradicts $(\cap_{k=m+1}^{n}\ker (z_{n}^{k})^{*})\cap (\cap_{k=j+1}^{n}\ker (z_{k}^{n})^{*})\neq \{0\}.$ 
\\

Proof of the claim: It is easy to see that $z_{k}^{l}$ with $k>m$ and $l<n$ commutes with $R.$ As $R$ is self adjoint, then $z_{k}^{l}$ must reduce its kernel. Similarly, it is easy to see that $z_{k}^{n}$ with $k<m$ commutes with $R.$ We can then use the arguments from Lemma~\ref{miss} to see that the restriction of such $z_{k}^{n}$ to $\ker R$ is zero. Now consider $z_{m}^{k}$ with $k<n.$ Then
$$
R z_{m}^{k}=z_{m}^{n}(z_{m}^{n})^{*}z_{m}^{k}-(z_{m}^{n})^{*}z_{m}^{n}z_{m}^{k}=z_{m}^{n}(z_{m}^{n})^{*}z_{m}^{k}-q^{-1}(z_{m}^{n})^{*}z_{m}^{k}z_{m}^{n}=
$$
$$
=z_{m}^{n}\left(q z_{m}^{k}(z_{m}^{n})^{*}-(q-q^{-1})\sum_{l=m+1}^{n}z_{l}^{k}(z_{l}^{n})^{*}\right)-
$$
$$
-q^{-1}\left(q z_{m}^{k}(z_{m}^{n})^{*}-(q-q^{-1})\sum_{l=m+1}^{n}z_{l}^{k}(z_{l}^{n})^{*}\right)z_{m}^{n}=
$$
$$
=(qz_{m}^{n}z_{m}^{k}(z_{m}^{n})^{*}-z_{m}^{k}(z_{m}^{n})^{*}z_{m}^{n})-
$$
$$
-(q-q^{-1})\left(z_{m}^{n}\sum_{l=m+1}^{n}z_{l}^{k}(z_{l}^{n})^{*}-q^{-1}\sum_{l=m+1}^{n}z_{l}^{k}(z_{l}^{n})^{*}z_{m}^{n}\right).
$$
We consider these two terms separately. The first term reduces to $$qz_{m}^{n}z_{m}^{k}(z_{m}^{n})^{*}-z_{m}^{k}(z_{m}^{n})^{*}z_{m}^{n}=z_{m}^{k}z_{m}^{n}(z_{m}^{n})^{*}-z_{m}^{k}(z_{m}^{n})^{*}z_{m}^{n}=z_{m}^{k}R.$$ For the second one, we obtain
$$
z_{m}^{n}\sum_{l=m+1}^{n}z_{l}^{k}(z_{l}^{n})^{*}-q^{-1}\sum_{l=m+1}^{n}z_{l}^{k}(z_{l}^{n})^{*}z_{m}^{n}=
$$
$$
=z_{m}^{n}\sum_{l=m+1}^{n}z_{l}^{k}(z_{l}^{n})^{*}-\sum_{l=m+1}^{n}z_{l}^{k}z_{m}^{n}(z_{l}^{n})^{*}=
$$
$$
=z_{m}^{n}\sum_{l=m+1}^{n}z_{l}^{k}(z_{l}^{n})^{*}-\sum_{l=m+1}^{n}z_{m}^{n}z_{l}^{k}(z_{l}^{n})^{*}=0
$$
and thus $z_{m}^{k}R=R z_{m}^{k}.$ For $k=n,$ we get
$$R z_{m}^{n}=z_{m}^{n}((z_{m}^{n})^{*}z_{m}^{n})-((z_{m}^{n})^{*}z_{m}^{n})z_{m}^{n}=$$
$$
=z_{m}^{n}\left(q^{2}z_{m}^{n}(z_{m}^{n})^{*}+(1-q^{2})\left(I-\sum_{l=m+1}^{n}z_{l}^{n}(z_{l}^{n})^{*}\right)\right)-
$$
$$
-\left(q^{2}z_{m}^{n}(z_{m}^{n})^{*}+(1-q^{2})\left(I-\sum_{l=m+1}^{n}z_{l}^{n}(z_{l}^{n})^{*}\right)\right)z_{m}^{n}
$$
and as $z_{m}^{n}$ commutes with $z_{l}^{n}(z_{l}^{n})^{*}$ for $l=m+1,\dots,n$ (by~\eqref{zaa1} and~\eqref{zaa4}), this sum reduces to $q^{2}z_{m}^{n}z_{m}^{n}(z_{m}^{n})^{*}-q^{2}z_{m}^{n}(z_{m}^{n})^{*}z_{m}^{n}=q^{2}z_{m}^{n}R$ and hence $\ker R$ is also reducing $z_{m}^{n}.$ Finally, let $k<m$ and $l<n$ and let $v\in \ker R.$ We have $$R z_{k}^{l} v=z_{m}^{n}(z_{m}^{n})^{*}z_{k}^{l}v-(z_{m}^{n})^{*}z_{m}^{n}z_{k}^{l} v=(z_{m}^{n}z_{k}^{l})(z_{m}^{n})^{*} v-(z_{m}^{n})^{*}(z_{m}^{n}z_{k}^{l}) v=$$
$$
=(z_{k}^{l}z_{m}^{n}-(q^{-1}-q)z_{m}^{l}z_{k}^{n})(z_{m}^{n})^{*} v-(z_{m}^{n})^{*}(z_{k}^{l}z_{m}^{n}-(q^{-1}-q)z_{m}^{l}z_{k}^{n})  v
$$
$$
=(z_{k}^{l}z_{m}^{n})(z_{m}^{n})^{*} v-(z_{m}^{n})^{*}(z_{k}^{l}z_{m}^{n}) v=z_{k}^{l}(z_{m}^{n}(z_{m}^{n})^{*}) v-z_{k}^{l}((z_{m}^{n})^{*}z_{m}^{n})  v=
$$
$$
z_{k}^{l}A v=0
$$
where the fourth equality follows from $z_{k}^{n} v=0$ and $z_{k}^{n}(z_{m}^{n})^{*}v=(z_{m}^{n})^{*}z_{k}^{n} v=0.$ The case for $(z_{k}^{l})^{*}$ is treated similarly.
\end{proof}
\begin{lem}
$H$ is a reducing subspace for $\pi(z_{j}^{i}),z_{k}^{j}\in\textbf{Z}_{n}^{n}.$ Moreover, the map 
\begin{equation}\label{beans}
\begin{array}{cc}\pi':z_{j}^{i}\mapsto \pi(z_{j}^{i})|_{H},&i,j=1,\dots, n-1,\end{array}\end{equation} extends to an irreducible $*$-representation of $\mathrm{Pol}(\mathrm{Mat}_{n-1})_{q}.$
\end{lem}
\begin{proof}
As before, we suppress $\pi$ and write simply $a$ for the image $a\in \mathrm{Pol}(\mathrm{Mat}_{n})_{q}.$
To see that $H$ is invariant under the action of $(\textbf{Z}_{n}^{n})^{*},$ pick $(z_{k}^{j})^{*}\in \textbf{Z}_{n}^{n}.$ Notice that by~\eqref{zaa1*}-\eqref{zaa3*}, any $(z_{m}^{l})^{*}\in\textbf{I}^{*}$ either commutes or $q$-commutes with $(z_{k}^{j})^{*},$ or $(z_{m}^{l})^{*}$ and $(z_{k}^{j})^{*}$ satisfy~\eqref{zaa3*}. In the first two cases, it follows that
$$(z_{m}^{l})^{*}(z_{k}^{j})^{*}H=(z_{k}^{j})^{*}(z_{m}^{l})^{*}H=\{0\}$$
$$(z_{m}^{l})^{*}(z_{k}^{j})^{*}H=q(z_{k}^{j})^{*}(z_{m}^{l})^{*}H=\{0\}$$
respectively. In the last case when $m<k$ and $l<j,$ we get
$$(z_{m}^{l})^{*}(z_{k}^{j})^{*}H=((z_{k}^{j})^{*}(z_{m}^{l})^{*}-(q-q^{-1})(z_{m}^{j})^{*}(z_{k}^{l})^{*})H=(q-q^{-1})(z_{m}^{j})^{*}(z_{k}^{l})^{*}H$$
and as $(z_{m}^{j})^{*}(z_{k}^{l})^{*}=(z_{k}^{l})^{*}(z_{m}^{j})^{*}$ and at least one of the integers $m$ or $i$ is equal to $n,$ we get $$(q-q^{-1})(z_{m}^{j})^{*}(z_{k}^{l})^{*}H=\{0\}.$$
From this it follows that $H$ is invariant with respect to $(z_{k}^{j})^{*}$. But in this case the action of the extra term also annihilates $H,$ as one of the factors will be in $\textbf{I}^{*}$ and the two factors commute. 
\\

To prove the invariance under $\textbf{Z}_{n}^{n},$ we pick again some $z_{k}^{j}\in \textbf{Z}_{n}^{n}.$ By~\eqref{zaa4}, the only two elements in $\textbf{I}^{*}$ not commuting with $z_{k}^{j}$ are $(z_{n}^{j})^{*}$ and $(z_{k}^{n})^{*}.$ By~\eqref{zaa4}
$$
(z_{k}^{n})^{*}z_{k}^{j}v=q z_{k}^{j}(z_{k}^{n})^{*}v-(q-q^{-1})\sum_{r=k+1}^{n}z_{r}^{j}(z_{r}^{n})^{*}v=0,v\in H
$$
giving $z_{k}^{j}H\subseteq \ker (z_{k}^{n})^{*}.$ Similarly $z_{k}^{j}H\subseteq \ker(z_{n}^{j})^{*}.$ As $z_{k}^{j}$ commutes with the rest of $\textbf{I}^{*},$ we obtain $z_{k}^{j}H\subseteq H.$
\\

It is easy to see that~\eqref{zaa1}-\eqref{zaa3} and~\eqref{zaa1*}-\eqref{zaa3*} hold for the operators $\pi(z_{k}^{j})|_{H},$ as these relations for the restriction is just the restriction of the original relations, so the only crux is~\eqref{zaa4}. However, if we look at the equations~\eqref{zaa42}-\eqref{zaa44}, then we see that under $\pi,$ the terms in the sums corresponding to $j=n$ or $m=n$ annihilate $H$ and disappear in the restriction, thus only leave the equations for $n-1.$ We will postpone the proof of the irreducibility of $\pi'$ untill after the proof of Proposition~\ref{metabo}. 
\end{proof}
\begin{prop}\label{metabo}
For $\textbf{m}\in \mathbb{Z}_{+}^{2n-1},$ let  $H_{\textbf{m}}:=\overline{\pi(z(\textbf{m})) H}.$ Then $H_{\textbf{m}}\bot H_{\textbf{m}'}$ for $\textbf{m}\neq \textbf{m}'$ and 
\begin{equation}\label{metabo1}
K=\bigvee_{\textbf{m}\in\mathbb{Z}_{+}^{2n-1}}H_{\textbf{m}}.
\end{equation}
\end{prop}
\begin{proof}
The claim that $H_{\textbf{m}}\bot H_{\textbf{m}'}$ if $\textbf{m}\neq \textbf{m}'$ can be deduced in the same way as the analogous claim in Proposition~\ref{molly}. Similarly, we prove the equality in~\eqref{metabo1} as before, by showing that the right hand side is reducing $\pi.$ Invariance under $\textbf{Z}$ is once again deduced from Lemma~\ref{span1}. Invariance under $\textbf{Z}^{*}$ can be seen by combining~\eqref{isotensor} with Lemma~\ref{span1}. We get the result that any $a\in\mathrm{Pol}(\mathrm{Mat}_{n})_{q} $ can be written as a sum 
$$
a=\sum_{i}z(\textbf{m}_{i})a_{i}(a_{i}')^{*}z(\textbf{m}_{i}')^{*}
$$
with $a_{i},b_{i}'$ in the unital algebra generated by $\textbf{Z}_{n}^{n}.$ For $v\in H_{\textbf{m}}$ of the form $\pi(z(\textbf{m})) u$ with $u\in H,$ we then have that for any $(z_{k}^{j})^{*}\in \textbf{Z}^{*}$
$$
\pi(z_{k}^{j})^{*} v=\pi(z_{k}^{j})^{*}\pi(z(\textbf{m}))u=\sum_{i}\pi(z(\textbf{m}_{i})b_{i}(b_{i})^{*}z(\textbf{m}_{i}')^{*}) u=
$$
$$
=\sum_{i}\pi(z(\textbf{m}_{i}))(\pi(b_{i}(b_{i})^{*}z(\textbf{m}_{i}')^{*})u)\in \bigvee_{\textbf{m}\in\mathbb{Z}_{+}^{2n-1}}H_{\textbf{m}}
$$
as $\pi(b_{i}(b_{i})^{*}z(\textbf{m}_{i}')^{*}) u\in H.$
\end{proof}
We now prove that $\pi'$ is irreducible.
Notice that in the above proof, the only property of $H$ that is used to prove that the right hand side of~\eqref{metabo1} is reducing $\pi$, is that $H$ is reducing $\textbf{Z}_{n}^{n}$ and is invariant under $\textbf{I}^{*}.$ If $\pi'$ is not irreducible and $H$ can be decomposed into non-trivial orthogonal subspaces $H_{1}\oplus H_{2},$ both reducing $\pi'.$ Then as $H_{1},H_{2}$ are both annihilated by the elements in $\textbf{I}^{*}$ (since they are subspaces of $H$) and reducing $\textbf{Z}_{n}^{n}$ by the definition of $\pi',$ it follows by the proof of Proposition~\ref{metabo} that the subspaces $\bigvee_{\textbf{m}\in\mathbb{Z}_{+}^{2n-1}}\overline{\pi(z(\textbf{m}))H_{1}}$ and $\bigvee_{\textbf{m}\in\mathbb{Z}_{+}^{2n-1}}\overline{\pi(z(\textbf{m}))H_{2}}$ are both reducing $\pi$ and as they are not equal, and in fact even orthogonal to each other, we derive a contradiction to the assumption that $\pi$ is irreducible.
\\

We can now use the induction hypothesis on $\pi'$ that it can be lifted to a $*$-representation $\Pi':\mathbb{C}[SU_{2n-2}]_{q}\rightarrow B(H)$  such that 
$\pi'=\Pi'\circ \zeta.$ 
\\

We now lift $\Pi'$ to a $*$-representation $\Pi$ of $\mathbb{C}[SU_{2n}]_{q}$ by first defining a $*$-homomorphism \begin{equation}\label{delta}\delta:\mathbb{C}[SU_{2n}]_{q}\rightarrow \mathbb{C}[SU_{2n-2}]_{q}\end{equation} determined by 
$$\delta(t_{k,j})=
\begin{cases}
t_{kj} & \text{ if $1\leq k,j\leq 2n-2$}\\ \delta_{k,j}I & \text{ otherwise}
\end{cases}$$ (this corresponds to putting the matrix $(t_{kj})_{k,j=1}^{2n-2}$ in the upper left corner of the matrix $(t_{kj})_{k,j=1}^{2n}$) and then letting 
$\Pi=\Pi'\circ \delta.$
\\

Let us now consider $d=s_{n}s_{n+1}\dots s_{2n-1}$ and $b=s_{2n-2}s_{2n-3}\dots s_{n}$ and make the $*$-representation $\Delta$ of $\mathrm{Pol}(\mathrm{Mat}_{n})_{q}$ given by the formula $\Delta=(\pi_{d}\otimes \Pi \otimes \pi_{b})\circ \zeta.$ Again we can use the diagram representation of $\pi_{d}$ and $\pi_{b}$ to calculate easily the images of the generators in $\mathbb{C}[SU_{2n}]_{q}.$ We have that $\pi_{d}$ and $\pi_{b}$ correspond to the diagrams
$$
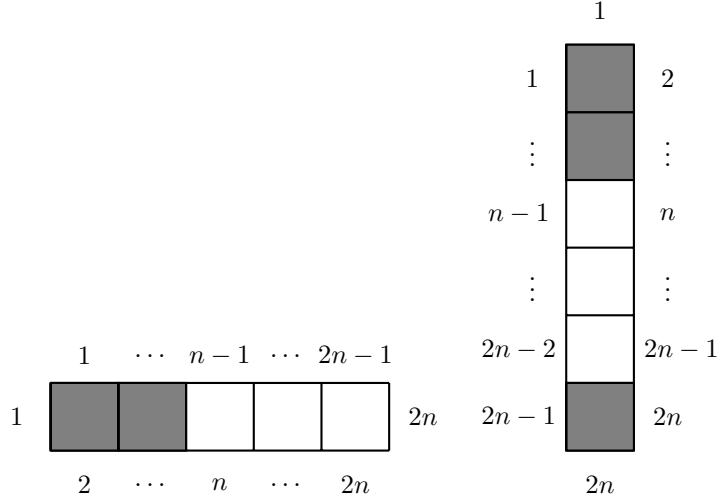
\begin{array}{cc}
\begin{tikzpicture}[thick,scale=0.9]
\draw[step=1.0,black,thick] (1,1) grid (6,2);
\node at (1.5,0.5) {$2$};
\node at (2.5,0.5) {$\dots$};
\node at (3.5,0.5) {$n$};
\node at (4.5,0.5) {$\dots$};
\node at (5.5,0.5) {$2n$};
\node at (0.5,1.5) {$1$};
\node at (6.5,1.5) {$2n$};
\node at (1.5,2.4) {$1$};
\node at (2.5,2.4) {$\dots$};
\node at (3.5,2.4) {$n-1$};
\node at (4.5,2.4) {$\dots$};
\node at (5.5,2.4) {$2n-1$};
\filldraw[fill=black!50!white, draw=black](1,2) -- (1,1) -- (2,1)--(2,2)--(1,2)--(1,1);
\filldraw[fill=black!50!white, draw=black](2,2) -- (2,1) -- (3,1)--(3,2)--(2,2)--(2,1);
\end{tikzpicture}
&
\begin{tikzpicture}[thick,scale=0.9]
\draw[step=1.0,black,thick] (1,1) grid (2,7);
\node at (0.5,6.5) {$1$};
\node at (0.5,5.5) {$\vdots$};
\node at (0.3,4.5) {$n-1$};
\node at (0.5,3.5) {$\vdots$};
\node at (0.3,2.5) {$2n-2$};
\node at (0.3,1.5) {$2n-1$};
\node at (1.5,0.5) {$2n$};
\node at (1.5,7.5) {$1$};
\node at (2.5,6.5) {$2$};
\node at (2.5,5.5) {$\vdots$};
\node at (2.5,4.5) {$n$};
\node at (2.5,3.5) {$\vdots$};
\node at (2.7,2.5) {$2n-1$};
\node at (2.5,1.5) {$2n$};
\filldraw[fill=black!50!white, draw=black](1,2) -- (1,1) -- (2,1)--(2,2)--(1,2)--(1,1);
\filldraw[fill=black!50!white, draw=black](1,7) -- (1,6) -- (2,6)--(2,7)--(1,7)--(1,6);
\filldraw[fill=black!50!white, draw=black](1,6) -- (1,5) -- (2,5)--(2,6)--(1,6)--(1,5);
\end{tikzpicture}
\end{array}
$$
respectively, where in the vertical diagram we calculate the image of $\pi_{b}(t_{kj})$ by connecting $k$ on the left (if $k\neq 2n$) or bottom side (if $k=2n$) with $j$ on the top (if $j=1$) or right side (if $j\neq 1$).
\begin{lem}\label{visa}
We have 
$$
(\cap_{k=1}^{n}\ker \Delta(z_{n}^{k})^{*})\cap (\cap_{k=1}^{n}\ker \Delta(z_{k}^{n})^{*})=
e_{\textbf{0}}\otimes H \otimes e_{\textbf{0}}
$$
and for all $v\in H,$ we have
$$
\Delta(z_{j}^{n})^{l}(e_{\textbf{0}}\otimes v \otimes e_{\textbf{0}})=
$$
$$
a(l)e_{\textbf{0}+(l)_{j}}\otimes v \otimes  e_{\textbf{0}}
$$
for $1\leq  j \leq n,$ where $\textbf{0}+(l)_{j}$ means that $l$ is added to the $j$'th cordinate. Also
$$
\Delta(z_{n}^{j})^{l}e_{\textbf{0}}\otimes v\otimes e_{\textbf{0}}=
$$
$$
a(l)e_{\textbf{0}}\otimes v\otimes e_{\textbf{0}+(l)_{n-j}}
$$
for $1\leq j\leq n-1$ and $a(l)=\prod_{m=1}^{l-1}\sqrt{1-q^{2m}}.$
\end{lem}
\begin{proof}
Evidently, if we prove the claims made on the action of the operators on $e_{\textbf{0}}\otimes H\otimes e_{0},$ then the claims about intersections of the kernels are easy to see.
For $1\leq k\leq n-1,$ we have
$$
\Delta(z_{n}^{k})=\sum_{r,s=1}^{2n}\pi_{d}(t_{2n,r})\otimes \Pi(t_{r,s})\otimes \pi_{b}(t_{s,n+k})=
$$
$$
\sum_{s=1}^{2n}\sum_{r=2n-1}^{2n}\pi_{d}(t_{2n,r})\otimes \Pi(t_{r,s})\otimes \pi_{b}(t_{s,n+k})
$$
as $\pi_{d}(t_{2n,r})=0$ unless $r=2n-1,2n.$ If $r=2n,$ then $\Pi(t_{r,s})=\delta_{r,s}I$ so that $s=2n,$ but then $\pi_{b}(t_{2n,n+k})=\delta_{2n,n+k}I=0,$ by the bounds on $k.$ So the sum reduces to   
$$
\sum_{s=1}^{2n}\pi_{d}(t_{2n,2n-1})\otimes \Pi(t_{2n-1,s})\otimes \pi_{b}(t_{s,n+k})=\sum_{s=1}^{2n}\pi_{d}(t_{2n,2n-1})\otimes \delta_{2n-1,s}I\otimes \pi_{b}(t_{s,n+k})=
$$
$$
=\pi_{d}(t_{2n,2n-1})\otimes I\otimes \pi_{b}(t_{2n-1,n+k})
$$
and since $\pi_{d}(t_{2n,2n-1})e_{\textbf{0}}=e_{\textbf{0}},$ and $T_{22}^{l}e_{0}=a(l)e_{l},$ it is not hard to see by calculating the action of $\pi_{b}(t_{2n-1,n+k})$ that the lemma holds in this case.
On the other hand, if $1\leq m\leq n,$ then 
$$
\Delta(z_{m}^{n})=(-q)^{m-n}\sum_{r,s=1}^{2n}\pi_{d}(t_{n+m,r})\otimes \Pi(t_{r,s})\otimes \pi_{b}(t_{s,2n})=
$$
$$
=(-q)^{m-n}\sum_{r,s=1}^{2n}\pi_{d}(t_{n+m,r})\otimes \Pi(t_{r,s})\otimes \delta_{s,2n}I=
(-q)^{m-n}\sum_{r=1}^{2n}\pi_{d}(t_{n+m,r})\otimes \Pi(t_{r,2n})\otimes I=
$$
$$
=(-q)^{m-n}\sum_{r=1}^{2n}\pi_{d}(t_{n+m,r})\otimes \delta_{r,2n}I\otimes I=
$$
$$
=(-q)^{m-n}\pi_{d}(t_{n+m,2n})\otimes I\otimes I
$$
and a calculation using the diagram for $\pi_{d}$ confirms the lemma also in this case.
\end{proof}

\begin{cor}\label{metabo2}
The linear span of the vectors in $\ell^{2}(\mathbb{Z}_{+})^{\otimes n}\otimes H\otimes \ell^{2}(\mathbb{Z}_{+})^{\otimes n-1}$ of the form $$\Delta(z(\textbf{m}))(e_{\textbf{0}}\otimes v\otimes e_{\textbf{0}})$$ is dense in $\ell^{2}(\mathbb{Z}_{+})^{\otimes n}\otimes H\otimes \ell^{2}(\mathbb{Z}_{+})^{\otimes n-1}$.
\end{cor}
\begin{lem}\label{utanmyra}

For $z_{k}^{j}\in \textbf{Z}_{n}^{n}$ and $v\in H,$ we have
\begin{equation}\label{final}
\Delta(z_{k}^{j}) (e_{\textbf{0}}\otimes v \otimes e_{\textbf{0}})=e_{\textbf{0}}\otimes  \pi(z_{k}^{j}) v \otimes e_{\textbf{0}}.
\end{equation}
\end{lem}
\begin{proof}
By the definition of $\Delta$
$$\Delta(z_{k}^{j})(e_{\textbf{0}}\otimes v\otimes e_{\textbf{0}})=(-q)^{k-n}\sum_{r,s=1}^{2n}\pi_{d}(t_{n+k,r})\otimes \Pi(t_{r,s})\otimes \pi_{b}(t_{s,n+j})(e_{\textbf{0}}\otimes v\otimes e_{\textbf{0}})$$
Using the diagram for $\pi_{d},$ it is easy to see that the only value of $r$ such that the path for $\pi_{d}(t_{n+k,r})$ does not contain the hook 
$\begin{tikzpicture}[line width=0.21mm,scale=0.4]
\draw (1,0) -- (0,0) -- (0,1)--(1,1)--(1,0)--(0,0);
\draw [->] (0.1,0.5)--(0.5,0.5)--(0.5,0.9);
\end{tikzpicture}$ (i.e. a $T_{11}$ factor, annihilating $e_{\textbf{0}}$), is $r=n+k-1$ and hence the sum reduces to
$$
(-q)^{k-n}\sum_{s=1}^{2n}\pi_{d}(t_{n+k,n+k-1})\otimes \Pi(t_{n+k-1,s})\otimes \pi_{b}(t_{s,n+j})(e_{\textbf{0}}\otimes v\otimes e_{\textbf{0}})=
$$
$$
(-q)^{k-n}\sum_{s=1}^{2n}I\otimes \Pi(t_{n+k-1,s})\otimes \pi_{b}(t_{s,n+j})(e_{\textbf{0}}\otimes v\otimes e_{\textbf{0}}).
$$
as $\pi_{d}(t_{n+k,n+k-1})e_{\textbf{0}}=e_{\textbf{0}}.$
Likewise, as $1\leq j\leq n-1,$ it is easy to see that the only value of $s$ such that $\pi_{b}(t_{s,n+j})$ does not annihilate $e_{\textbf{0}}$ is $s=n+j-1.$ Thus the sum contains the only non-zero term
$$
(-q)^{k-n}I\otimes \Pi(t_{n+k-1,n+j-1})\otimes \pi_{b}(t_{n+j-1,n+j})(e_{\textbf{0}}\otimes v\otimes e_{\textbf{0}})=
$$
$$
(-q)^{k+1-n}(I\otimes \Pi(t_{n+k-1,n-j-1})\otimes I)(e_{\textbf{0}}\otimes v\otimes e_{\textbf{0}})=
$$
$$
(I\otimes (\pi(t_{k,j})|_{H})\otimes I)(e_{\textbf{0}}\otimes v\otimes e_{\textbf{0}})=e_{\textbf{0}}\otimes \pi(z_{k}^{j})v\otimes e_{\textbf{0}}
$$
as $\pi_{b}(t_{n+j-1,n+j})e_{\textbf{0}}=(-q)e_{\textbf{0}}$ and $(-q)^{k+1-n}\Pi(t_{n+k-1,n-j-1})=\pi(t_{k,j})|_{H}$ by the definition of $\Pi.$
\end{proof}
\begin{prop}
$\Delta$ is unitary equivalent to $\pi$ by the unitary isometry $U:K\rightarrow \ell^{2}(\mathbb{Z}_{+})^{n}\otimes H\otimes \ell^{2}(\mathbb{Z}_{+})^{n-1}$ determined by $$\pi(z(\textbf{m}))v\mapsto \Delta(z(\textbf{m}))( e_{\textbf{0}}\otimes v\otimes e_{\textbf{0}})$$ for $v\in H.$
\end{prop}
\begin{proof}
This is shown in a completely analogous way to Proposition~\ref{robert}, by first proving that $U$ defines an isometry. For this one uses Proposition~\ref{metabo} to show that $U$ is well defined on a dense subset and then using lemmas~\ref{linj} and~\ref{linj2} to see that it is isometric and hence can be extended to the whole space. Corollary~\ref{metabo2} then implies that $U$ is onto and the definition of $U$ shows that it intertwines the images of the operators in $\textbf{I}.$ We can then use Lemma~\ref{utanmyra} and a similar calculation as the one that was made in Proposition~\ref{robert}, to see that $U$ also intertwines the images of the operators in $\textbf{Z}_{n}^{n}.$
\end{proof}
This completes the proof of the first item of Theorem~\ref{main}. 
\section{Irreducible $*$-representations of $\mathrm{Pol}(\mathrm{Mat}_{n})_{q}$}
\subsection{Elements of Minimal Length in Certain Subsets of  $S_{2n}$}
Let $\sigma\in S_{n}.$ By $\ell(\sigma)$ we denote the length of $\sigma,$ that is, the number of elements in the minimal decomposition of $\sigma$ as the product of adjacent transpositions $\sigma=s_{j_{1}}s_{j_{2}}\dots s_{j_{\ell(\sigma)}}.$ It is known that
\begin{equation}
\ell(\sigma)=\#\{1\leq j<i\leq n|\sigma(i)<\sigma(j)\},
\end{equation} 
where $\#$ denotes the number of elements in the set.
\\

Let $S\subseteq S_{2n}$ be the subgroup consisting of permutations that only permutes the first $n$ elements and leave the rest unchanged. It is clear that $S\cong S_{n}.$ For $\sigma\in S_{2n}$, let
$$
O_{\sigma}:=\{g_{1} \sigma g_{2}|g_{1},g_{2}\in S\}.
$$
We have the following:
\begin{prop}\label{1}
There is a unique element $w\in O_{\sigma}$ of minimal length $\ell(w)$ and with minimal decomposition $w=s_{j_{1}}s_{j_{2}}\cdots s_{j_{\ell(w)}},$ such that for any other $t\in O_{\sigma}$ there are $h=s_{k_{1}}s_{k_{2}}\cdots s_{k_{\ell(h)}} $ and $g=s_{l_{1}}s_{l_{2}}\cdots s_{l_{\ell(g)}}$ in $S$ such that 
$$t=h\cdot w\cdot g=(s_{k_{1}}s_{k_{2}}\cdots s_{k_{\ell(h)}})(s_{j_{1}}s_{j_{2}}\cdots s_{j_{\ell(w)}})(s_{l_{1}}s_{l_{2}}\cdots s_{l_{\ell(g)}})
$$
is a minimal decomposition of $t.$ 
\end{prop}
\begin{proof}
For any $t\in O_{\sigma}$ we split $\{1,2,\dots,n,n+1,\dots ,2n\}$ into the four sets
$$
N_{1}^{t}=\{1\leq j\leq n|1\leq t(j)\leq n\}
$$
$$
N_{2}^{t}=\{1\leq j\leq n|n+1\leq t(j)\leq 2n\}
$$
$$
N_{3}^{t}=\{n+1\leq j\leq 2n|1\leq t(j)\leq n\}
$$
$$
N_{4}^{t}=\{n+1\leq j\leq 2n|n+1\leq t(j)\leq 2n\}.
$$
By decomposing $\{1,2,\dots,2n\}=N_{1}^{t}\cup N_{2}^{t}\cup N_{3}^{t}\cup N_{4}^{t}$ we can for any $t\in O_{\sigma}$ calculate the length $\ell(t)$ as 
$$\ell(t)=\#\{1\leq i<j\leq 2n | t(j)<t(i)\}=
$$
$$
\sum_{m,l=1}^{4}\#\{i\in N_{m}^{t},j\in N_{l}^{t}, i<j | t(j)<t(i)\}=$$
\begin{equation}\label{one}
\begin{array}{cc}
=\#\{i\in N_{1}^{t},j\in N_{1}^{t}, i<j | t(j)<t(i)\}+\#\{i\in N_{1}^{t},j\in N_{3}^{t}, i<j | t(j)<t(i)\}+\\
+\#\{i\in N_{2}^{t},j\in N_{1}^{t}, i<j | t(j)<t(i)\}+\#\{i\in N_{2}^{t},j\in N_{2}^{t}, i<j | t(j)<t(i)\}+\\
+\#\{i\in N_{2}^{t},j\in N_{3}^{t}, i<j | t(j)<t(i)\}+\#\{i\in N_{2}^{t},j\in N_{4}^{t}, i<j | t(j)<t(i)\}+\\
+\#\{i\in N_{3}^{t},j\in N_{3}^{t}, i<j | t(j)<t(i)\}+\#\{i\in N_{4}^{t},j\in N_{3}^{t}, i<j | t(j)<t(i)\}+\\
+\#\{i\in N_{4}^{t},j\in N_{4}^{t}, i<j | t(j)<t(i)\}.\\
\end{array}
\end{equation}
In~\eqref{one}, we have removed the terms we know to be zero. For example, the term $\#(\{i\in N_{3}^{t},j\in N_{2}^{t}, i<j | t(j)<t(i)\})$ is equal to zero because $j<i$ holds for all $i\in N_{3}^{t}$ and $j\in N_{2}^{t}.$
\\

Notice that $ |N_{1}^{t} |= | N_{4}^{t}|$ and $ | N_{2}^{t}|= |N_{3}^{t} |=n-|N_{1}^{t} |.$ It is not hard to see that if $g,h\in S$ then
$|N_{i}^{\sigma}|=|N_{i}^{h\sigma g}|$ for $i=1,2,3,4,$ and also $N_{2}^{h\sigma}=N_{2}^{\sigma},N_{2}^{\sigma g}=g^{-1}(N_{2}^{\sigma}),N_{3}^{h\sigma g}=N_{3}^{\sigma}$ and $N_{4}^{h\sigma g}=N_{4}^{\sigma}.$ For all $h,g\in S,$ we have
\begin{equation}\label{23}
\#(\{i\in N_{2}^{h\sigma g},j\in N_{3}^{h\sigma g}|h\sigma g(j)<h\sigma g(i)\})=
\frac{|N_{2}^{h\sigma g} |\left(|N_{2}^{h\sigma g} |+1 \right)}{2}=\frac{|N_{2}^{\sigma } |\left(|N_{2}^{\sigma} |+1 \right)}{2}
\end{equation}
 For any $g,h\in S,$ the inequality $i<j $ always holds for $i\in N_{2}^{h\sigma g},j\in N_{4}^{h\sigma g}$ and as $N_{2}^{h\sigma g}=N_{2}^{\sigma g}=g^{-1}(N_{2}^{\sigma })$ and $N_{4}^{h\sigma g}=N_{4}^{\sigma}$ we get 
$$
\#\{i\in N_{2}^{h\sigma g},j\in N_{4}^{h\sigma g}, i<j | h\sigma g(j)<h\sigma g(i)\}=
$$
$$
\#\{i\in N_{2}^{h\sigma g},j\in N_{4}^{h\sigma g} | h\sigma g(j)<h\sigma g(i)\}=
$$
$$
\#\{g(i)\in N_{2}^{\sigma},j\in N_{4}^{\sigma} | \sigma(j)<\sigma(g(i))\}=\#\{i\in N_{2}^{\sigma},j\in N_{4}^{\sigma} | \sigma(j)<\sigma(i)\}
$$
and hence for any $t\in O_{\sigma},$ we have
\begin{equation}\label{24}
\#\{i\in N_{2}^{t},j\in N_{4}^{t}, i<j | t(j)<t(i)\}=\#\{i\in N_{2}^{\sigma},j\in N_{4}^{\sigma}, i<j | \sigma(j)<\sigma(i)\}.
\end{equation}
A very similar reasoning gives that for all $t\in O_{\sigma},$
\begin{equation}\label{43}
\#\{i\in N_{4}^{t},j\in N_{3}^{t}, i<j | t(j)<t(i)\}=\#\{i\in N_{4}^{\sigma},j\in N_{3}^{\sigma}, i<j | \sigma(j)<\sigma(i)\}.
\end{equation}
For $t\in O_{\sigma},$ let $$m_{lk}^{t}=\#\{i\in N_{l}^{t},j\in N_{k}^{t}, i<j|t(j)<t(i)\},$$ then we can rewrite~\eqref{one} as 
\begin{equation}\label{sum}
\ell(t)=m_{11}^{t}+m_{13}^{t}+m_{21}^{t}+m_{22}^{t}+m_{23}^{t}+m_{24}^{t}+m_{33}^{t}+m_{43}^{t}+m_{44}^{t}.
\end{equation}

From the equalities~\eqref{23},~\eqref{24},~\eqref{43} we get, for any $t\in O_{\sigma},$ the lower bound 
\begin{equation}\label{2}
m_{44}^{\sigma}+m_{43}^{\sigma}+m_{24}^{\sigma}+m_{23}^{\sigma}\leq \ell(t).
\end{equation}
We claim that there is a unique element $w\in O_{\sigma}$ such that equality holds in~\eqref{2}. We show the existence first. We are going to do this by showing that there exist $g,h\in S$ such that \begin{equation}\label{asss}m_{11}^{h\sigma g}=m_{13}^{h\sigma g}=m_{21}^{h\sigma g}=m_{22}^{h\sigma g}=m_{33}^{h\sigma g}=0.\end{equation}
If we then let $w:=h\sigma g$, then as~\eqref{2} holds for any $t\in O_{\sigma},$ it follows from~\eqref{sum} that the element $w$ would be of minimal length in $O_{\sigma}.$
\\
  
Consider $\{\sigma(1),\sigma(2),\dots ,\sigma(n)\}=N^{\sigma}_{1}\cup N^{\sigma}_{2}.$ There is a unique permutation $g\in S$ such that 
$$
\sigma g(1)<\sigma g(2)<\dots<\sigma g(n)
$$ 
and hence $m_{11}^{\sigma g}=m_{21}^{\sigma g}=m_{22}^{\sigma g}=0.$ Now, consider $$\{(\sigma g)^{-1}(1),(\sigma g)^{-1}(2),\dots ,(\sigma g)^{-1}(n)\}=N^{\sigma g}_{1}\cup N^{\sigma g}_{3}.$$
There is $h\in S$ such that for all $1\leq j,k\leq n$ we have $(\sigma g)^{-1}(j)<(\sigma g)^{-1}(k)$ if and only if $h(j)<h(k).$ We claim that~\eqref{asss} holds for $h\sigma g$. To see this, first notice that $N_{2}^{\sigma g}=N_{2}^{h\sigma g},$ since $h$ fixes $n+1,\dots,2n.$ Hence $m_{22}^{h\sigma g}=m_{22}^{\sigma g}=0.$ Also, we have $m_{11}^{h\sigma g}=m_{13}^{h\sigma g}=m_{33}^{h\sigma g}=0$ by the definition of $h.$ We need to check that $h\sigma g$ preserves the ordering of $\{1,2,\dots,n\},$ so that $m_{21}^{h\sigma g}=0.$
Let $1\leq i<j\leq n,$ then by the definition of $g,$ we have $\sigma g(i)<\sigma g(j),$ If $n+1\leq \sigma g(j)\leq 2n,$ then $h$ fixes $\sigma g(j)$ and $\sigma g(i)$ can at most be mapped by $h$ to another integer less than $n,$ and so we have $h\sigma g(i)<h\sigma g(j).$ If $1\leq \sigma g(j)\leq n,$ then $$\{i,j\}=\{(\sigma g)^{-1}(\sigma g(i)),(\sigma g)^{-1}(\sigma g(j))\}\subseteq N^{\sigma g}_{1}\cup N^{\sigma g}_{3},$$ and hence $h(\sigma g(i))<h(\sigma g(j))$ by the definition of $h.$ This shows that also $m_{12}^{h\sigma g}=0$ and we can now let $w=h\sigma g.$ Hence, there exists an elements $w\in O_{\sigma},$ such that equality holds in~\eqref{2}.
\\

To prove the uniqueness of $w$, we will show that the equality in~\eqref{2} determines $w$ completely. If 
$$
m_{13}^{w}=\#\{i\in N_{1}^{w},j\in N_{3}^{w}, i<j | w(j)<w(i)\}=0,
$$
then as $w(N_{1}^{w}\cup N_{3}^{w})=\{1,2,\dots,n\},$ it follows that the image of $N_{1}^{w}$ under $w$ must be $\{1,2,\dots,|N_{1}^{w}| \}.$ Furthermore, as $w(j)<w(i)$ holds automatically for $j\in N_{1}^{w}$ and $i\in N_{2}^{w},$ we have that $$m_{21}^{w}=\#\{i\in N_{2}^{w},j\in N_{1}^{w}, i<j | w(j)<w(i)\}=0$$ implies \begin{equation}\label{eqe}\#\{i\in N_{2}^{w},j\in N_{1}^{w}, i<j\}=0.\end{equation} From~\eqref{eqe}, it follows that $w(N_{1}^{w})=N_{1}^{w}=\{1,\dots, |N_{1}^{w}| \}.$ As $m_{11}^{w}=0,$ we can then deduce that $w(j)=j$ for $j\in N_{1}^{w}.$ Moreover, we have $N_{2}^{w}=\{N_{1}^{w}+1,\dots n\}.$ Thus $w:N_{3}^{w}\to N_{2}^{w}$ and $w:N_{2}^{w}\to \{n+1,n+2,\dots, 2n\}\backslash w(N_{4}^{w})$ are bijective maps that preserves the ordering i.e they are uniquely determined. Finally, the action of $w$ on $N_{4}^{w}$ are the same for all elements in $O_{\sigma}.$
\\

To prove the second part, it is clear that we only need to prove this claim for $\sigma,$ as $O_{t}=O_{\sigma}$ for any $t\in O_{\sigma}.$ We will show that \begin{equation}\label{4}\ell(\sigma)=\ell(w)+\ell(g)+\ell(h),\end{equation} with $g,h\in S$ as above. Notice that by the definition of $g,$ we have
$$
\ell(g)=\ell(g^{-1})=\#\{1\leq i<j\leq n | \sigma(j)<\sigma(i)\}
$$
and as $\{1,2,\dots,n\}=N_{1}^{\sigma}\cup N_{2}^{\sigma},$ we get that 
$$\ell(g)=m_{11}^{\sigma}+m_{12}^{\sigma}+m_{21}^{\sigma}+m_{22}^{\sigma}=m_{11}^{\sigma}+m_{21}^{\sigma}+m_{22}^{\sigma}$$
as $m_{12}^{\sigma}=0.$ While, by the definition of $h,$ we have
$$
\ell(h)=\#\{i,j\in N_{1}^{\sigma g}\cup N_{3}^{\sigma g}, i<j | \sigma g(j)<\sigma g(i)\}=
$$
$$
m_{11}^{\sigma g}+m_{13}^{ \sigma g}+m_{31}^{\sigma g}+m_{33}^{\sigma g}=m_{11}^{\sigma g}+m_{13}^{\sigma g}+m_{33}^{\sigma g}
$$
as $m_{31}^{\sigma g}=0.$ By the definition of $g,$ we have $\#\{i\in N_{1}^{\sigma g},j\in N_{1}^{\sigma g},  i<j  | \sigma g(j)<\sigma g(i)\}=0.$ Moreover, we have $$\#\{i\in N_{3}^{\sigma g},j\in N_{3}^{\sigma g},  i<j  | \sigma g(j)<\sigma g(i)\}=\#\{i\in N_{3}^{\sigma},j\in N_{3}^{\sigma},  i<j  | \sigma(j)<\sigma(i)\}$$ as $g$ does not permute any of the integers in $N_{3}^{\sigma}.$ We will now show the equality 
$$
\#\{i\in N_{1}^{\sigma g},j\in N_{3}^{\sigma g},  i<j  | \sigma g(j)<\sigma g(i)\}=\#\{i\in N_{1}^{\sigma},j\in N_{3}^{\sigma},  i<j  | \sigma(j)<\sigma(i)\}.
$$
We have that  $i<j$ already holds for all $i\in N_{1}^{\sigma g},j\in N_{3}^{\sigma g}$ and hence
$$
\#\{i\in N_{1}^{\sigma g},j\in N_{3}^{\sigma g},  i<j  | \sigma g(j)<\sigma g(i)\}=\#\{i\in N_{1}^{\sigma g},j\in N_{3}^{\sigma g}  | \sigma g(j)<\sigma g(i)\}=
$$
$$
\#\{i\in N_{1}^{\sigma g},j\in N_{3}^{\sigma}  | \sigma(j)<\sigma g(i)\}
$$
as, again, $\sigma g(j)=\sigma(j)$ and $N_{3}^{\sigma g}=N_{3}^{\sigma}.$ If $i\in N_{1}^{\sigma g},$ then $g(i)\in N_{1}^{\sigma}$ and hence
$$
\#\{i\in N_{1}^{\sigma g},j\in N_{3}^{\sigma}  | \sigma(j)<\sigma (g(i))\}=\#\{i\in N_{1}^{\sigma},j\in N_{3}^{\sigma}  | \sigma(j)<\sigma(i)\}=
$$
$$
\#\{i\in N_{1}^{\sigma},j\in N_{3}^{\sigma},  i<j  | \sigma(j)<\sigma(i)\}=m_{13}^{\sigma}.
$$
This shows that $\ell(\sigma)=\ell(w)+\ell(g)+\ell(h)$ and hence if $w=s_{j_{2}}\cdots s_{j_{\ell(w)}},g=s_{l_{1}}s_{l_{2}}\cdots s_{l_{\ell(g)}},h=s_{k_{1}}s_{k_{2}}\cdots s_{k_{\ell(h)}}$ are minimal decompositions, then so is 
$$
\sigma=(s_{k_{1}}s_{k_{2}}\cdots s_{k_{\ell(h)}})(s_{j_{1}}s_{j_{2}}\cdots s_{j_{\ell(w)}})(s_{l_{1}}s_{l_{2}}\cdots s_{l_{\ell(g)}}).
$$
\end{proof}
For integers $1\leq j\leq n$ and $0\leq k\leq n$ we define the cycles in $S_{2n}$
\begin{equation}\label{cycles}
c_{k,j}=\begin{cases}s_{j+n-k}s_{j+n-k+1}\cdots s_{j+n-1}& \text{if $1\leq k\leq n$}\\ e & \text{if $k=0$}\end{cases}
\end{equation}
where $e\in S_{2n}$ is the identity element. We shall now prove the following result, that gives a decomposition of the element of minimal length in $O_{\sigma}$ in terms of such cycles. 
\begin{prop}\label{decomp}
If $w\in O_{\sigma}\subseteq S_{2n}$ is of minimal length $\ell(w)$, then there is a multi-index $\textbf{k}=[k_{n},k_{n-1},\dots,k_{1}]\in \mathbb{Z}_{+}^{n}$ with the property that \begin{equation}\label{didi}k_{i}\leq \max_{i<j\leq n}(k_{j}+i-j+1,i)\end{equation}(i.e. $\textbf{k}$ is admissible) and such that 
$$
w=c_{k_{n},n}c_{k_{n-1},n-1}\cdots c_{k_{1},1}
$$
is a minimal decomposition of $w$ and $\ell(w)=\sum_{j=1}^{n}k_{j}.$
\end{prop}
\begin{proof}
 Consider pairs $(m,n)$ with $n\in \mathbb{N}$ and $0\leq m\leq n.$ We order these as
\begin{equation}\label{ordering}
\begin{array}{ccc}(m_{1},n_{1})< (m_{2},n_{2}), & \text{if $n_{1}<n_{2},$ or if $n_{1}=n_{2}$ and $m_{1}<m_{2}.$}
\end{array}
\end{equation}
Clearly, this gives a well-ordering on the set of such pairs $(m,n)$.
\\

To every $w\in O_{\sigma}\subseteq S_{2n}$ of minimal length $\ell(w),$ we can associate a pair $(m,n),$ where $0\leq m\leq n$ is the number of integers $n+1\leq j\leq 2n$ such that $1\leq w(j)\leq n.$ It follows that $0\leq m\leq n.$  We prove the proposition by induction on $(m,n)$ with the ordering~\eqref{ordering}. In the case $n=1,$ it follows that $w$ is either the identity permutation, associated to the pair $(0,1),$ or $s_{1}=\left(\begin{array}{cc}1 & 2 \\ 2 & 1\end{array}\right),$ associated to $(1,1).$ In both cases it is clear that the proposition holds.
\\

Assume now that the statement holds for all minimal length elements associated to pairs $(m_{1},n_{1}) <(m,n).$ Let $w\in O_{\sigma}\subseteq S_{2n}$ be of minimal length in $O_{\sigma}$ and associated to $(m,n).$ Consider the integer $w(2n).$ We claim that $n\leq w(2n)\leq 2n.$ In fact, it follows from the proof of Proposition~\ref{1} that
$$
\#\{1\leq i<j\leq 2n|1\leq w(j)<w(i)\leq n\}=m_{11}^{w}+m_{13}^{w}+m_{33}^{w}=0+0+0
$$
and a simple argument shows that if $1\leq w(2n)\leq n,$ then we must have $w(2n)=n.$\\

We will split the proof into the two cases 
\begin{enumerate}
\item[(1)] $n+1\leq w(2n)\leq 2n.$
\item[(2)] $w(2n)=n.$
\end{enumerate}
We start with the case $(1).$ If $w(2n)=2n,$ let $v=w.$ Otherwise, consider the cycle $$c:=s_{2n-1}\cdots s_{w(2n)+1} s_{w(2n)}=c_{2n-w(2n),n}^{-1}=$$
\begin{equation}
\left(\begin{array}{cccccccccc}1 & 2 &  \dots & w(2n) & w(2n)+1 &\dots & 2n\\1 & 2 &\dots & 2n &w(2n)&\dots & 2n-1\end{array}\right)
\end{equation}
and let $v=cw.$
\\

We claim that $v$ is the element of minimal length in $O_{v}.$ We prove this by showing that
$$m_{11}^{v}=m_{13}^{v}=m_{21}^{v}=m_{22}^{v}=m_{33}^{v}=0.$$
The latter follows directly from the formula
\begin{equation}\label{v1}
v(j)=\begin{cases}w(j), & \text{if $1\leq w(j)<  w(2n)$}\\ w(j)-1, & \text{if $w(2n)<  w(j)< 2n$}\\ 2n, & \text{if $j=2n$} \end{cases}
\end{equation}
or the observation that $c$ does not change the order of the integers in $$\{1,2,\dots,n\}\cup w(\{1,2,\dots,n\}).$$
We have also
$$
\ell(w)=\#\{1\leq i<j\leq 2n | w(j)<w(i)\}=$$ \begin{equation}\label{split} =\#\{1\leq i<j < 2n | w(j)<w(i)\}+\#\{1\leq i <2n | w(2n)<w(i)\}.
\end{equation}
Since $\#\{1\leq i <2n : w(2n)<w(i)\}=2n-w(2n),$ it follows from~\eqref{v1} that $$\#\{1\leq i<j < 2n | w(j)<w(i)\}=\ell(v).$$
Hence from~\eqref{split} we obtain \begin{equation}\label{l1}\ell(v)=\ell(w)-(2n-w(2n)).\end{equation}
As $m_{11}^{v}=0$ and $m$ is not equal to $n,$ it follows from the proof of Proposition~\ref{1} that $v$ fixes $\{1\}.$ In $S_{2n},$ consider the subgroup of elements that fix $\{1,2n\}.$ It is clearly isomorphic to $S_{2(n-1)}$ with the isomorphism given by \begin{equation}\label{phis}\phi:s_{i}\in S_{2n-2}\mapsto s_{i+1}\in S_{2n}.\end{equation}
It is easy to see that $\phi^{-1}(v)\in O_{\phi^{-1}(v)}\subseteq S_{2(n-1)}$ is of minimal length and $\phi^{-1}(v)$ is associated to the pair $(m,n-1)<(m,n)$. By induction, we have a decomposition $$\phi^{-1}(v)=c_{k_{n-1},n-1}c_{k_{n-2},n-2}\cdots c_{k_{1},1}\in S_{2(n-1)}$$ such that \begin{equation}\begin{array}{ccc}\label{in-1}k_{i}\leq \max_{i<j\leq n-1}(k_{j}+i+1-j,i),&  j=1,\dots, n-1.\end{array}\end{equation} Notice that if $c_{k,j}$ is the cycle~\eqref{cycles} in $S_{2(n-1)},$ then the image of $c_{k,j}$ under $\phi$ is $\phi(c_{k,j})= c_{k,j},$ the latter is now the cycle~\eqref{cycles} in $S_{2n}.$ It follows that 
$$
w=c^{-1} v=c_{k_{n},n}c_{k_{n-1},n-1}c_{k_{n-2},n-2}\cdots c_{k_{1},1} 
$$
with $k_{n}=2n-w(2n).$ Moreover, by~\eqref{in-1}
$$\begin{array}{ccc}k_{i}\leq \max_{i<j\leq n-1}(k_{j}+i+1-j,i)\leq \max_{i<j\leq n}(k_{j}+i+1-j,i), & i=1,\dots , n-1\end{array}$$  and $k_{n}=2n-w(2n)<n, $ and hence it follows that $$\begin{array}{cc}k_{i}\leq \max_{i<j\leq n}(k_{j}+i+1-j,i), & i=1,\dots,n\end{array}$$ and that~\eqref{didi} holds. Moreover, by~\eqref{l1} and $\ell(c)=2n-w(2n)=k_{n},$ it follows by induction that $$\ell(w)=2n-w(2n)+\ell(v)=k_{n}+\ell(v)=k_{n}+\sum_{j=1}^{n-1}k_{j}=\sum_{j=1}^{n}k_{j}.$$ 
\\

Let us now assume that $(2)$ holds, so that $w(2n)=n.$ In this case we have to argue slightly differently. Let 
\begin{equation}\label{t}
t=\left(\begin{array}{cccccc}1 &2&\dots &2n\\ 2n&1&\dots& 2n-1\end{array}\right)
\end{equation}
and consider $$v=t^{-1} c_{n,n}^{-1}w t\in S_{2n}.$$
We claim that $v\in O_{v}$ has minimal length. Notice that $c_{n,n}^{-1}w(2n)=2n,$ which gives $v(1)=1$ and as
$$c_{n,n}^{-1}=\left(\begin{array}{ccccccc} 1 & \dots & n & n+1 & \dots & 2n\\ 1 & \dots & 2n & n & \dots & 2n-1\end{array} \right)$$
we can, for $2\leq j\leq 2n,$ calculate
\begin{equation}\label{v2}
v(j)=t^{-1} c_{n,n}^{-1}w(t(j))=t c_{n,n}^{-1}(w(j-1))=\begin{cases} w(j-1)+1 & \text{if $2\leq w(j-1)\leq n-1$}\\ w(j-1) & \text{if $n+1\leq w(j-1)\leq 2n$} \end{cases}
\end{equation}
(as $w(2n)=n,$ these are the only cases). From this it is not hard to see that $m_{11}^{v}=m_{13}^{v}=m_{21}^{v}=m_{22}^{v}=m_{33}^{v}=0$ and that $$\ell(v)=\ell(w)-n.$$
It also follows from~\eqref{v2} that the number of $n+1\leq j\leq 2n$ such that $1\leq v(j)\leq n$ is $m-1,$ i.e. $v$ is associated to $(m-1,n)<(m,n).$ By induction, we have a minimal decomposition $$v=c_{j_{n},n}c_{j_{n-1},n-1}\cdots c_{j_{1},1}$$ such that $j_{i}\leq \max_{i+1<r\leq n}(j_{r}+i+1-r,i).$ However, we claim that $j_{1}=0.$ To see this, notice that if $j_{1}>0,$ then as 
$$
c_{j_{1},1}=\left(\begin{array}{ccccccccc}1&\dots&n-j_{1}+1&\dots & n&n+1 &\dots& 2n\\1&\dots&n-j_{1}+2 &\dots&n+1&n-j_{1}+1&\dots &2n \end{array}\right)
$$
it follows that $c_{j_{1},1}(n+1)<c_{j_{1},1}(n).$ It is not hard to see that this inequality persists when applying the other cycles, but it follows from~\eqref{v2} that $v(n)<v(n+1)$ and this contradiction gives that $j_{1}=0.$ For $2\leq i\leq n-1,$ we have $ts_{i}t^{-1}=s_{i-1},$ and hence 
$$
c_{n,n}^{-1}w=t v t^{-1}=(t c_{j_{n},n}t^{-1})(t c_{j_{n-1},n-1}t^{-1})\cdots (t c_{j_{2},2}t^{-1})= c_{j_{n},n-1} c_{j_{n-1},n-1}\cdots c_{j_{2},2}
$$
as no of the cycles $c_{j_{i}}$, $i=2,\dots,n$ contains $s_{1}$ and thus $t c_{j_{i},i}t^{-1}=c_{j_{i},i-1}.$ So we have the decomposition $$w=c_{k_{n},n}c_{k_{n-1},n-1}\cdots c_{k_{1},1}$$ with $k_{n}:=n$ and $k_{i}:=j_{i+1}$, $i=1,\dots,n-1.$ To show that the sequence $k_{i}$ has the property~\eqref{didi}, we again notice that the inequality holds trivially for $k_{n}=n$ and also for $k_{n-1},$ as
$$
\max_{n-1<j\leq n}(k_{j}+n-1+1-j,n-1)=\max_{n-1<j\leq n}(k_{n}+n-n,n-1)=\max(n,n-1)=n.
$$
Otherwise, for $1\leq i\leq n-2$ we get
$$
k_{i}=j_{i+1}\leq \max_{i+1<r\leq n}(j_{r}+(i+1)+1-r,i+1)= \max_{i<r\leq n-1}(k_{r}+i+1-r,i+1)=
$$
$$
 \max_{i<r\leq n-1}(k_{r}+i+1-r,n+i+1-n)= \max_{i<r\leq n}(k_{r}+i+1-r)\leq \max_{i<r\leq n}(k_{r}+i+1-r,i)
$$
and hence~\eqref{didi} is true for $k_{n}.$ We now calculate the length as
$$
\ell(w)=n+\ell(v)=n+\sum_{i=2}^{n}j_{i}=n+\sum_{i=1}^{n-1}k_{n}=\sum_{i=1}^{n}k_{n}
$$
and from this it follows that the decomposition $w=c_{k_{n},n}c_{k_{n-1},n-1}\cdots c_{k_{1},1}$ is a minimal one.
\end{proof}

\begin{prop}\label{decomp1}
If we have an admissible sequence $\textbf{k}=[k_{n},k_{n-1},\dots,k_{1}]\in \mathbb{Z}_{+}^{n},$ then the element $$w=c_{k_{n},n}c_{k_{n-1},n-1}\cdots c_{k_{1},1}\in S_{2n}$$ is the unique element of minimal length in $O_{w}$ and the admissible sequence from Proposition~\ref{decomp} coincides with $\textbf{k}.$
\end{prop}
\begin{proof}
First, notice that for any sequence of positive integers $[j_{n},\dots,j_{1}]$ such that $1\leq j_{m}\leq n,$ the product
$$
t=c_{j_{n},n}c_{j_{n-1},n-1}\cdots c_{j_{1},1}
$$
has the property that for $1\leq i<m\leq n,$ we have $t(i)<t(m).$ In fact, as
$$
c_{k,j}=\left(\begin{array}{cccccccccc}1&\dots &j+n-k &j+n-k+1 &\dots & j+n &\dots & 2n \\ 1 & \dots &j+n-k+1 &j+n-k+2 &\dots & j+n-k &\dots & 2n\end{array}\right)
$$
it follows that $c_{k,j}$ does not change the ordering of $(1,\dots , j+n-1)$ and as $$c_{k,j}((1,\dots , j+n-1))\subseteq (1,\dots ,n+j),$$
we can deduce that the composition $t=c_{j_{n},n}c_{j_{n-1},n-1}\cdots c_{j_{1},1}$ does not change the ordering of $(1,\dots,n).$
\\

Let $v\in O_{w}$ be the minimal element and let $\textbf{k}'$ be its associated admissible sequence. We are going to show $w=v$ by proving $\textbf{k}=\textbf{k}'.$ From Proposition~\ref{1}, we know that there exists $h,g\in S,$ such that $w=hvg.$ Since $w$ does not change the ordering of $(1,\dots, n),$ it follows from the construction of $g$ that we can assume $g=e.$ Hence $w=hv.$
\\

Notice that we have $w(2n)=v(2n),$ since if $n+1\leq w(2n)\leq 2n,$ then $w(2n)=(hv)(2n)=v(2n)$ and if $1\leq w(2n)\leq n,$ then as $w(2n)=c_{k_{n},n}(2n)$ it follows that $k_{n}=n$ and thus $w(2n)=n,$ and this implies that also $v(2n)=c_{k_{n}',n}(2n)=n.$ Hence $$k_{n}=2n-w(2n)=2n-v(2n)=k_{n}'.$$
We can proceed now in the same manner as in the proof of Proposition~\eqref{decomp}, by considering $c_{k_{n},n}^{-1}w$ and $c_{k_{n},n}^{-1}v.$
Again, we split the proof into the two cases $1\leq k_{n}<n$ and $k_{n}=n.$
\\

Notice that if $1\leq k_{n}< n,$ then the condition on $\textbf{k}$ gives $1\leq k_{1}< n$ and thus \begin{equation}\label{1=1}w(1)=v(1)=1.\end{equation} From~\eqref{1=1} it follows that, as $w=hv,$ we have $h(1)=1$ and that $c_{k_{n},n}^{-1}w,c_{k_{n},n}^{-1}v$ are in the image of $\phi:S_{2(n-1)}\to S_{2n}$  given by~\eqref{phis}.
\\

The proof can be completed by induction on $n.$ If $n=1,$ then it is easy to see that the proposition holds. We get that \begin{equation}\label{member}\phi^{-1}(c_{k_{n},n}^{-1}v)\in O_{\phi^{-1}(c_{k_{n},n}^{-1}w)}\end{equation} is of minimal length (as $h(1)=1,$ the inclusion follows, and the minimal length of $\phi^{-1}(c_{k_{n},n}^{-1}v)$ follows from the proof of Propostion~\ref{decomp}). Also
$$
\phi^{-1}(c_{k_{n},n}^{-1}w)=c_{k_{n-1},n-1}\cdots c_{k_{1},1}\in S_{2n-2}
$$
and $k_{i}\leq \max_{i<j\leq n}(k_{j}+i-j+1,i)\leq n-1.$ By induction, it follows that $\phi^{-1}(c_{k_{n},n}^{-1}w)=\phi^{-1}(c_{k_{n},n}^{-1}v).$
\\

If $k_{n}=n,$ let $t$ be as in~\eqref{t} and consider the permutations $t^{-1} c_{n,n}^{-1}w t$ and $t^{-1}c_{n,n}^{-1}v t.$ If is not hard to see that $$t^{-1} c_{n,n}^{-1}=(c_{n,n}t)^{-1}=g^{-1},$$ where 
$$
g=\left(\begin{array}{cccccccc}1&2&\dots &n&n+1&\dots&2n\\ n&1&\dots& n-1&n+1 &\dots &2n \end{array}\right)\in S.
$$
It follows that $t^{-1}c_{n,n}^{-1}v\in O_{t^{-1} c_{n,n}^{-1}w t}$ as 
$$
t^{-1} c_{n,n}^{-1}v t=gvt=gh^{-1}wt=ghg^{-1} gwt=(ghg^{-1})t^{-1} c_{n,n}^{-1}w t.
$$
By the proof of Proposition~\ref{decomp}, we have that $t^{-1}c_{n,n}^{-1}v t$ is the element of minimal length in $O_{t^{-1} c_{n,n}^{-1}w t}$ and 
$$
t^{-1} c_{n,n}^{-1}w t=c_{k_{n-1},n}\cdots c_{k_{1},2}.
$$
From the proof of Proposition~\ref{decomp}, we get that $[k_{n-1},\dots ,k_{1},0]$ is an admissible sequence. If we continue in this way, we get either $$k_{n}=k'_{n}=k_{n-1}=k'_{n-1}=\dots = k_{1}=k'_{1}=n$$ 
or we will eventually end up in the first case, if some $1\leq k_{j}<n,$ and then it will follow by induction that the remaining integers are equal.
\end{proof}
\begin{prop}\label{switch}
Let $w\in O_{\sigma}$ be of minimal length and $w=c_{k_{n},n}c_{k_{n-1},n-1}\cdots c_{k_{1},1}$ the decomposition from Proposition~\ref{decomp}. For an $1\leq i\leq n,$ we have  \begin{equation}\label{equ}k_{i}= \max_{i<j\leq n}(k_{j}+i+1-j,i)\end{equation} if and only if $1\leq w(n+i)\leq n.$
\end{prop}
\begin{proof}
Let $1\leq i\leq n$ be such that~\eqref{equ} holds. The equality~\eqref{equ} means that for all $i<j\leq n,$ the indices of the factors in $c_{k_{j},j}=s_{j+n-k_{j}}s_{j+n-k_{j}+1}\cdots s_{j+n-1}$ are all larger than $i+n-k_{i}$ as by~\eqref{equ}, for all $i<j\leq n,$ $$k_{j}+i+1-j\leq k_{i}\Rightarrow k_{j}+i-j<k_{i}\Rightarrow$$
\begin{equation}\label{ineq}
i+n-k_{i}<j+n-k_{j}.
\end{equation}
As $c_{k_{i},i}(n+i)=n+i-k_{i}$ and $c_{m}(n+i)=n+i$ for $1\leq m<i$, it follows from~\eqref{ineq} that $$w(n+i)=c_{k_{n},n}c_{k_{n-1},n-1}\cdots c_{k_{i},i}(n+i)=
$$
$$
c_{k_{n},n}c_{k_{n-1},n-1}\cdots c_{k_{i+1},i+1}(n+i-k_{i})=n+i-k_{i}$$ and as $k_{i}\geq i$ by~\eqref{equ}, we get $1\leq n+i-k_{i}\leq n.$
\\

Let us turn to the other direction. We are now going to prove that $1\leq w(n+i)\leq n$, implies $k_{i}= \max_{i<j\leq n}(k_{j}+i+1-j,i).$ This follows a scheme similar to the one used to prove Proposition~\ref{decomp}; we associate to a minimal length $w\in O_{\sigma}\subseteq S_{2n}$ the pair $(m,n),$ where $ m,$ is the number of integers $n+1\leq j\leq 2n$ such that $1\leq w(j)\leq n$ ($0\leq m\leq n$), and prove the statement by induction on $(m,n),$ using the ordering~\eqref{ordering}. The claim is easy to see by direct inspection for the identity $e\in S_{2},$ associated to $(0,1)$ and for $s_{1}=\left(\begin{array}{cc}1 & 2 \\ 2 & 1\end{array}\right),$ associated to $(1,1).$ Assume the proposition hold for minimal length elements associated to pairs $(m',n')<(m,n).$ For the minimal length element $w\in O_{\sigma}\subseteq S_{2n}$ associated to $(m,n),$ consider $w(2n).$ As in Proposition~\ref{decomp}, there are two options
\begin{enumerate}
\item[(1)] $n+1\leq w(2n)\leq 2n.$
\item[(2)] $w(2n)=n.$
\end{enumerate}
If $n+1\leq w(2n)\leq 2n,$ then $k_{n}=2n-w(2n)<n$ and we can, as in the proof of Proposition~\ref{decomp}, consider $$v:=\phi^{-1}(c_{k_{n},n}^{-1}w)=c_{k_{n-1},n-1}\cdots c_{k_{1},1}\in S_{2(n-1)}.$$
If is easy to see that $1\leq v(n-1+i)\leq n-1$ and, as $v$ is associated with $(m,n-1)<(m,n),$ by induction, $$k_{i}= \max_{i<j\leq n-1}(k_{j}+i+1-j,i)=\max_{i<j\leq n}(k_{j}+i+1-j,i)$$
as $k_{n}=2n-w(2n)<n$ and thus $k_{n}+i+1-n\leq i.$ 
\\

If $ w(2n)= n,$ we consider $$v:=t^{-1}c_{n,n}^{-1}w t=c_{k_{n-1},n}c_{k_{n-2},n-1}\cdots c_{k_{1},2}\in S_{2n}.$$
As in Proposition~\ref{decomp}, it follows that $v$ is of minimal length in $O_{v}$ and that $v$ is associated to $(m-1,n)<(m,n).$ Moreover, is it easy to see that $1\leq v(n+i+1)\leq n.$ By induction, if $m_{j}=k_{j-1},$ then
$$
k_{i}=m_{i+1}= \max_{i+1<j\leq n}(m_{j}+(i+1)+1-j,i+1)=\max_{i<j<n}(k_{j}+i+1-j,i+1)=
$$
$$
 \max_{i+1<j\leq n}(m_{j}+(i+1)+1-j,i+1)=\max_{i<j<n}(k_{j}+i+1-j,n+i+1-n)= 
$$
$$
\max_{i<j\leq n}(k_{j}+i+1-j)=\max_{i<j\leq n}(k_{j}+i+1-j,i).
$$
\end{proof}
\subsection{Irreducible $*$-representations of $\mathrm{Pol}(\mathrm{Mat}_{n})_{q}$}

In this section we will complete the proof on the classification of irreducible $*$-representations of $\mathrm{Pol}(\mathrm{Mat}_{n})_{q}$ described by Theorem~\ref{mama} and the second part of Theorem~\ref{main}.
\begin{lem}\label{sunandmoon}
Let $\sigma,\mu\in S_{2n},$ $\sigma\neq \mu,$ and let $\chi_{\varphi},\chi_{\psi}$ be one-dimensional representations of $\mathbb{C}[SU_{n}]_{q},$ for $\varphi,\psi\in [0,2\pi)^{2n}.$ Then 
\begin{enumerate}[(1)]
\item $(\pi_{\sigma}\otimes \chi_{\varphi})\circ \zeta$ is an irreducible $*$-representation of $\mathrm{Pol}(\mathrm{Mat}_{n})_{q}$ if and only if $\sigma$ is the element of minimal length in $O_{\sigma}.$  
\item If $\sigma\in O_{\sigma},\mu\in O_{\mu},$ $\sigma \neq \mu$ are elements of minimal length, then 
$$
(\pi_{\sigma}\otimes \chi_{\varphi})\circ\zeta\not \cong  (\pi_{\mu}\otimes \chi_{\psi})\circ \zeta.
$$
\end{enumerate}
\end{lem}
\begin{rem}
Lemma~\ref{sunandmoon} implies the second part of Theorem~\ref{main}.
\end{rem}
\begin{proof}
$(1)$ Assume that $(\pi_{\sigma}\otimes \chi_{\varphi})\circ \zeta$ is irreducible. Let $w\in O_{\sigma}$ be of minimal length and assume that $\sigma\neq w.$ Then we know from Proposition~\ref{1} that there are $g,h\in S$ so that $\sigma=gwh$ is a minimal decomposition. This gives that $\pi_{\sigma}=\pi_{g}\otimes \pi_{w}\otimes \pi_{h}.$ As minimal decompositions of $g$ and $h$ contain only $s_{i}$ with $1\leq i\leq n-1,$ it follows that  $$\pi_{g}(t_{jk})=\pi_{h}(t_{jk})=\delta_{jk}I$$ if either $n+1\leq j\leq 2n$ or $n+1\leq k\leq 2n.$
From this, it is not hard to see that for $n+1\leq j,k\leq 2n$ we have
$$
\pi_{\sigma}(t_{jk})=(\pi_{g}\otimes \pi_{w}\otimes \pi_{h})(t_{jk})=I\otimes \pi_{w}(t_{jk})\otimes I
$$
and hence $(\pi_{\sigma}\otimes \chi_{\varphi})\circ \zeta=\pi$ is not irreducible.
\\

We will now prove the converse to $(1)$ as well as $(2).$
\\

First, the properties of the minimal element $\sigma\in O_{\sigma}$ implies that $n\leq \sigma(2n)\leq 2n.$ This can be deduced from the expansion $\sigma=c_{k_{n},n}\cdots c_{k_{1},1},$ since $\sigma(2n)=c_{k_{n},n}(2n)=2n-k_{n}.$
Notice also that $(\pi_{\sigma}\otimes \chi_{\alpha})\circ \zeta$ composed with the $*$-automorphism $z_{m}^{j}\mapsto z_{j}^{m}$ is equivalent to $(\chi_{\alpha}\otimes \pi_{\sigma^{-1}})\circ \zeta.$ As $\{s^{-1}|s\in O_{\sigma}\}=O_{\sigma^{-1}} $ and $\ell(s)=\ell(s^{-1}),$ it follows that if $\sigma$ is of minimal length in $O_{\sigma},$ then $\sigma^{-1}$ is of minimal length in $O_{\sigma^{-1}}.$ This makes it possible to reduce the proof to the cases when $\sigma$ has the property that either 
\begin{itemize}
\item $n+1\leq\sigma(2n)\leq 2n$ or 
\item $\sigma(2n)=n$ and $\sigma^{-1}(2n)=n.$
\end{itemize}
We remark that these two cases correspond to the classes $\textbf{A}$ and $\textbf{B}$ respectively.
\\

The proof is by induction on $n\in \mathbb{N}.$ For $n=1,$ it is easily seen to be true as there are only two options $\sigma=\left(\begin{array}{cc}1 & 2 \\ 2 & 1\end{array}\right)$ and the identity element $\sigma=e;$ both are of minimal length and correspond to the Fock representation and the one-dimensional representation respectively. Assume now that the statement holds for $n-1\geq 1.$ Suppose first that $n+1\leq\sigma(2n)\leq 2n.$ If $[k_{n},\dots,k_{1}]$ is an admissible string such that 
$$
\sigma=c_{k_{n},n}\cdots c_{k_{1},1},
$$
then it follows from Proposition~\ref{switch} that $k_{n}<n,$ as $n+1\leq\sigma(2n)=c_{k_{n},n}(2n)=2n-k_{n}.$ By the proof of Proposition~\ref{decomp}, we have a reduced decomposition $\sigma =c_{k_{n}}\phi(s)$ for some $s\in S_{2(n-1)}$ of minimal length in $O_{s}$ and the homomorphism $\phi: S_{2(n-1)}\to S_{2n}$ determined by $\phi(s_{i})= s_{i+1},$ $i=1, \dots ,2n-3$ (see~\eqref{phis}).
It follows that \begin{equation}\label{silverspoon}\pi_{\sigma}\cong \pi_{c_{k_{n},n}}\otimes \pi_{\phi(s)}\end{equation} and hence for $j=1,\dots, n,$ we can determine
$$
\pi(z_{j}^{n}):=(\pi_{c_{k_{n},n}}\otimes \pi_{\phi(s)}\otimes \chi_{\alpha})\circ \zeta(z_{j}^{n})=\alpha_{2n}\pi_{c_{k_{n}}}(t_{n+j,2n})\otimes I=
$$
\begin{equation}\label{casies}
=e^{i\alpha_{2n}}\times
\begin{cases}
\underbrace{I\otimes \cdots \otimes I}_{\text{$k_{n}-n+j-1$ factors}} \otimes T_{22}\otimes \underbrace{T_{12}\otimes \cdots \otimes T_{12}}_{\text{$n-j$ factors}} \otimes  I, & \text{if $n-k_{n}< j \leq 2n$}\\
\underbrace{T_{12}\otimes \cdots \otimes T_{12}}_{\text{$k_{n}$ factors}} \otimes  I, & \text{if $j=n-k_{n}$}\\
0& \text{otherwise.} 
\end{cases}
\end{equation}
From~\eqref{casies}, we see that 
\begin{equation}\label{childsballoon}\langle e_{\textbf{0}}\rangle \otimes \ell^{2}(\mathbb{Z}_{+})^{\otimes\ell(\phi(s))}=\left(\ell^{2}(\mathbb{Z}_{+})^{\otimes k_{n}}\otimes \ell^{2}(\mathbb{Z}_{+})^{\otimes\ell(\phi(s))}\right)\cap\left(\cap_{j=k+1}^{n}\ker\pi(z_{j}^{n})^{*}\right)=:H.\end{equation}
We can explicitly determine $\pi_{c_{k_{n},n}}(t_{mj}),$ using Lemma~\ref{coopout} and the diagram
\begin{equation}
\begin{tikzpicture}[thick,scale=1.6]
\draw[step=1.0,black,thick] (1,1) grid (7,2);
\node at (1.5,0.5) {$\dots$};
\node at (2.5,0.5) {$n-k_{n}$};
\node at (3.5,0.5) {$n-k_{n}+1$};
\node at (4.5,0.5) {$\dots$};
\node at (5.5,0.5) {$\dots$};
\node at (6.5,0.5) {$n$};
\node at (0.5,1.5) {$1$};

\node at (7.5,1.5) {$n$};
\node at (1.5,2.4) {$1$};
\node at (2.5,2.4) {$\dots$};
\node at (3.5,2.4) {$n-k_{n}$};
\node at (4.5,2.4) {$n-k_{n}+1$};
\node at (5.5,2.4) {$\dots$};
\node at (6.5,2.4) {$n-1$};
\filldraw[fill=black!50!white, draw=black](1,2) -- (1,1) -- (2,1)--(2,2)--(1,2)--(1,1);
\filldraw[fill=black!50!white, draw=black](2,2) -- (2,1) -- (3,1)--(3,2)--(2,2)--(2,1);
%\draw [->] (0.5+3,0.05+1)--(0.5+3,0.5+1)--(0.95+3,0.5+1);\draw [->] (0.05+4,0.5+1)--(0.95+4,0.5+1);\draw [->] (0.05+5,0.5+1)--(0.95+5,0.5+1);\draw [->] (0.05+6,0.5+1)--(0.5+6,0.5+1)--(0.5+6,0.95+1);
\end{tikzpicture}
\end{equation}
We have that for $e_{\textbf{0}}\in \ell^{2}(\mathbb{Z}_{+})^{\otimes k_{n}}$ and $j\neq 2n$ 
\begin{equation}\label{break}
\pi_{c_{k_{n},n}}(t_{n+m,j})e_{\textbf{0}}=
\begin{cases}
\delta_{n+m,j} e_{\textbf{0}} ,& \text{if $1\leq m<n-k_{n}$ and $j\neq 2n$,}\\
\delta_{n+m,j+1}e_{\textbf{0}},& \text{if $n-k_{n}< m \leq n$ and $j\neq 2n$,}\\
0,& \text{if $ m=n-k_{n}$ and $j\neq 2n$.}\\
\end{cases}
\end{equation}
The action of $\pi_{c_{k_{n},n}}(t_{k,2n})e_{\textbf{0}}$ is easily identified from~\eqref{casies}. 
We can define the $*$-representation $\pi':\mathrm{Pol}(\mathrm{Mat}_{n-1})_{q}\to B(H)$ by the formula
\begin{equation}\label{toosoon}
\pi'(z_{m}^{j})=\begin{cases}
-q \pi(z_{m}^{j})|_{H} & \text{If $1\leq m< n-k_{n}$}\\           
\pi(z_{m+1}^{j})|_{H}  & \text{If $n-k_{n}\leq m\leq n-1 $}\\
\end{cases}\end{equation}

(see~\eqref{pii}). If we let 
$$
\psi(t_{k,j})= 
\begin{cases}
t_{k-1,j-1}& \text{ if $2\leq k,j\leq 2n-1$}\\
\delta_{k,j}& \text{otherwise},
\end{cases}
$$  then it is easy to see that for any $t\in S_{2(n-1)}$
\begin{equation}\label{phipsi}
\pi_{\phi(t)}\cong \pi_{t}\circ \psi.
\end{equation}
It then follows from~\eqref{break} that for $z_{m}^{j}\in \mathrm{Pol}(\mathrm{Mat}_{n})_{q},$ $m,j=1,\dots n-1$
$$
\pi'(z_{m}^{j})=-q \pi(z_{m}^{j})|_{H}
$$
$$
=(-q)^{m-n+1}e^{i\alpha_{n+j}}\sum_{l=1}^{n}\pi_{c_{k_{n},n}}(t_{n+m,l})\otimes \pi_{\phi(s)}(t_{l,n+j})|_H=
$$
$$
e^{i\alpha_{n+j}}(-q)^{m-(n-1)}I\otimes \pi_{\phi(s)}(t_{(n-1)+(m+1),(n-1)+(j+1)})|_{H}=
$$
$$
e^{i\alpha_{n+j}}(-q)^{m-(n-1)}I\otimes \pi_{s}(\psi(t_{(n-1)+(m+1),(n-1)+(j+1)}))|_{H}=
$$
$$
e^{i\alpha_{n+j}}(-q)^{m-(n-1)}I\otimes \pi_{s}(t_{(n-1)+m,(n-1)+j)}))|_{H},
$$
if $1\leq m< n-k_{n}$ and 
$$
\pi'(z_{m}^{j})=\pi(z_{m+1}^{j})|_{H}
$$
$$
=(-q)^{m+1-n}e^{i\alpha_{n+j}}\sum_{l=1}^{n}\pi_{c_{k_{n},n}}(t_{n+m+1,l})\otimes \pi_{\phi(s)}(t_{lj})|_H=
$$
$$
e^{i\alpha_{n+j}}(-q)^{m-(n-1)}I\otimes \pi_{\phi(s)}(t_{(n-1)+(m+1),(n-1)+(j+1)})|_{H}=
$$
$$
e^{i\alpha_{n+j}}(-q)^{m-(n-1)}I\otimes \pi_{s}(\psi(t_{(n-1)+(m+1),(n-1)+(j+1)}))|_{H}=
$$
\begin{equation}\label{A}
e^{i\alpha_{n+j}}(-q)^{m-(n-1)}I\otimes \pi_{s}(t_{(n-1)+m,(n-1)+j)}))|_{H},
\end{equation}
if $n-k_{n}\leq m \leq n-1.$ It follows that \begin{equation}\label{riget}\pi'\cong(\pi_{s}\otimes \chi_{\bar\alpha})\circ \zeta.\end{equation} Let now $(\pi_{\sigma}\otimes \chi_{\alpha})\circ \zeta$ and $(\pi_{\mu}\otimes \chi_{\beta})\circ\zeta$ be equivalent $*$-representations via a unitary operator $U$. Then if 
$$
\begin{array}{cc}\sigma=c_{k_{n},n}\cdots c_{k_{1},1},& \mu=c_{m_{n},n}\cdots c_{m_{1},1},\end{array}
$$
it follows from~\eqref{casies} that $k_{n}=m_{n}.$ Thus if $k:=k_{n}=m_{n}<n,$ it follows that $\sigma(2n)=2n-k=\mu(2n)$ and hence $n+1\leq \sigma(2n)=\mu(2n)\leq 2n.$ Hence
$$\begin{array}{cc}\sigma=c_{k,n}\phi(s),& \mu=c_{k,n}\phi(t)\end{array},$$
for $s,t\in S_{2(n-1)}$ of minimal lengths. As $U$ will map 
\begin{equation}\label{nosence}
\ker ((\pi_{\sigma}\otimes \chi_{\alpha})\circ \zeta)(z_{m}^{j})^{*}\to \ker((\pi_{\mu}\otimes \chi_{\beta})\circ\zeta)(z_{m}^{j})^{*}
\end{equation}
it follows from~\eqref{riget} that there are $\chi_{\bar\alpha},\chi_{\bar\beta}$ such that
$$
(\pi_{s}\otimes \chi_{\bar\alpha})\circ \zeta\cong(\pi_{t}\otimes \chi_{\bar\beta})\circ \zeta.
$$
By induction $s=t$ and hence $\sigma=\mu.$
\\

 To see that $\pi$ is irreducible, we observe that by~\eqref{casies}, any non-zero operator in $B(\ell^{2}(\mathbb{Z}_{+})^{\otimes k}\otimes \ell^{2}(\mathbb{Z}_{+})^{\ell(s)})$ commuting with the range of $\pi=(\pi_{c_{k,n}}\otimes \pi_{\phi(s)}\otimes \chi_{\beta})\circ \zeta$ can be written in the form $I\otimes A,$ for $A\in B(\ell^{2}(\mathbb{Z}_{+})^{\otimes\ell(s)}).$ Restricting $I\otimes A$ to $$H=\langle e_{\textbf{0}}\rangle \otimes \ell^{2}(\mathbb{Z}_{+})^{\otimes\ell(s)}\cong \ell^{2}(\mathbb{Z}_{+})^{\otimes\ell(s)}$$ gives that $A$ commutes with the range of $\pi'\cong(\pi_{s}\otimes \chi_{\bar\alpha})\circ \zeta$ and hence by induction, $A$ must be a constant multiple of $I.$
\\

Assume now that $\sigma(2n)=n$ and $\sigma^{-1}(2n)=n.$ We claim that in this case we have 
\begin{equation}\label{sorry}\begin{array}{cccc}\sigma=c_{k_{n},n}\cdots c_{k_{1},1},& k_{n}=n, & k_{j}\geq 1,& j=1,\dots,n-1\end{array}.\end{equation}
That $k_{n}=n$ follows from $n=\sigma(2n)=c_{k_{n},n}(2n)=2n-k_{n}.$ To see that all $k_{j}$ are non-zero, notice that $$\begin{array}{ccc}c_{j,k}(m)\leq m+1,&j,k=1,\dots,n,&m=1,\dots , 2n\end{array}$$ and hence, as $c_{k_{n},n}\cdots c_{k_{1},1}(n)=n+n,$ we obtain the claim.
\\

Recall the $*$-homomorphism $\delta:\mathbb{C}[SU_{2n}]_{q}\to \mathbb{C}[SU_{2(n-1)}]_{q}$ from Section $5.4$ given by
$$\delta(t_{kj})=
\begin{cases}
t_{kj} & \text{ if $1\leq k,j\leq 2n-2$}\\ \delta_{k,j}I & \text{ otherwise}
\end{cases}$$
and let $\gamma:S_{2(n-1)}\to S_{2n}$ be the homomorphism determined by $s_{i}\to s_{i},$ $i=1,\dots , 2(n-1).$ For $s\in S_{2(n-1)},$ it follows that
\begin{equation}\label{fine}
\pi_{s}\circ \delta\cong \pi_{\gamma(s)}.
\end{equation}
Let now $s\in S_{2(n-1)}$ be of minimal length in $O_{s}.$ with decomposition
$$
s=c_{k_{n-1},n-1}\cdots c_{k_{1},1}.
$$
Notice that by~\eqref{cycles}, the cycle $c_{k,j}\in S_{2(n-1)}$ has image $$\gamma(c_{k,j})=\gamma(s_{j+n-1-k}\cdots s_{j+n-1-1})=s_{j+n-1-k}\cdots s_{j+n-1-1}=c_{k+1,j}s_{j+n-1}\in S_{2n}.$$
Thus, we can write \begin{equation}\label{cain}\gamma(s)=(c_{k_{n-1}+1,n-1}s_{2n-2})\cdots (c_{k_{1}+1,1}s_{n})=(c_{k_{n-1}+1,n-1}\cdots c_{k_{1}+1,1})\cdot (s_{n}\cdots s_{2n-2})
\end{equation} as $$c_{k_{j}+1,j}s_{j+n-1}=s_{j+n-k_{j}-1}\cdots s_{j+n-2}$$ commutes with $s_{m}$ if $m>j+n-1.$
\\

Next we note that if $[k_{n-1},\dots,k_{1}]$ is an admissible string, then so is $$[n,k_{n-1}+1,\dots , k_{1}+1].$$In fact, if not, let $$m_{j}=\begin{cases}k_{j}+1, &\text{ for $1\leq j< n$}\\ n, & \text{ if $j=n$.}\end{cases}$$ Then there exists $1\leq i<n$ (note that for $i=n$~\eqref{didi} reads as $k_{n}\leq n$), such that
$$
m_{i}>\max_{i<j\leq n}(m_{j}+i+1-j,i)
$$
and as $m_{n}+i+1-n=i+1,$ it follows that for $m_{i}=k_{i}+1$
$$
\begin{array}{cc}k_{i}+1>\max_{i<j<n}(k_{j}+2+i-j,i), & k_{i}+1>1+i\end{array}
$$
i.e. for all integers $i<j < n$ we have

\begin{enumerate}[(1)]\item $k_{i}+1>k_{j}+2+ i-j$ and \item $k_{i}+1> i+1.$ \end{enumerate}
As $[k_{n-1},\dots, k_{1}]$ is an admissible string, we have $k_{i}\leq \max_{i<j\leq n-1}(k_{j}+1+i-j,i),$ and since $(1)$ holds, we must have $k_{i}\leq i.$ But this contradicts $k_{i}+1>i+1.$ Hence $[n,k_{n-1}+1,\dots , k_{1}+1]$ is an admissible string. 
\\

Conversely, if $[k_{n},k_{n-1},\dots,k_{1}]$ is an admissible string such that $k_{n}=n$ and $k_{j}\geq 1$ for $j=1,\dots,n-1,$ then also $[k_{n-1}-1,\dots,k_{1}-1]$ is an admissible string. This follows from
$$
k_{i}-1 \leq \max_{i<j\leq n}(k_{j}+i+1-j,i)-1=\max_{i<j < n}(k_{j}+i+1-j,i+1)-1=
$$
$$
=\max_{i<j < n}(k_{j}+i+1-j-1,i+1-1)=\max_{i<j \leq  n-1}((k_{j}-1)+i+1-j,i)
$$
for $1\leq j <n.$
\\

This gives that every $\sigma\in S_{2n}$ of minimal length, such that the decomposition $\sigma=c_{k_{n},n}\cdots c_{k_{1},1}$ satisfies~\eqref{sorry}, can be written in a reduced form as
\begin{equation}\label{darkness}
\sigma=c_{n,n}\gamma(s)(s_{n}\cdots s_{2n-2})^{-1}=c_{n,n}\gamma(s)c_{n-1,n-1}^{-1}\end{equation}
for a unique $s\in S_{2(n-1)}$ of minimal length. Conversely, for any such $s\in S_{2(n-1)},$ the element on the right-hand side of~\eqref{darkness} is of minimal length. As 
$$
c_{n,n}\gamma(s)c_{n-1,n-1}^{-1}(2n)=c_{n,n}(2n)=n
$$
$$
c_{n,n}\gamma(s)c_{n-1,n-1}^{-1}(n)=c_{n,n}\gamma(s)(2n-1)=c_{n,n}(2n-1)=2n
$$
it thus follows that $\sigma(2n)=n$ and $\sigma^{-1}(2n)=n$ hold if and only if~\eqref{darkness} holds.
\\

Let now $\sigma=c_{n,n}\gamma(s)c_{n-1,n-1}^{-1}\in S_{2n}$ be such element, then for any $\chi_{\alpha}$ we let
$$\pi:\mathrm{Pol}(\mathrm{Mat}_{n})_{q}\to B(\ell^{2}(\mathbb{Z}_{+})^{\otimes n}\otimes \ell^{2}(\mathbb{Z}_{+})^{\otimes \ell(\gamma(s))}\otimes \ell^{2}(\mathbb{Z}_{+})^{\otimes n-1})$$ be the $*$-representation
\begin{equation}\label{knightmare}
(\pi_{\sigma}\otimes \chi_{\alpha})\circ \zeta\cong (\pi_{c_{n,n}}\otimes \pi_{\gamma(s)}\otimes \pi_{c_{n-1,n-1}^{-1}}\otimes \chi_{\alpha})\circ \zeta=:\pi.
\end{equation}
A similar calculation as in~\eqref{casies} gives that
\begin{equation}\label{Bspace}(\cap_{k=1}^{n}\ker \pi(z_{n}^{k})^{*})\cap (\cap_{k=1}^{n}\ker \pi(z_{k}^{n})^{*})=
\langle e_{\textbf{0}}\rangle\otimes \ell^{2}(\mathbb{Z}_{+})^{\otimes \ell(\gamma(s))} \otimes \langle e_{\textbf{0}}\rangle=:H.\end{equation}
We can then define a $*$-representation $\pi'$ of $\mathrm{Pol}(\mathrm{Mat}_{n-1})_{q}$ by 
\begin{equation}\label{handmade}
\begin{array}{cc}\pi':z_{j}^{i}\mapsto \pi(z_{j}^{i})|_{H},&i,j=1,\dots, n-1.\end{array}\end{equation}
In fact,~\eqref{break} together with
\begin{equation}\label{ofnoon}
\begin{array}{ccc}\pi_{c_{n-1,n-1}^{-1}}(t_{m,n+j})e_{\textbf{0}}=
(-q)\delta_{m+1,n+j}e_{\textbf{0}},& \text{if $1\leq j <n$ and $m\neq 2n-1$,}
\end{array}
\end{equation}
and $$\begin{array}{ccc}\pi_{\gamma(s)}(t_{ij})=\delta_{ij}I, & \text{if either $i\in \{2n-1,2n\}$ or $j\in \{2n-1,2n\}$} \end{array}$$ give for $m,j=1,\dots ,n-1$
$$
\pi(z_{m}^{j})|_{H}=e^{i\alpha_{n+j}}(-q)^{m-n}\sum_{l,i=1}^{2n}\pi_{c_{n,n}}(t_{n+m,l})\otimes \pi_{\gamma(s)}(t_{li})\otimes \pi_{c_{n-1,n-1}^{-1}}(t_{i,n+j})|_{H}=
$$
$$
e^{i\alpha_{n+j}}(-q)^{m-n}\sum_{l,i=1}^{2n-2}\pi_{c_{n,n}}(t_{n+m,l})\otimes \pi_{\gamma(s)}(t_{li})\otimes \pi_{c_{n-1,n-1}^{-1}}(t_{i,n+j})|_{H}+
$$
$$
+\underbrace{e^{i\alpha_{n+j}}(-q)^{m-n}\sum_{l=2n-1}^{2n}\pi_{c_{n,n}}(t_{n+m,l})\otimes I\otimes \pi_{c_{n-1,n-1}^{-1}}(t_{l,n+j})|_{H}}_{\text{$=0$ by~\eqref{break} and~\eqref{ofnoon}}}=
$$
$$
=e^{i\alpha_{n+j}}(-q)^{m-n+1}\sum_{l,i=1}^{2n-2}\delta_{n+m,l+1}I\otimes \pi_{\gamma(s)}(t_{li})\otimes \delta_{i+1,n+j}I|_{H}=
$$
$$
=e^{i\alpha_{n+j}}(-q)^{m-n+1}I\otimes \pi_{\gamma(s)}(t_{n+m-1,n+j-1})\otimes I|_{H}=
$$
\begin{equation}\label{B}
e^{i\alpha_{n+j}}(-q)^{m-n+1}I\otimes \pi_{s}(t_{(n-1)+m,(n-1)+j})\otimes I|_{H}.
\end{equation}
Hence \begin{equation}\label{trying}\pi'\cong (\pi_{s}\otimes\chi_{\bar\alpha})\circ \zeta\end{equation} for some $\chi_{\bar\alpha}.$ We can then again use induction, as in the first case, to get that $$(\pi_{\sigma}\otimes \chi_{\alpha})\circ \zeta\cong(\pi_{\mu}\otimes \chi_{\beta})\circ\zeta$$ implies that $\sigma=\mu$ and also that $(\pi_{\sigma}\otimes \chi_{\alpha})\circ \zeta$ is irreducible for every $\chi_{\alpha}.$
\end{proof}

\begin{prop}\label{makon}
Let $\sigma\in O_{\sigma}$ be the element of minimal length and for $\varphi,\theta\in[0,2\pi)^{2n},$ let $\chi_{\varphi},\chi_{\theta}$ be one-dimensional representations of $\mathbb{C}[SU_{2n}]_{q}.$ Then 
$$
(\pi_{\sigma}\otimes \chi_{\varphi})\circ \zeta \cong (\pi_{\sigma}\otimes\chi_{\theta})\circ \zeta
$$
if and only if $\varphi_{i}=\theta_{i}$ for those $i$ such that 
$$
n+1\leq \sigma(i)\leq 2n.
$$
\end{prop}
\begin{proof}
First, in the case of the Fock representation, we have by its uniqueness that $$\pi_{F,n}\cong \pi_{s}\circ \zeta\cong (\pi_{s}\otimes \chi_{\varphi})\circ \zeta$$ for any one-dimensional $*$-representation $\chi_{\varphi}.$ As in this case \begin{equation}\label{blade}s=\left(\begin{array}{cccccccc}1  &\dots& n & n+1  &\dots  & 2n\\ n+1  & \dots & 2n & 1 & \dots  &n\end{array}\right)\end{equation} we therefore have
$$\begin{array}{cc}1\leq s(n+j)\leq n, &\text{for all $j=1,\dots ,n$.}\end{array}$$
Assume now that $\sigma\neq s.$ As $\sigma$ is the minimal length element in $O_{\sigma},$ we have from the proof of uniqueness in Proposition~\ref{1} that $\sigma(1)=1.$ Let $\{i_{1},\dots,i_{m}\}$ be the integers such that $n+1\leq i_{j}\leq 2n$ and $1\leq \sigma(i_{j})\leq n,$ $j=1,\dots ,m$ and let $\beta\in [0,2\pi)^{2n}$ be defined by the formula
\begin{equation}
\beta_{i}=\begin{cases}\varphi_{i}, & \text{ if $i\in \{i_{1},\dots,i_{m}\}$}\\ 0, & \text{ if $i\notin \{i_{1},\dots,i_{m}\}\cup \{1\}$}\\ -\sum_{j=1}^{m}\varphi_{i_{j}}\pmod{2\pi}, & \text{ if $i=1$} \end{cases}
\end{equation}
It follows that $\chi_{\varphi}= \chi_{\beta}\otimes \chi_{\alpha}$ for some $\alpha\in [0,2\pi)^{2n}$ with $\alpha_{i_{j}}=0$ for $j=1,\dots, m,$ and thus by~\eqref{twist}
$$
\pi\cong (\pi_{\sigma}\otimes \chi_{\varphi})\circ \zeta =(\pi_{\sigma}\otimes \chi_{\beta}\otimes \chi_{\alpha})\circ \zeta\cong (\chi_{\sigma^{-1}(\beta)}\otimes \pi_{\sigma}\otimes \chi_{\alpha})\circ \zeta=(\pi_{\sigma}\otimes \chi_{\alpha})\circ \zeta
$$
where the last equality follows from the fact that $\sigma^{-1}(\beta)_{i}=0$ for $i=n+1,\dots, 2n.$ Clearly by~\eqref{chi}, the $*$-representation $(\pi_{\sigma}\otimes \chi_{\alpha})\circ \zeta$ does not depend on $\alpha_{i}$ for $i=1,\dots, n.$ Hence, by Proposition~\ref{switch}, if $\sigma=c_{k_{n},n}\cdots c_{k_{1},1}$ then $(\pi_{\sigma}\otimes \chi_{\alpha})\circ \zeta$ only depends on $\alpha_{n+i}$ such that $k_{i}<\max_{i<j\leq n}(k_{j}+i+1-j,i).$
\\

We now prove that if two $*$-representations
$$
\begin{array}{cc}\pi_{1}:=(\pi_{\sigma}\otimes \chi_{\varphi})\circ \zeta, & \pi_{2}:=(\pi_{\sigma}\otimes \chi_{\theta})\circ \zeta\end{array}
$$
are equivalent, then $\varphi_{n+i}= \theta_{n+i}$ for any $n+i$ such that $n+1\leq \sigma(n+i)\leq 2n.$ 
\\

Again, we prove this by induction. This can be seen to hold for $n=1,$ simply by inspection. Let $U$ be an unitary intertwining  $\pi_{1}$ and $ \pi_{2}.$ Let us first assume that $n+1\leq \sigma(2n)\leq 2n$ and let $H_{1},H_{2}$ be given by~\eqref{childsballoon} for $\pi_{1}$ and $ \pi_{2}$ respectively. By~\eqref{nosence}, the restriction of $U$ is then a unitary map $H_{1}\to H_{2}.$ It follows from~\eqref{casies} that we must have $\varphi_{2n}=\theta_{2n}.$ Moreover, if $(\pi_{s}\otimes \chi_{\bar\varphi})\circ \zeta$ and $(\pi_{s}\otimes \chi_{\bar\theta})\circ\zeta$ are the restrictions of $\pi_{1}$ and $\pi_{2}$ to $H_{1}$ and $H_{2}$ respectively, given by~\eqref{toosoon}, if follows from~\eqref{A} that we can assume \begin{equation}\begin{array}{ccc}\bar\alpha_{n-1+j}=\alpha_{n+j},& \bar\theta_{n-1+j}=\theta_{n+j},& j=1,\dots, n-1.\end{array}\end{equation}
By induction, it follows that $$\alpha_{n+j}=\bar\alpha_{n-1+j}=\bar\theta_{n-1+j}=\theta_{n+j}$$ if $n\leq s(n-1+j)\leq 2n-2.$ But since $\sigma=c_{k_{n},n}\phi(s),$ by the proof of Proposition~\ref{decomp} and~\eqref{A}, it follows from $c_{k_{n},n}(\{n+1,\dots,2n\})=\{n+1,\dots,2n\}$ and the definition of $\phi:S_{2(n-1)}\to S_{2n},$ that for $j=1,\dots ,n-1$ we have $$n+1\leq \sigma(n+j)\leq 2n $$ if and only if $$ n\leq s(n+j)\leq 2n-2.$$

Consider now the case when $\sigma(2n)=n$ and $\sigma^{-1}(2n)=n.$ Denote by $H_{1}$ and $H_{2}$ the subspaces defined by~\eqref{Bspace} for $\pi_{1}$ and $\pi_{2}$ respectively. As before, the restriction of $U$ is a unitary from $H_{1}$ to $H_{2}$ intertwining the restricted representations of $\pi_{1}$ and $\pi_{2},$ the latter equivalent to $(\pi_{s}\otimes \chi_{\bar\varphi})\circ \zeta$ and $(\pi_{s}\otimes \chi_{\bar\theta})\circ\zeta$ respectively. By~\eqref{B}, we can assume that $\bar\alpha_{n-1+j}=\alpha_{n+j}$ and $\bar\theta_{n-1+j}=\theta_{n+j}$ for $j=1,\dots,n-1.$ It follows again by induction that $$\alpha_{n+j}=\bar\alpha_{n-1+j}=\bar\theta_{n-1+j}=\theta_{n+j}$$ if $n\leq s(n-1+j)\leq 2n-2.$ By~\eqref{B} and~\eqref{darkness}, we have $\sigma=c_{n,n}\gamma(s)c_{n-1,n-1}^{-1}.$ For $j=1,\dots n-1,$ we have $c_{n-1,n-1}^{-1}(n+j)=n-1+j$ and $c_{n,n}(n-1+j)=n+j.$ Hence it follows from the definition of $\gamma$ that for $j=1,\dots,n-1$ we have
$$
n+1\leq \sigma(n+j)\leq 2n
$$ 
if and only if
$$
n\leq s(n-1+j)\leq 2n-2.
$$
This completes the proof.
\end{proof}
\begin{proof}[Proof of Theorem~\ref{mama}]
For each admissible $\textbf{k}_{\varphi}$ we have that $\pi_{\textbf{k}_{\varphi}}$ is irreducible, which follows from Lemma~\ref{sunandmoon}. We now prove that the map $\textbf{k}_{\varphi}\to [\pi_{\textbf{k}_{\varphi}}]$ is a bijection.
\\

Recall the $C^{*}$-algebra homomorphism $\tau_{\varphi}:C^{*}(S)\to \mathbb{C}$ determined by $\tau_{\varphi}(S)=e^{i\varphi},$ for $\varphi\in [0,2\pi).$ It is not hard to see that for all $i=1,\dots,2n-1,$ the composition $\tau_{0}\circ \pi_{i}=\epsilon:\mathbb{C}[SU_{2n}]_{q}\to \mathbb{C}$ is the co-unit of $\mathbb{C}[SU_{2n}]_{q}$. We can thus write  $\pi_{\sigma}$ as
$$
\pi_{\sigma}=\pi_{c_{k_{n},n}}\otimes \pi_{c_{k_{n-1},n-1}}\otimes \cdots \otimes \pi_{c_{k_{1},1}}=
$$
$$
\begin{array}{cc}
\Bigg(\left(\underbrace{\tau_{0}\otimes \cdots \otimes\tau_{0}}_{\text{$n-k_{n}$ times}}\otimes \underbrace{I\otimes \cdots\otimes I}_{\text{$k_{n}$ times}}\right)\otimes \left(\underbrace{\tau_{0}\otimes \cdots \otimes\tau_{0}}_{\text{$n-k_{n-1}$ times}}\otimes \underbrace{I\otimes \cdots\otimes I}_{\text{$k_{n-1}$ times}}\right)\otimes \cdots \otimes \\
\left(\underbrace{\tau_{0}\otimes \cdots \otimes\tau_{0}}_{\text{$n-k_{1}$ times}}\otimes \underbrace{I\otimes \cdots\otimes I}_{\text{$k_{1}$ times}}\right)\Bigg)\circ (\pi_{c_{n,n}}\otimes \pi_{c_{n,n-1}}\otimes \cdots \otimes \pi_{c_{n,1}})=\\
\Bigg(\left(\underbrace{\tau_{0}\otimes \cdots \otimes\tau_{0}}_{\text{$n-k_{n}$ times}}\otimes \underbrace{I\otimes \cdots\otimes I}_{\text{$k_{n}$ times}}\right)\otimes \left(\underbrace{\tau_{0}\otimes \cdots \otimes\tau_{0}}_{\text{$n-k_{n-1}$ times}}\otimes \underbrace{I\otimes \cdots\otimes I}_{\text{$k_{n-1}$ times}}\right)\otimes \cdots \otimes \\
\left(\underbrace{\tau_{0}\otimes \cdots \otimes\tau_{0}}_{\text{$n-k_{1}$ times}}\otimes \underbrace{I\otimes \cdots\otimes I}_{\text{$k_{1}$ times}}\right)\Bigg)\circ  (\pi_{c_{n,n}}\otimes \pi_{c_{n,n-1}}\otimes \cdots \otimes \pi_{c_{n,1}})
\end{array}
$$
and $\pi_{c_{n,j}}$ corresponds to the $j$'th row of the grid~\eqref{boxer}. As it was explained in section $3.1,$ the application of $\tau_{0}$ to a box is visualized by coloring that box dark gray. Assume now that $\pi\cong (\pi_{\sigma}\otimes \chi_{\beta})\circ \zeta,$ then, as we showed above, we can assume $\beta_{i}\neq 0$ if and only if $i$ is such that $k_{i}<\max_{i<j\leq n}(k_{j}+i+1-j,i).$ In particular, we will have $k_{i}<n$ for such $i$'s and hence for this row we can write 
\begin{equation}\label{inheaven}
(\tau_{\alpha_{i}}\circ \pi_{n+i-k_{i}-1})\otimes \pi_{c_{k_{i},i}}=\left(\underbrace{\tau_{0}\otimes \cdots \otimes\tau_{0}}_{\text{$n-k_{i}-1$ times}}\otimes  \tau_{\alpha_{i}} \otimes \underbrace{I\otimes \cdots\otimes I}_{\text{$k_{i}$ times}}\right)\circ \pi_{c_{n,i}}.
\end{equation}
Visually, this row corresponds to the $1\times n$ block
$$
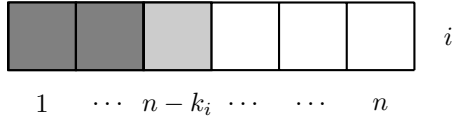
\begin{tikzpicture}[thick,scale=0.9]
\draw[step=1.0,black,thick] (1,1) grid (7,2);
\node at (1.5,0.5) {$1$};
\node at (2.5,0.5) {$\dots$};
\node at (3.5,0.5) {$n-k_{i}$};
\node at (4.5,0.5) {$\dots$};
\node at (5.5,0.5) {$\dots$};
\node at (6.5,0.5) {$n$};
\node at (7.5,1.5) {$i$};

\filldraw[fill=black!50!white, draw=black](1,2) -- (1,1) -- (2,1)--(2,2)--(1,2)--(1,1);
\filldraw[fill=black!50!white, draw=black](2,2) -- (2,1) -- (3,1)--(3,2)--(2,2)--(2,1);
\filldraw[fill=black!20!white, draw=black](3,2) -- (3,1) -- (4,1)--(4,2)--(3,2)--(3,1);
\end{tikzpicture}
$$
The $*$-homomorphism $\tau_{\varphi}\circ \pi_{i}$ is easily seen to be $\chi_{\varphi^{(i)}},$ where $$\varphi^{(i)}_{j}=\begin{cases}-\varphi, & \text{ if $j=i$}\\ \varphi, & \text{ if $j=i+1$}\\ 0, & \text{ otherwise} \end{cases}$$

By using~\eqref{twist}, we see that we can write~\eqref{inheaven} as $\chi_{\varphi^{(n+i-k_{i}-1)}}\otimes \pi_{c_{k_{i},i}}$ and \begin{equation}\label{constance}\begin{array}{ccc}\chi_{\varphi^{(n+i-k_{i}-1)}}\otimes \pi_{c_{k_{i},i}}(t_{k,n+ i})=e^{i\varphi}\pi_{c_{k_{i},i}}(t_{k,n+ i}),& k=1,\dots, 2n.\end{array}\end{equation}
If we now have an irreducible $*$-representation $(\pi_{\sigma}\otimes \chi_{\beta})\circ \zeta,$ with $\sigma=c_{k_{n},n}\cdots c_{k_{1},1},$ then consider the $*$-representation of $\mathrm{Pol}(\mathrm{Mat}_{n})_{q}$ defined as
\begin{equation}\label{creative}
\begin{array}{cc}
\Bigg(\left(\underbrace{\tau_{0}\otimes \cdots \otimes\tau_{0}}_{\text{$n-k_{n}-1$ times}}\otimes  \tau_{\alpha_{n}} \otimes \underbrace{\iota\otimes \cdots\otimes \iota}_{\text{$k_{n}$ times}}\right)\otimes \left(\underbrace{\tau_{0}\otimes \cdots \otimes\tau_{0}}_{\text{$n-k_{n-1}-1$ times}}\otimes  \tau_{\alpha_{n-1}} \otimes \underbrace{\iota\otimes \cdots\otimes \iota}_{\text{$k_{n-1}$ times}}\right)\otimes  \\
\cdots \otimes\left(\underbrace{\tau_{0}\otimes \cdots \otimes\tau_{0}}_{\text{$n-k_{1}-1$ times}}\otimes  \tau_{\alpha_{1}} \otimes \underbrace{\iota\otimes \cdots\otimes \iota}_{\text{$k_{1}$ times}}\right)\Bigg)\circ (\pi_{c_{n,n}}\otimes \pi_{c_{n,n-1}}\otimes \cdots \otimes \pi_{c_{n,1}}),
\end{array}
\end{equation}
For $j>i$ we have $\pi_{c_{n,i}}(t_{k,n+j})=\delta_{k,n+j}I$ and hence for $j=1,\dots, n$
$$
(\pi_{c_{n,n}}\otimes \pi_{c_{n,n-1}}\otimes \cdots \otimes \pi_{c_{n,1}})(t_{k,n+j})=
$$
\begin{equation}\label{starrep}
=(\pi_{c_{n,n}}\otimes \pi_{c_{n,n-1}}\otimes \cdots \otimes \pi_{c_{n,j}})(t_{k,n+j})\otimes I.
\end{equation}
If we combine~\eqref{constance} with~\eqref{starrep}, then it follows that~\eqref{creative} applied to $t_{k,n+j}$ becomes
$$
\begin{array}{cc}
\Bigg(\left(\underbrace{\tau_{0}\otimes \cdots \otimes\tau_{0}}_{\text{$n-k_{n}-1$ times}}\otimes  \tau_{\alpha_{n}} \otimes \underbrace{\iota\otimes \cdots\otimes \iota}_{\text{$k_{n}$ times}}\right)\otimes \left(\underbrace{\tau_{0}\otimes \cdots \otimes\tau_{0}}_{\text{$n-k_{n-1}-1$ times}}\otimes  \tau_{\alpha_{n-1}} \otimes \underbrace{\iota\otimes \cdots\otimes \iota}_{\text{$k_{n-1}$ times}}\right)\otimes  \\
\cdots \otimes\left(\underbrace{\tau_{0}\otimes \cdots \otimes\tau_{0}}_{\text{$n-k_{j}-1$ times}}\otimes  \tau_{\alpha_{j}} \otimes \underbrace{\iota\otimes \cdots\otimes \iota}_{\text{$k_{j}$ times}}\right)\Bigg)\circ (\pi_{c_{n,n}}\otimes \pi_{c_{n,n-1}}\otimes \cdots \otimes \pi_{c_{n,j}})(t_{k,n+j})=
\end{array}
$$
$$
\begin{array}{cc}
=e^{i\alpha_{j}}\Bigg(\left(\underbrace{\tau_{0}\otimes \cdots \otimes\tau_{0}}_{\text{$n-k_{n}-1$ times}}\otimes  \tau_{\alpha_{n}} \otimes \underbrace{\iota\otimes \cdots\otimes \iota}_{\text{$k_{n}$ times}}\right)\otimes \left(\underbrace{\tau_{0}\otimes \cdots \otimes\tau_{0}}_{\text{$n-k_{n-1}-1$ times}}\otimes  \tau_{\alpha_{n-1}} \otimes \underbrace{\iota\otimes \cdots\otimes \iota}_{\text{$k_{n-1}$ times}}\right)\otimes  \\
\cdots \otimes\left(\underbrace{\tau_{0}\otimes \cdots \otimes\tau_{0}}_{\text{$n-k_{j+1}-1$ times}}\otimes  \tau_{\alpha_{j+1}} \otimes \underbrace{\iota\otimes \cdots\otimes \iota}_{\text{$k_{j+1}$ times}}\right)\Bigg)\circ (\pi_{c_{n,n}}\otimes \pi_{c_{n,n-1}}\otimes \cdots \otimes \pi_{c_{k_{j},j}})(t_{k,n+j})
\end{array}
$$
and thus we can determine the constants $\alpha_{j}$ inductively, starting at the largest index, in such way that~\eqref{creative} is equivalent to $\pi_{\sigma}\otimes \chi_{\lambda}$ where $\lambda_{n+i}$ coincides with $\beta_{n+i}$ if $n+1\leq \sigma(n+i)\leq 2n.$ This, together with Proposition~\ref{sunandmoon} and Proposition~\ref{makon} shows that every irreducible $*$-representation of $\mathrm{Pol}(\mathrm{Mat}_{n})_{q}$ can be written uniquely in the way~\eqref{creative} and hence can be represented uniquely as an admissible string $$[(k_{n},\alpha_{n}),(k_{n-1},\alpha_{n-1}),\dots , (k_{1},\alpha_{1})]$$ (subject to the conditions~\eqref{seq} and ~\eqref{seq2}) as was claimed in Theorem~\ref{mama}. 
\end{proof}

\subsection{Irreducible $*$-Representations Annihilating The Shilov Boundary}
In~\cite{ool_preprint}, the author, together with O. Bernstein and L. Turowska, determined the Shilov boundary ideal $\bar{J}_{n}$ for the closed sub-algebra (note, \textit{not} a $*$-algebra) $$A(\mathbb{D}_{n})\subseteq C_{F}(\mathbb{D}_{n})$$ generated by the images
$$
\begin{array}{ccc}
\pi_{F,n}(z_{k}^{j}),& k.j=1,2,\dots ,n
\end{array}
$$
to be the closure, under $\pi_{F,n},$ of the two-sided ideal $J_{n}\subseteq \mathrm{Pol}(\mathrm{Mat}_{n})_{q}$ generated by
$$
\begin{array}{ccc}
\sum_{j=1}^n q^{2n-\alpha-\beta}z_j^\alpha(z_j^\beta)^*-\delta_{\alpha\beta}, &\alpha,\beta=1,\ldots,n.
\end{array}
$$
From~\cite{vaksman_shilov1}, we have the result
$$
\mathrm{Pol}(\mathrm{Mat}_{n})_{q}/J_{n}\cong \mathbb{C}[U_{n}]_{q},
$$
where $\mathbb{C}[U_{n}]_{q}$ is the $*$-algebra of functions on the quantum $U_{n}$ (see~\cite{koelnik}). By Lemma $12$ in~\cite{ool_preprint}, every irreducible $*$-representation that annihilates the Shilov boundary ideal is equivalent to a $*$-representation $\pi$ of the form $\pi(z_{k}^{j})= e^{i\lambda_{j}}(-q)^{k-n}\pi_{\sigma}(t_{k,j})$, $k,j=1,\dots,n$ for an irreducible $*$-representation $\pi_{\sigma}:\mathbb{C}[SU_{n}]_{q}\to \ell^{2}(\mathbb{Z_{+}})^{\otimes \ell(\sigma)},$ where $\sigma\in S_{n},$ and some $\lambda_{j}\in [0,2\pi).$ If we let $\tilde{\sigma}$ be the image of $\sigma$ under the homomorphism $S_{n}\to S_{2n}$ defined as $s_{j}\mapsto s_{n+j}$ on the adjacent transpositions, then the $*$-representation $\pi_{\tilde{\sigma}}:\mathbb{C}[SU_{2n}]\to B(\ell^{2}(\mathbb{Z}_{+})^{\otimes \ell(\sigma)})$ (since $\ell(\tilde{\sigma})=\ell(\sigma)$) satisfies $$\pi_{\tilde{\sigma}}(t_{n+k,n+j})=\pi_{\sigma}(t_{k,j}).$$
By letting 
$$\varphi_{1}=\dots=\varphi_{n-1}=0$$
$$\varphi_{n}\equiv-\sum_{k=1}^{n}\lambda_{k} \pmod{2\pi}$$
$$\begin{array}{cc}\varphi_{n+j}=\lambda_{j}, & j=1,\dots,n\end{array}$$
it follows that the $*$-representation $(\pi_{\tilde{\sigma}}\otimes \chi_{\varphi})\circ \zeta$ coincides with $\pi.$ If we determine the admissible string $[(k_{n},\beta_{n}),\dots,(k_{1},\beta_{1})]$ corresponding to $\pi,$ then as $\tilde{\sigma}$ leaves $\{n+1,\dots , 2n\}$ invariant, it follows from Proposition~\ref{switch} that  
\begin{equation}\label{lower}
\begin{array}{cc}k_{i}<\max_{i<j\leq n}(k_{j}+i+1-j,i),& i=1,\dots, n\end{array}.
\end{equation}
We claim that from~\eqref{lower}, it follows that \begin{equation}\label{lessthan}\begin{array}{cc} k_{i}<i, & i=1,\dots ,n.\end{array}\end{equation} In fact when $i=n,$ we have~\eqref{lessthan} from~\eqref{lower}. Assume now that~\eqref{lessthan} holds for indices $j$ such that $n\geq j>i,$ then 
$$k_{i}<\max_{i<j\leq n}(k_{j}+i+1-j,i)=\max_{i<j\leq n}(k_{j}-(j-1)+i,i)=i $$ as $k_{j}-(j-1)+i\leq i$ holds for $n\geq j>i $ by induction. Thus~\eqref{lessthan} holds also for $i.$ It follows that the diagrams of irreducible $*$-representations that annihilates the Shilov boundary only contains white squares in the the strictly lower right sub-triangle
$$
\begin{tikzpicture}[thick,scale=0.5]
\filldraw[fill=black!50!white, draw=black]  (1,1) -- (1,7)--(7,7)--(7,1)--(1,1);
\draw[step=1.0,black,thick] (1,1) grid (7,7);
\node at (1.5,0.5) {$1$};
\node at (2.5,0.5) {$2$};
\node at (3.5,0.5) {$3$};
\node at (4.5,0.5) {$4$};
\node at (5.5,0.5) {$\dots$};
\node at (6.5,0.5) {$n$};
\node at (7.5,1.5) {$n$};
\node at (7.5,2.5) {$\vdots$};
\node at (7.5,3.5) {$4$};
\node at (7.5,4.5) {$3$};
\node at (7.5,5.5) {$2$};
\node at (7.5,6.5) {$1$};
\filldraw[fill=white, draw=black]   (1,1) -- (1,2)--(2,2)--(2,3)--(3,3)--(3,4)--(4,4)--(4,5)--(5,5)--(5,6)--(6,6)--(6,7)--(7,7)--(7,1)--(1,1);
\draw [thick,pattern=crosshatch, pattern color=black!50!white]   (1,1) -- (1,2)--(2,2)--(2,3)--(3,3)--(3,4)--(4,4)--(4,5)--(5,5)--(5,6)--(6,6)--(6,7)--(7,7)--(7,1)--(1,1);
\filldraw[fill=white, draw=black]  (2,1) -- (2,2)--(3,2)--(3,3)--(4,3)--(4,4)--(5,4)--(5,5)--(6,5)--(6,6)--(7,6)--(7,7)--(7,1)--(2,1);
\draw [thick,pattern=north west lines, pattern color=black!20!white]   (2,1) -- (2,2)--(3,2)--(3,3)--(4,3)--(4,4)--(5,4)--(5,5)--(6,5)--(6,6)--(7,6)--(7,7)--(7,1)--(2,1);
\draw [ultra thick]   (2,1) -- (2,2)--(3,2)--(3,3)--(4,3)--(4,4)--(5,4)--(5,5)--(6,5)--(6,6)--(7,6);
\draw[step=1.0,black,thick] (1,1) grid (7,7);
\end{tikzpicture}
$$

\end{document}